\documentclass[11pt,letter]{article}

\usepackage{amsmath,amssymb,amsthm}
\usepackage{mathrsfs}
\usepackage{graphicx,epsfig,psfrag}
\usepackage{xcolor}
\usepackage{enumerate}
\usepackage{subcaption}

\allowdisplaybreaks

\newtheorem{theorem}		{Theorem}[section]
\newtheorem{corollary}		[theorem]{Corollary}

\newtheorem{lemma}		[theorem]{Lemma}
\newtheorem{sideremark}		[theorem]{Remark}
\newtheorem{sidenote}		[theorem]{Note}
\newtheorem{sideeg}		[theorem]{Example}
\newtheorem{sideconj}		[theorem]{Conjecture}
\newtheorem{sideassumption}	{Assumption}

\newenvironment{remark}		{\begin{sideremark}\rm}{\end{sideremark}}

\DeclareMathOperator*{\argmin}{arg\,min}
\DeclareMathOperator*{\argmax}{arg\,max}

\DeclareMathOperator*{\dom}{dom}
\DeclareMathOperator*{\closure}{cl}
\DeclareMathOperator*{\lsc}{lsc}

\newcommand{\itsep}	{\itemsep=0mm}

\newcommand{\ba}				{\begin{array}}
\newcommand{\ea}				{\end{array}}
\newcommand{\ddtone}[2]			{{\frac{d {#1}}{d {#2}}}}
\newcommand{\er}[1]			{{(\ref{#1})}}
\newcommand{\nn}				{{\nonumber}}

\newcommand{\pdtone}[2]			{{\frac{\partial {#1}}{\partial {#2}}}}
\newcommand{\pdttwo}[2]			{{\frac{\partial^2 {#1}}{\partial {#2}^2}}}

\newcommand{\ts}[1]				{{\textstyle{#1}}}

\newcommand{\funspace}[1]		{{\mathscr{#1}}}
\newcommand{\ol}[1]				{{\overline{#1}}}
\newcommand{\wh}[1]			{{\widehat{#1}}}

\newcommand{\op}[1]			{{\mathcal{#1}}}
\newcommand{\Frechet}			{{Fr\`{e}chet}}

\newcommand{\N}				{{\mathbb{N}}}
\newcommand{\R}				{{\mathbb{R}}}
\newcommand{\Z}				{{\mathbb{Z}}}

\newcommand{\cA}				{{\funspace{A}}}
\newcommand{\cB}				{{\funspace{B}}}
\newcommand{\cL}				{{\funspace{L}}}
\newcommand{\cU}				{{\funspace{U}}}
\newcommand{\cV}				{{\funspace{V}}}

\newcommand{\bo}				{{\mathcal{L}}}
\newcommand{\demi}			{{\ts{\frac{1}{2}}}}
\newcommand{\eps}				{{\epsilon}}

\newcommand{\Ltwo}			{{{\cL}_2}}
\newcommand{\ggrad}			{{\nabla}}

\newcommand{\weakly}			{\rightharpoonup}

\newcommand{\rev}[1]			{{\color{red} {#1}}}

\title{\bf
Game representations for state constrained continuous time linear regulator problems
\thanks{Research partially supported by the Australian Research Council, AFOSR, and NSF.}
}

\author{Peter M. Dower$^{\dagger}$\qquad
William M. McEneaney${^\ddagger}$\qquad
Michael Cantoni
\thanks{\{Dower, Cantoni\} are with the Department of Electrical \& Electronic Engineering, 
		University of Melbourne,
		Victoria 3010, Australia.
		{\tt\small \{pdower,cantoni\}@unimelb.edu.au}
$^{\ddagger}$McEneaney is with the Department of Mechanical and Aerospace Engineering,
		University of California at San Diego,
		La Jolla, CA 92093, USA. 
		{\tt\small wmceneaney@ucsd.edu}}%
}

\begin{document}

\date{}

\maketitle
\thispagestyle{empty}



\begin{abstract}
A supremum-of-quadratics representation for convex barrier-type constraints is developed and applied within the context of a class of continuous time state constrained linear regulator problems. Using this representation, it is shown that a linear regulator problem subjected to such a convex barrier-type constraint can be equivalently formulated as an unconstrained two-player linear quadratic game. By demonstrating equivalence of the upper and lower values of this game, state feedback characterizations for the optimal policies of both players are developed. These characterizations are subsequently illustrated by example. 
\end{abstract}


\section{Introduction}
The study of unconstrained continuous time linear quadratic regulator (LQR) problems has provided the foundation of numerous advances in systems theory over many decades, including in optimal control \cite{AM:71}, and in the development of practical receding horizon / model predictive control strategies \cite{GPM:89}.
The value function that attends their formulation as an optimal control problem is guaranteed to be finite everywhere on sufficiently short time horizons, and is quadratic in the initial state, see for example \cite{AM:71,GL:95}. Indeed, it is well known that the Hessian of the value function is characterized in terms of the unique solution of a corresponding final value problem defined by a differential Riccati equation (DRE) subject to a terminal condition set by the Hessian of the terminal payoff. Standard tools exist for the efficient solution of DREs, and thus continuous time LQR problems.

The introduction of state constraints into LQR problems fundamentally impacts their solvability. Indeed, the nonlinearity inherent in such a constraint naturally destroys the quadratic structure that underpins the solvability of LQR problems via DREs. Instead, the value functions involved are inherently non-quadratic, and satisfy a more general non-stationary Hamilton-Jacobi-Bellman (HJB) partial differential equation (PDE). It is well-known that HJB PDEs are difficult to solve for nonlinear regulator problems, and that computational strategies that attend their solution suffer from a curse-of-dimensionality \cite{M:06}. These difficulties limit the  imposition of state constraints in continuous time LQR problems.

In this paper, implementation of a simple state constraint in an otherwise standard finite dimensional continuous time LQR problem is considered. Specifically, in addition to the standard linear open-loop dynamics and quadratic costs, a general convex barrier function is introduced with a view to constraining the state to a ball in the state space. By employing what is effectively a convex relaxation, an alternative solution strategy for such problems is explored via the study of a related unconstrained two-player game.
Fundamental to this exploration is the development of an exact {\em sup(remum)}-of-quadratics representation for the introduced convex barrier function. In this representation, the associated non-quadratic state penalty is expressed as the supremum of a parameterized family of quadratic penalties. The parameter identifying elements of this family is a single unbounded real variable that is related to the Hessian of the associated quadratic penalty. Its manipulation ultimately allows a quadratic state penalty to be selected in a state dependent fashion, in lieu of the general convex barrier penalty, as the underlying state trajectory evolves in time. Where necessary, an infinite state penalty corresponding to activation of the state constraint can be levied by allowing this quadratic penalty parameter to tend to infinity.
An approximate sup-of-quadratics representation follows by limiting the parameter involved to a bounded interval. This approximation exactly represents the barrier function for states inside a ball, and approximates it by a single quadratic function for states outside that ball. It is parameterized by the upper bound of the interval involved, and converges to the exact sup-of-quadratics representation as this upper interval bound tends to infinity.
Examples of the type of sup-of-quadratics representation obtained are illustrated in Figure \ref{fig:sup-of-quad}. 

\begin{figure}[h!]
\vspace{0mm}
\begin{center}
\begin{subfigure}[c]{0.48\textwidth}
\psfrag{sigma}{$\sqrt{\rho}$}
\psfrag{Convex barrier}{\small\hspace{-17mm}Barrier $\Phi(\rho)$ (-\,-), quadratics (-\,\!\!-)}
\psfrag{LogbarrierXXXXXXXX}{barrier}
\psfrag{Quadratics}{quadratics}
\includegraphics[width=\textwidth,height=65mm]{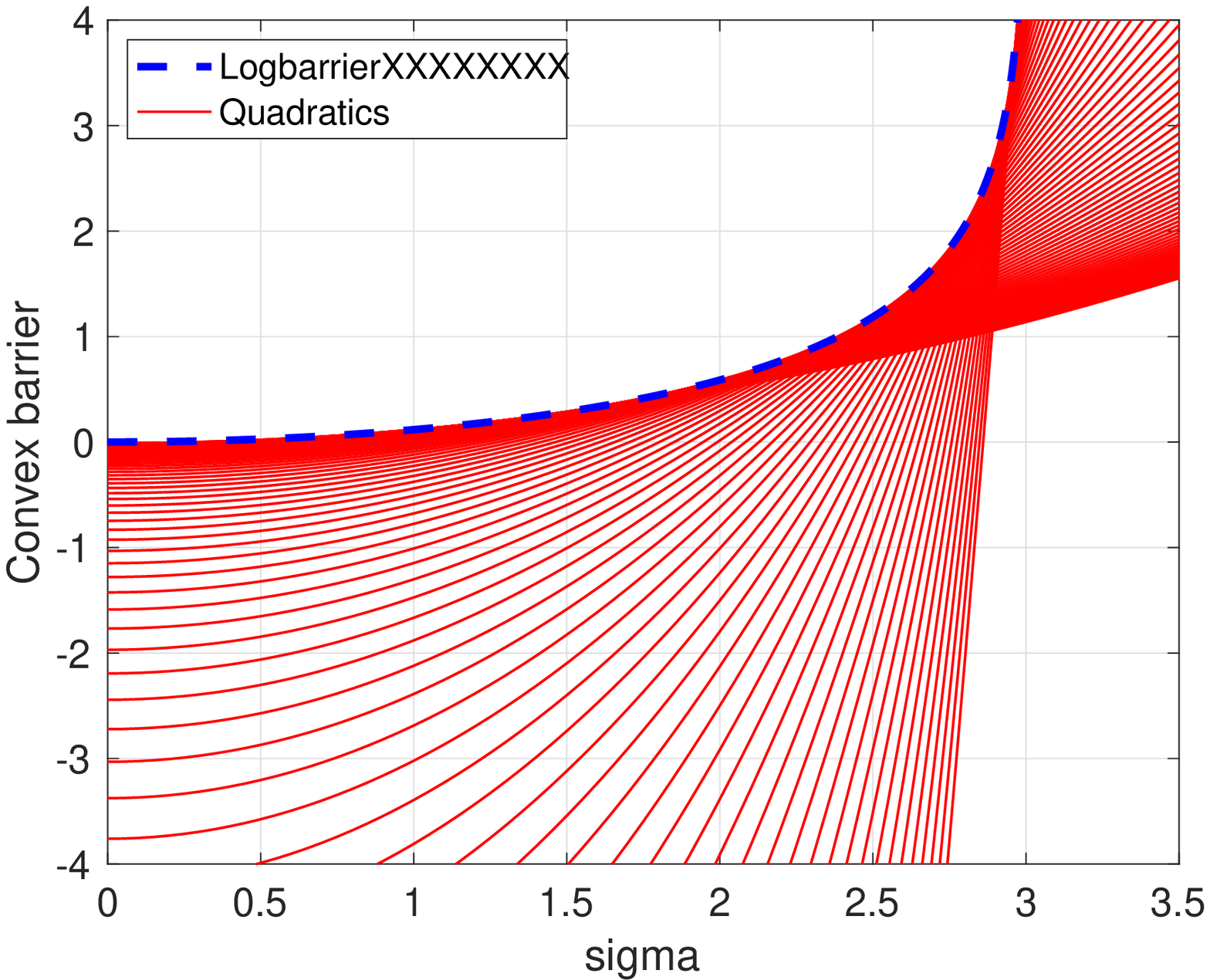}
\vspace{-5mm}
\caption{Convex barrier, see Theorem \ref{thm:Phi-sup-of-quadratics} and Section \ref{sec:examples}.}
\label{fig:sup-of-quad-no-dip}
\end{subfigure}
\quad
\begin{subfigure}[c]{0.48\textwidth}
\psfrag{sigma}{$\sqrt{\rho}$}
\psfrag{Convex barrier}{\small\hspace{-17mm}Barrier $\Phi(\rho)$ (-\,-), quadratics (-\,\!\!-)}
\psfrag{LogbarrierXXXXXXXX}{barrier}
\psfrag{Quadratics}{quadratics}
\includegraphics[width=\textwidth,height=65mm]{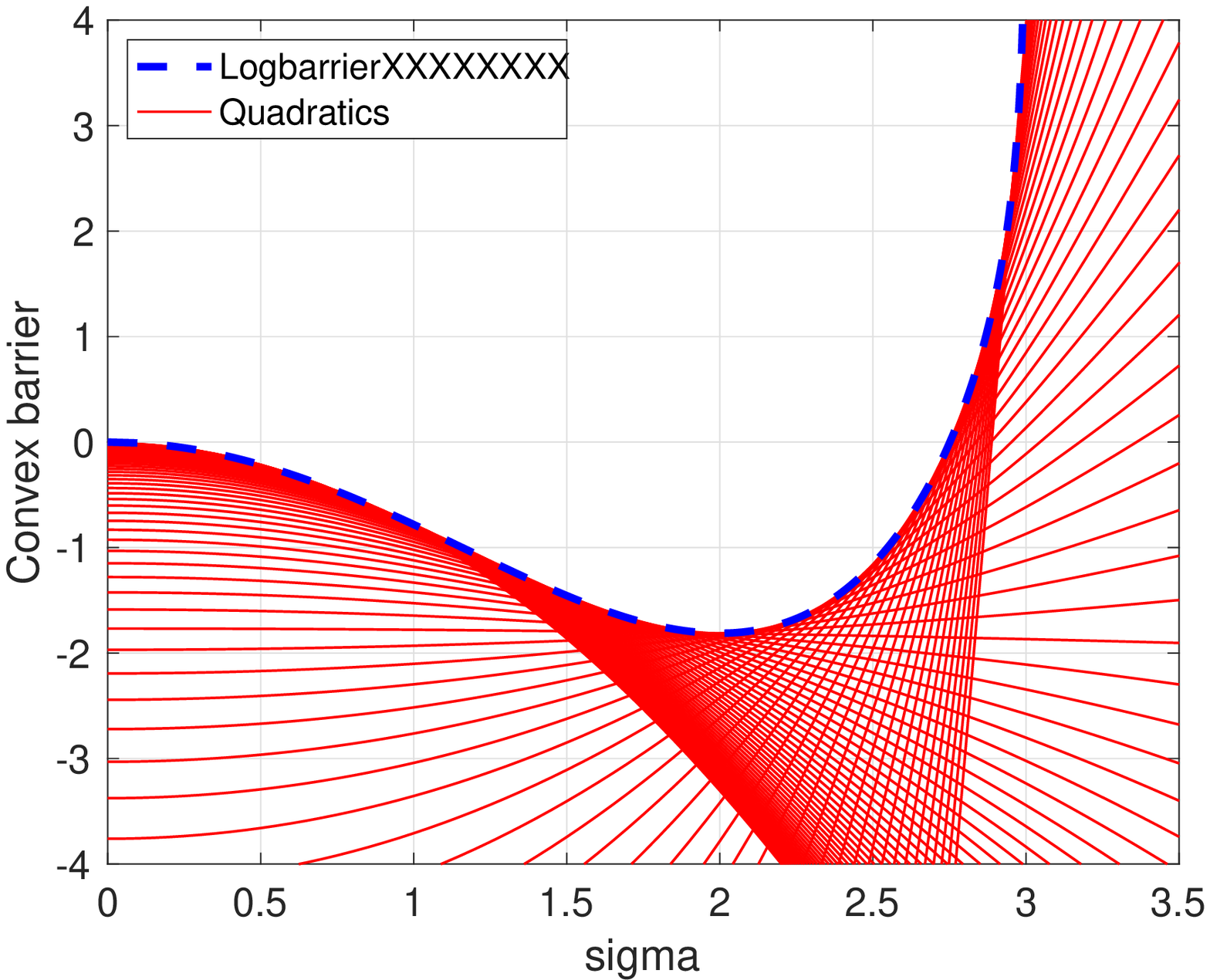}
\vspace{-5mm}
\caption{Semiconvex barrier, see Remark \ref{rem:Phi-semiconvex}.}
\label{fig:sup-of-quad-dip}
\end{subfigure}
\vspace{-2mm}
\caption{Sup-of-quadratics representation for $\Phi$ of \er{eq:barrier}, \er{eq:example}.}
\label{fig:sup-of-quad}
\end{center}
\vspace{-6mm}
\end{figure}

Invocation of the exact or approximate sup-of-quadratics representation for the convex barrier function in the state constrained regulator problem of interest yields respectively the same regulator problem, or a convergent approximation to it. Convergence is demonstrated both in terms of the value functions involved, and the behaviour of their optimal state trajectories. It is shown explicitly that the exact representation yields an implementation of the state constraint, with trajectories always confined to the closed ball of interest, and to its interior almost always (a.e. in time). The approximate problem is shown to yield trajectories that converge to this behaviour. 
Measurable selection subsequently
leads to the consideration of corresponding two player games. In these games, the minimizing player corresponds to the usual control, while the maximizing player is an adversary that negotiates an appropriate state penalty, given the current state of the trajectory relative to the state constraint. It is shown that the upper values of the exact and approximate games are equivalent respectively to the exact and approximate regulator problems indicated above, in the same quantitative manner regarding value and constraint satisfaction. The upper and lower values are subsequently shown to be equivalent, see also \cite{MD1:15}, which is useful in computation.
Consideration of the approximating game lower value yields corresponding state feedback characterizations for the optimal policies of both players. These policies are shown to explicitly depend on the solution of the state dynamics driven by the approximate optimal control, and the solution of an attendant family of DREs. It is demonstrated that solutions of these DREs always exist, and that they encode the actions of the maximizing (state penalty negotiating) player. Convergence of these policies to those of the exact game is guaranteed, so that the optimal control in particular converges to that of the original state constrained regulator problem of interest.
An illustrative example is included that evaluates an approximate optimal strategy for both players, given a specific initial plant state and terminal cost. The effect of the state constraint on the optimal control and trajectory is also identified. 

A selected collection of immediately relevant prior works that invoke duality in optimal control, and the implementation of constraints in otherwise linear quadratic regulator problems, include \cite{W:95,DG:06,GS:07,RG:08,BKM:14,FE:17}, and (from the authors) \cite{DMC1:16, DMC2:16, DC1:17}. Specifically, \cite{W:95} develops an approximation scheme for general convex costs, and studies consistency of this approximation, while \cite{DG:06} considers continuous time constrained control in a model predictive control setting, subject to an interiority condition on the feedback policy. One of many related investigations exploiting barrier functions in the implementation of constraints is detailed in \cite{FE:17}, via a discrete time setting. Duality and saddle point properties are explored in a more general setting in \cite{GS:07,BKM:14}, albeit in the restricted case of control constraints.
The tools of convex analysis are employed in the general treatment of a closely related class of continuous time problems in \cite{RG:08} that addresses both control and state constraints. Motivated by \cite{MD1:15}, a key contribution of the current work relative to \cite{RG:08} concerns the sup-of-quadratics representation developed for the extended real barrier functions involved, and their invocation in studying the optimal control problem via an unconstrained game.
Preliminary efforts \cite{DMC1:16,DMC2:16,DC1:17} (by the authors) document the genesis of this contribution, with log-barrier functions considered initially, and more general convex barrier functions considered subsequently. The more recent work \cite{DC1:17} has demonstrated how this approach can be generalized to more general convex state constraints, in the company of time-varying dynamics. For brevity, those details are not included here.

\if{false}

\cite{W:95} primal / dual, piecewise constant controls, approximation scheme, consistency.

\cite{DG:06} cts time NMPC, barrier functions for trajectory and terminal constraints, requires interiority conditions the control / barrier functions.

\cite{GS:07} control constraints, otherwise very related approach.

\cite{RG:08} state and control constraints, no sup-of-quadratics representation.

\cite{BKM:14} control constraints

\cite{FE:15} discrete time, MPC

\fi

In terms of organization, the state constrained linear regulator problem of interest is posed in Section \ref{sec:optimal}, along with the class of convex barrier functions involved. This is followed in Section \ref{sec:sup-of-quadratics} by development of the exact and approximate sup-of-quadratics representations for these convex barrier functions, and the introduction of the approximate regulator problem. Existence and uniqueness of optimal trajectories for the exact and approximate regulator problems are considered in Section \ref{sec:existence}, along with their behaviour relative to the state constraint of interest. Exact and approximate two player games are formulated in Section \ref{sec:game}, and their respective equivalences with the exact and approximate regulator problems is demonstrated. A further equivalence of the upper and lower values is demonstrated in each case. This in turn motivates characterization of the optimal policies involved via solution of a two-point boundary value problem defined in terms of a DRE. This characterization is subsequently illustrated by example of Section \ref{sec:examples}. The paper concludes with some minor summarizing remarks in Section \ref{sec:conc}. An appendix is included for technicalities that might otherwise interrupt the developments described.

Throughout, $\R$, $\N$, $\Z$ denote the reals, natural numbers, and integers, while $\R_{\ge 0}$, $\R_{>0}$, and $\ol{\R}$ denote the non-negative, positive, and extended reals respectively, with the latter defined by $\ol{\R}\doteq\ol{\R}^-\cup\ol{\R}^+$, $\ol{\R}^\pm\doteq\R\cup\{\pm\infty\}$. For convenience, $\R_{\ge a} \doteq [a,\infty)$ and $\R_{>a}\doteq(a,\infty)$ for any $a\in\R$. An $n$-dimensional Euclidean space is denoted by $\R^n$. The space of matrices mapping $\R^m$ to $\R^n$ is denoted by $\R^{n\times m}$. The subset of positive semidefinite symmetric matrices in $\R^{n\times n}$ is denoted by $\Sigma^n$. The Euclidean and induced matrix norms are denoted by $|\cdot|$ and $\|\cdot\|$ respectively. Otherwise, the norm on a Banach space $\cU$ is denoted by $\|\cdot\|_{\cU}$, or simply $\|\cdot\|$ if the space is contextually apparent. Open and closed balls of radius $r\in\R_{\ge 0}$ in $\cU$ are denoted respectively by $\cB_{\cU}(0;r)$ and $\cB_{\cU}[0;r]$ respectively. Weak convergence of a sequence $\{u_k\}_{k\in\N}\subset\cU$ to some $\bar u\in\cU$ is denoted by $u_k \weakly \bar u$ (as $k\rightarrow\infty$). The product space $\cU\times \cdots \times \cU$ of $k\in\N$ instances of $\cU$ is denoted by $\cU^k$. The space of bounded linear operators between Banach spaces $\cU$ and $\cV$ is denoted by $\bo(\cU;\cV)$.
The respective spaces of continuous and $k$-times continuously differentiable functions mapping $\cU$ to $\cV$ are denoted by $C(\cU;\cV)$ and $C^{k}(\cU;\cV)$ for $k\in\N\cup\{\infty\}$. Differentiability at a closed left or right end-point of an interval is interpreted throughout to mean right- or left-differentiability respectively. The space of (Lebesgue) square integrable mappings from $[0,t]\subset\R_{\ge 0}$ to $\cU$ is denoted by $\Ltwo([0,t];\cU)$. Unless otherwise specified, $C([0,t];\cU)$ is equipped with the sup norm, i.e. $\|F\|\doteq\|F\|_{C([0,t];\cU)} \doteq \sup_{s\in[0,t]} \|F(s)\|_\cU$, $F\in C([0,t];\cU)$.
The following concern a function $f:\cU\rightarrow\ol{\R}$ on Banach space $\cU$:

$f:\cU\rightarrow\ol{\R}$ has (possibly empty) domain $\dom f\doteq \{ u\in\cU \, | \, f(u)< \infty \}$.

$f:\cU\rightarrow\ol{\R}$ is proper if $\dom f \neq \emptyset$ and $f$ is finite on $\dom f$, i.e. $f(u)>-\infty$ for all $u\in\dom f$.

$f:\cU\rightarrow\ol{\R}$ is lower semicontinuous if $\lsc f(u) \doteq \liminf_{\tilde u\rightarrow u} f(\tilde u) \ge f(u)$ for all $u\in\cU$.

$f:\cU\rightarrow\ol{\R}$ is (lower) closed if $f = \closure f$, see  \cite[(3.8), p.15]{R:74}, where
$$
	\closure f(u)
	\doteq \left\{ \ba{cl}
		\lsc f(u),
		& \forall \ u\in\cU \text{ if }
		\lsc f(u) >-\infty \ \forall \ u\in\cU,
		\\
		-\infty, & \forall \ u\in\cU \text{ if } \exists \, \bar u\in\cU \text{ s.t. } \lsc f(\bar u) = -\infty.
	\ea
	\right.
$$

$f:\cU\rightarrow\ol{\R}$ is coercive if $\lim_{\|u\|\rightarrow\infty} f(u) / \|u\| = \infty$. 

$f:\cU\rightarrow\ol{\R}$ is (strictly) convex if $f:\dom f\rightarrow\ol{\R}^-$ is (strictly) convex, i.e.
$f((1-\lambda) u + \lambda\tilde u) \le (1-\lambda) f(u) + \lambda f(\tilde u)$ (strictly) for all $\lambda\in(0,1)$, $u,\tilde u\in\dom f$. The map $u\mapsto\infty$ is strictly convex.

\if{false}

$f:\cU\rightarrow\ol{\R}$ is (strictly) mid-point convex if $f(u + \tilde u) - 2\, f(u) + f(u - \tilde u) \ge 0$ (strictly) for all $u,\tilde u\in\cU$.

\begin{lemma}
\label{lem:mid-point}
Suppose $f:\cU\rightarrow\ol{\R}$ is proper, and (strictly) mid-point convex and continuous on $\dom f\subset\cU$. Then, $f$ is (strictly) convex on $\cU$.
\end{lemma}

The proof of Lemma \ref{lem:mid-point} is postponed to Appendix \ref{app:chi-and-dom}.

\fi


\section{State constrained linear regulator problem}
\label{sec:optimal}

Interest is restricted to optimal control problems defined on a finite time horizon $t\in\R_{\ge 0}$, with respect to linear dynamics and a convex barrier state constraint. The value function $\ol{W}_t:\R^n\rightarrow\ol{\R}^+$ involved is defined by
\begin{align}
	\ol{W}_t(x)
	& \doteq \inf_{u\in\cU[0,t]} \bar{J}_t(x,u)\,,
	\label{eq:value-W}
\end{align}
for all $x\in\R^n$, in which $\cU[0,t] \doteq \Ltwo([0,t];\R^m)$ is the space of open loop controls, and $\bar J_t$ is a cost function defined with respect to the integrated running costs $\bar I_t$ and $I_t^\kappa$, $\kappa\in\R_{>0}$, and a terminal cost $\Psi$. In particular, $\bar J_t, \bar I_t:\R^n\times\cU[0,t]\rightarrow\ol{\R}^+$, $I_t^\kappa:\cU[0,t]\rightarrow\R_{\ge 0}$, and $\Psi:\R^n\rightarrow\R_{\ge 0}$, are defined by
\begin{gather}
	\bar J_t(x,u)
	\doteq \bar I_t(x,u) + I_t^\kappa(u) + \Psi(x_t)\,,
	\label{eq:cost-J-bar}
	\\
	\bar I_t(x,u)
	\doteq \int_0^t \ts{\frac{K}{2}}\, |\xi_s|^2 + \demi\, \Phi(|\xi_s|^2)\, ds\,,
	\label{eq:cost-I-bar}
	\quad
	I_t(u) \doteq I_t^\kappa(u)
	\doteq \ts{\frac{\kappa}{2}}\, \|u\|_{\cU[0,t]}^2\,,
	\\
	\Psi(x)
	\doteq \demi\, \langle x - z,\, P_t\, (x - z) \rangle\,,
	\label{eq:cost-Psi}
\end{gather}
for all $x\in\R^n$, $u\in\cU[0,t]$, in which $K\in\R$ and $P_t\in\Sigma^n$ are a priori fixed, and $\Phi$ is an extended real valued barrier function to be specified below. The map $s\mapsto\xi_s\in\R^n$, $s\in[0,t]$, describes the unique trajectory of a linear dynamical system corresponding to an initial state $x\in\R^n$ and input $u\in\cU[0,t]$, given explicitly via a map $\chi:\R^n\times\cU[0,t]\rightarrow\R^n$, where
\begin{align}
	\xi_s 
	& = [\chi(x,u)]_s \doteq e^{A\, s}\, x + \int_0^s e^{A\, (s-\sigma)}\, B\, u_\sigma\, d\sigma\,,
	\label{eq:dynamics}
\end{align}
for all $s\in[0,t]$, given $A\in\R^{n\times n}$, $B\in\R^{n\times m}$, $B\ne 0$. The barrier function $\Phi:\R\rightarrow\ol{\R}^+$ is defined by
\begin{align}
	\Phi(\rho)
	& \doteq \left\{ \ba{cl}
		\phi(\rho)\,,
		& \rho\in[0,b^2),
		\\
		+\infty\,,
		& \rho\not\in[0,b^2), 
	\ea \right.
	\label{eq:barrier}
\end{align}
for fixed $b\in\R_{>0}$, in which $\phi:[0,b^2)\rightarrow\R$ satisfies the following properties:
\begin{gather}
	\begin{aligned}
	&\, \left\{ \ba{ll}
	\text{\em (i)} & \text{$\phi$ is twice continuously differentiable, with $\phi''$ strictly positive;}
	\\
	\text{\em (ii)} \ & \text{$\lim_{\rho\uparrow b^2} \phi(\rho) = \infty$, and $\phi'(0)\ge -K$;}
	\\
	\text{\em (iii)} & \text{$\phi$ is strictly convex;}
	\\
	\text{\em (iv)} & \text{$\phi'$ is strictly increasing,
	and $\phi':[0,b^2)\rightarrow[\phi'(0),\infty)$; and}
	\\
	\text{\em (v)} & \text{$(\phi')^{-1}$ exists, is strictly increasing, and
	$(\phi')^{-1}:[\phi'(0),\infty)\rightarrow[0,b^2)$.}
	\ea \right.
	\end{aligned}
	\label{eq:phi-properties-1-2}
\end{gather}
Note in particular that {\em (iii)--(v)} follow as a consequence of {\em (i)--(ii)}, see for example \cite[Theorem 2.13, p.46]{RW:97}. As a consequence, $\phi$ has a well-defined convex dual $a:\R_{\ge\phi'(0)}\rightarrow\R_{\ge -\phi(0)}$ given by
\begin{align}
	&
	a(\beta)
	\doteq \beta\, (\phi')^{-1}(\beta) - \phi\circ(\phi')^{-1}(\beta)\,,
	\label{eq:exact-a}
\end{align}
for all $\beta\in\R_{\ge\phi'(0)}$, that satisfies a variety of properties, including invertibility, etc, see Lemma \ref{lem:a-properties} in Appendix \ref{app:properties}. It defines a useful change of coordinates in the sup-of-quadratics representation that is developed for barrier $\Phi$ in Section \ref{sec:sup-of-quadratics}. Two preliminary lemmas concerning \er{eq:value-W}, \er{eq:dynamics} are included prior to commencing this development. Their proofs are standard and are omitted for brevity. 

\begin{lemma}
\label{lem:chi}
Given any $t\in\R_{\ge 0}$, $x\in\R^n$, $\ol{U}\in\R_{\ge 0}$, $\{u_k\}_{k\in\N}\subset\cB_{\cU[0,t]}[0;\ol{U}]$, with $\xi_k\doteq \chi(x,u_k)$ defined via \er{eq:dynamics} for all $k\in\N$, the following properties hold:
\begin{enumerate}[\it (i)] \itsep
\item $\xi \doteq\chi(x,u):[0,t]\rightarrow\R^n$ is uniformly continuous, given any fixed $u\in\cU[0,t]$;
\item $\chi(x,\cdot)\in C^\infty(\cU[0,t];C([0,t];\R^n))$, with $k^{th}$-order {\Frechet} derivatives given for $k\in\N$ by
\begin{gather}
	\begin{gathered}
	D_u^k\, \chi(x,\cdot):\cU[0,t]\rightarrow\bo((\cU[0,t])^k;C([0,t];\R^n)),
	\\
	[D_u\, \chi(x,u)\, h]_s
	= [\op{A}\, h]_s \doteq \int_0^s e^{A\, (s-\sigma)}\, B\, h_\sigma\, d\sigma\,,
	\quad
	D_u^k\, \chi(x,u) = 0, \ k\in\N_{\ge 2}\,,
	\end{gathered}
	\label{eq:derivatives}
\end{gather}
for all $u,h\in\cU[0,t]$, $s\in[0,t]$, in which $0\in\bo((\cU[0,t])^k;C([0,t];\R^n))$ denotes the zero operator. 
Moreover, $\op{A}\in\bo(\cU[0,t];C([0,t];\R^n))$, with $\|\op{A}\, h\|_{C([0,t];\R^n)} \le \sup_{s\in[0,t]} \| e^{A\, s}\, B \| \, \sqrt{t}\, \|h\|_{\cU[0,t]}$ for all $h\in\cU[0,t]$; and

\item $\{ \xi_k \}_{k\in\N}\subset C([0,t];\R^n)$ is uniformly equicontinuous and uniformly bounded. Furthermore, there exists a $\bar u\in\cU[0,t]$ and subsequences $\{v_k\}_{k\in\N}\subset\{ u_k \}_{k\in\N}$ and $\{y_k\}_{k\in\N} \subset \{ \xi_k \}_{k\in\N}$ such that $v_k\weakly \bar u$ weakly and $y_k\rightarrow\bar\xi\doteq\chi(x,\bar u)$ uniformly, in which $y_k = \chi(x,v_k)$ for all $k\in\N$.
\end{enumerate}
\end{lemma}

\begin{lemma}
\label{lem:dom-value-W-bar}
$0\in\dom\ol{W}_t$ for all $t\in\R_{\ge 0}$.
\end{lemma}


\section{Barrier representations and an approximate regulator problem} 
\label{sec:sup-of-quadratics}
Exact and approximate sup-of-quadratics representations for closed convex barrier functions of the form of $\Phi$ of \er{eq:barrier} can be established via convex duality \cite{RW:97,R:74}. These representations are fundamental to the development of a convergent approximation for the state constrained regulator problem \er{eq:value-W}, and its subsequent representation via unconstrained linear quadratic games. The development of these representations and the approximate regulator problem follow below.


\newcommand{\Ahat}		{{\Theta}}

\subsection{Exact sup-of-quadratics representation for convex barriers}

\begin{lemma}
\label{lem:exact-convex-duality}
The barrier function $\Phi:\R\rightarrow\ol{\R}^+$ of \er{eq:barrier} is closed and convex, and there exists a closed and convex function $\Ahat:\R\rightarrow\R$ such that 
\begin{align}
	\Phi(\rho)
	& = \sup_{\beta\in\R} \left\{ \beta\, \rho - \Ahat(\beta) \right\},
	\qquad
	\Ahat(\beta)
	= \sup_{\rho\in\R} \left\{ \beta\, \rho - \Phi(\rho) \right\} 
	= \left\{ \ba{rl}
		-\phi(0), &	\beta\in\R_{<\phi'(0)},
		\\
		a(\beta),
		& \beta\in\R_{\ge \phi'(0)},
	\ea \right.
	\label{eq:exact-Phi-and-A}
\end{align}
for all $\rho,\beta\in\R$, with $a$ as per \er{eq:exact-a}. Furthermore, the optimizers $\hat\beta^*:\R\rightarrow\ol{\R}$ and $\hat\rho^*:\R\rightarrow\R$ in \er{eq:exact-Phi-and-A}, defined by $\hat\beta^*(\rho) \doteq \argmax_{\beta\in\R} \{ \beta\, \rho - \Ahat(\beta) \}$ and $\hat\rho^*(\beta) \doteq \argmax_{\rho\in\R} \{ \beta\, \rho - \Phi(\rho) \}$, are given by
\begin{align}
	\hat\beta^*(\rho)
	& = \left\{ \ba{rl}
		-\infty,
		& \rho\in\R_{<0},
		\\
		\phi'(\rho),
		& \rho\in[0,b^2)\,,
		\\
		+\infty,
		& \rho\in\R_{\ge b^2},
	\ea \right.
	\quad
	\hat\rho^*(\beta)
	= \left\{ \ba{rl}
		0, &	\beta\in\R_{<\phi'(0)},
		\\
		(\phi')^{-1}(\beta),
		& \beta\in\R_{\ge\phi'(0)},
	\ea \right.
	\label{eq:exact-beta-and-rho-star}
\end{align}
for all $\beta,\rho\in\R$.
\end{lemma}
\begin{proof}
See Appendix \ref{app:convex-duality-proofs}.
\end{proof}
This lemma, and a change of coordinates via \er{eq:exact-a}, yields the sup-of-quadratics representation.
\begin{theorem}
\label{thm:Phi-sup-of-quadratics}
The barrier function $\Phi(|\cdot|^2):\R^n\rightarrow\ol{\R}^+$ appearing in \er{eq:value-W} via \er{eq:cost-J-bar}, \er{eq:cost-I-bar}, and defined by \er{eq:barrier}, has the exact sup-of-quadratics representation
\begin{align}
	\Phi(|x|^2)
	& = \sup_{\alpha\ge -\phi(0) } \{ a^{-1}(\alpha)\, |x|^2 - \alpha \}
	\label{eq:Phi-sup-of-quadratics}
\end{align}
for all $x\in\R^n$, in which $a^{-1}$ is defined via \er{eq:exact-a}. Furthermore, the optimizer $\hat\alpha^*(|\cdot|^2):\R^n\rightarrow\R_{\ge -\phi(0)}^+$ in \er{eq:Phi-sup-of-quadratics} is defined via $\phi'$, $a$ of \er{eq:phi-properties-1-2}, \er{eq:exact-a} by
\begin{align}
	\hat\alpha^*(|x|^2)
	& = \argmax_{\alpha\ge -\phi(0) } \{ a^{-1}(\alpha)\, |x|^2 - \alpha \}
	= \left\{ \ba{rl}
		a\circ \phi'(|x|^2)\,,	&	|x|<b\,,
		\\
		\infty\,,		&	|x|\ge b\,,
	\ea \right.			
	\label{eq:alpha-star}
\end{align}
for all $x\in\R^n$.
\end{theorem}
\begin{proof} 
Fix arbitrary $x\in\R^n$. Recall by Lemma \ref{lem:exact-convex-duality} that $\Phi(|x|^2)$ has the representation \er{eq:exact-Phi-and-A}, with the optimizer $\hat\beta^*(|x|^2)$ that achieves the supremum over $\beta\in\R$ there defined by \er{eq:exact-beta-and-rho-star}. As $|x|^2\in\R_{\ge 0}$, \er{eq:exact-beta-and-rho-star} implies by inspection that $\hat\beta^*(|x|^2) \ge \phi'(|x|^2)$. Furthermore, by property {\em (iv)} of \er{eq:phi-properties-1-2}, $\phi'(|x|^2)\ge\phi'(0)$. Hence, $\hat\beta^*(|x|^2) \ge \phi'(0)$. Meanwhile, $a$ of \er{eq:exact-a} defines the change of variable $\alpha = a(\beta)$ for all $\beta\in[\phi'(0),\infty)$. Note in particular that $\phi'(0) = a^{-1}(-\phi(0))$, as $a\circ\phi'(0) = - \phi(0)$ and $a$ is invertible, see \er{eq:exact-a} and Lemma \ref{lem:a-properties}.
Hence, $\Phi(|x|^2)$ simplifies from the left-hand equation in \er{eq:exact-Phi-and-A}, via \er{eq:exact-a}, to
$
	\Phi(|x|^2)
	= \sup_{\beta\ge\phi'(0)} \{ \beta\, |x|^2 - \Ahat(\beta) \}
	= \sup_{\beta\ge a^{-1}(-\phi(0))} \{ \beta\, |x|^2 - a(\beta) \}
$,
which yields \er{eq:Phi-sup-of-quadratics}. The same change of variable applied to $\hat\beta^*(|x|^2)$ via \er{eq:exact-beta-and-rho-star} similarly yields \er{eq:alpha-star}.
\end{proof}

\begin{remark}
\label{rem:Phi-semiconvex}
While the barrier map $\rho\mapsto\Phi(\rho):\R\rightarrow\ol{\R}^+$ of \er{eq:barrier} 
is guaranteed to be convex by Lemma \ref{lem:exact-convex-duality}, the corresponding state space map $x\mapsto\Phi(|x|^2):\R^n\rightarrow\ol{\R}^+$ need not be convex. However, Theorem \ref{thm:Phi-sup-of-quadratics} implies that $x\mapsto \Phi(|x|^2):\R^n\rightarrow\ol{\R}^+$ is {\em uniformly semiconvex} \cite{M:06,FM:00}. In particular, choosing any $\eta \ge -2\, a^{-1}(-\phi(0))$,  \er{eq:Phi-sup-of-quadratics} yields
\begin{align}
	\Phi(|x|^2) + \ts{\frac{\eta}{2}}\, |x|^2
	& = \sup_{\alpha\ge -\phi(0)} \left\{ [ a^{-1}(\alpha) + \ts{\frac{\eta}{2}} ] \, |x|^2 - \alpha \right\},
	\label{eq:semi-convex}
\end{align}
for all $x\in\R^n$, in which $a^{-1}(\alpha) + \ts{\frac{\eta}{2}} \ge a^{-1}(\alpha) - a^{-1}(-\phi(0)) \ge 0$ for all $\alpha\ge -\phi(0)$, as $a^{-1}$ is strictly increasing by Lemma \ref{lem:a-properties}. The right-hand side of \er{eq:semi-convex} is thus a supremum of convex functions, which is therefore also convex, see \cite[p.7]{R:74}. That is, \er{eq:semi-convex} implies that there exists an $\eta\in\R$ such that $\Phi(|\cdot|^2) + \ts{\frac{\eta}{2}}\, |\cdot|^2$ is convex, so that $\Phi(|\cdot|^2)$ is uniformly semiconvex by definition, see  \cite{M:06,FM:00}.
\end{remark}


\subsection{Approximate sup-of-quadratics representation for convex barriers}
An approximation of the sup-of-quadratics representation of Theorem \ref{thm:Phi-sup-of-quadratics} can be obtained by restricting the interval over which the supremum is evaluated in the left-hand equation in \er{eq:exact-Phi-and-A}. To this end, define $\Phi^M:\R\rightarrow\ol{\R}^+$ and $\hat\rho:[-\phi(0),\infty)\rightarrow[0,b^2)$ by
\begin{align}
	\Phi^M(\rho)
	& \doteq \sup_{\beta\le a^{-1}(M)} \{ \beta\, \rho - \Ahat(\beta) \}\,,
	\qquad
	\hat\rho(M) 
	\doteq (\phi')^{-1}\circ a^{-1}(M),
	\label{eq:Phi-M-from-A-and-rho-M}
\end{align}
for all $M\in\R_{\ge -\phi(0)}$, $\rho\in\R$, with $\phi'$, $a$, $\Ahat$ as per \er{eq:phi-properties-1-2}, \er{eq:exact-a}, \er{eq:exact-Phi-and-A}, with the range of $\Phi^M$ to be verified.

\begin{lemma}
\label{lem:approx-convex-duality}
The following properties hold:
\begin{enumerate}[\it (i)] \itsep
\item $\Phi^M:\R\rightarrow\ol{\R}^+$, $M\in\R_{\ge -\phi(0)}$, of \er{eq:Phi-M-from-A-and-rho-M} satisfies
\begin{align}
	& \hspace{-4mm}
	\Phi^M(\rho)
	= \left\{ \ba{rl}
		\infty,	&	\rho\in\R_{<0},
		\\
		\phi(\rho),	&	\rho\in[0,\hat\rho(M)],
		\\
		a^{-1}(M)\, \rho - M,	&	\rho\in\R_{>\hat\rho(M)},
	\ea \right.
	\label{eq:Phi-M-explicit}
\end{align}
for all $\rho\in\R$, in which the corresponding maximizer $\beta = \beta^{M*}:\R\rightarrow\ol{\R}^-$ is given by 
\begin{align}
	\hat\beta^{M*}(\rho)
	& \doteq \left\{ \ba{rl}
		-\infty,		&	\rho\in\R_{<0},
		\\
		\phi'(\rho),		&	\rho\in[0,\hat\rho(M)],
		\\
		a^{-1}(M),		&	\rho\in\R_{>\hat\rho(M)}.
	\ea \right.
	\label{eq:beta-M-star-explicit}
\end{align}

\item $\Phi^M\in C(\R_{\ge 0};\R)\cap C^1(\R_{>0};\R)$, $M\in\R_{\ge -\phi(0)}$, and it is closed and strictly convex on $\R$;

\item $\Phi^M$ is pointwise non-decreasing in $M\in\R_{\ge -\phi(0)}$, and satisfies
$
	\Phi(\rho)
	= \sup_{M \ge -\phi(0)} \Phi^M(\rho) = \lim_{M\rightarrow\infty} \Phi^M(\rho),
$
for all $\rho\in\R$, with $\Phi$ as per \er{eq:barrier};

\item There exists an $M_1\in\R_{\ge-\phi(0)}$ and $c\in\R$ such that $\inf_{M\ge M_1} \inf_{\rho\in\R} \Phi^M(\rho)> c$.
\end{enumerate}
\end{lemma}
\begin{proof}
See Appendix \ref{app:convex-duality-proofs}.
\end{proof}

As per the exact case of Theorem \ref{thm:Phi-sup-of-quadratics}, application of this lemma along with a change of coordinates defined by \er{eq:exact-a} admits the required approximate sup-of-quadratics representation.

\begin{theorem}
\label{thm:Phi-M-sup-of-quadratics}
Given $b\in\R_{>0}$, the following holds:
\begin{enumerate}[\it (i)] \itsep
\item Given $M\in\R_{\ge -\phi(0)}$, the convex approximation $\Phi^M$ of the convex barrier function $\Phi$ of \er{eq:barrier}, represented in \er{eq:Phi-M-from-A-and-rho-M}, \er{eq:Phi-M-explicit}, has the sup-of-quadratics representation 
\begin{align}
	\Phi^M(|x|^2)
	& = \sup_{\alpha\in[-\phi(0),M]} \{ a^{-1}(\alpha)\, |x|^2 - \alpha \}
	\label{eq:Phi-M-sup-of-quadratics}
\end{align}
for all $x\in\R^n$, in which $a^{-1}$ is as per \er{eq:exact-a}. Furthermore, the optimizer in \er{eq:Phi-M-sup-of-quadratics} is
\begin{align}
	& \hat\alpha^{M*}(|x|^2)
	\doteq\argmax_{\alpha\in[-\phi(0),M]} \{ a^{-1}(\alpha)\, |x|^2 - \alpha \}
	= \left\{ \ba{rl}
			a\circ\phi'(|x|^2)\,, 
			& |x|^2\le\hat\rho(M)\,,
			\\
			M\,,
			& |x|^2>\hat\rho(M)\,,
	\ea \right.
	\label{eq:alpha-M-star}
\end{align}
where $\phi'$, $a$, $\hat\rho$ are as per \er{eq:phi-properties-1-2}, \er{eq:exact-a}, \er{eq:Phi-M-from-A-and-rho-M}; and

\item $\Phi^M(|\cdot|^2):\R^n\rightarrow\R$ defined by \er{eq:Phi-M-sup-of-quadratics} is pointwise non-decreasing in $M\in\R_{\ge -\phi(0)}$, and converges pointwise to $\Phi(|\cdot|^2):\R^n\rightarrow\ol{\R}^+$ of \er{eq:Phi-sup-of-quadratics} in the limit as $M\rightarrow\infty$.
\end{enumerate}
\end{theorem}
\begin{proof}
{\em (i)} Fix $M\in\R_{\ge -\phi(0)}$, $x\in\R^n$. Applying Lemma \ref{lem:approx-convex-duality} {\em (i)}, note that the optimizer \er{eq:beta-M-star-explicit} in \er{eq:Phi-M-from-A-and-rho-M} satisfies $\hat\beta^{M*}(|x|^2) \in [\phi'(0), a^{-1}(M)]$, as $|x|^2\in\R_{\ge 0}$. Meanwhile, $a$ of \er{eq:exact-a} defines a change of variable $\alpha= a(\beta)$ for all $\beta\in[\phi'(0),\infty)$, via Lemma \ref{lem:a-properties}. Hence, $\Phi^M(|x|^2)$ transforms from \er{eq:Phi-M-from-A-and-rho-M} to
$
	\Phi^M(|x|^2)
	= \sup_{\beta\in[\phi'(0), a^{-1}(M)]} \{ \beta\, |x|^2 - \Ahat(\beta) \}
	= \sup_{\beta\in[a^{-1}(-\phi(0)), a^{-1}(M)]} \{ \beta\, |x|^2 - a(\beta) \}
$
via \er{eq:exact-Phi-and-A}, which yields \er{eq:Phi-M-sup-of-quadratics}. The same change of variable applied to $\hat\beta^{M*}(|x|^2)$ of \er{eq:beta-M-star-explicit} yields \er{eq:alpha-M-star}.

{\em (ii)} Immediate by Lemma \ref{lem:approx-convex-duality} {\em (iii)}.
\end{proof}

\begin{corollary}
\label{cor:Phi-M-bound}
Given $K\in\R_{\ge -\phi'(0)}$ as per \er{eq:cost-I-bar}, \er{eq:phi-properties-1-2},
\begin{align}
	& \Phi^M(\rho) \ge \phi'(0) \, \rho + \phi(0),
	&&
	K\, \rho + \Phi^M(\rho) \ge [K + \phi'(0)]\, \rho + \phi(0) \ge \phi(0),
	\nn
	\\
	& \Phi(\rho) \ge \phi'(0) \, \rho + \phi(0),
	&& K\, \rho + \Phi(\rho) \ge \phi(0),
\end{align}
for all $M\in\R_{\ge -\phi(0)}$, $\rho\in\R_{\ge 0}$.
\end{corollary}
\begin{proof}
Fix $M\in\R_{\ge -\phi(0)}$, $\rho\in\R_{\ge 0}$. Applying Theorem \ref{thm:Phi-M-sup-of-quadratics} {\em (i)}, $\alpha \doteq -\phi(0)$ is suboptimal in the right-hand side of \er{eq:Phi-M-sup-of-quadratics}, so that $\Phi^M(\rho) \ge a^{-1}(-\phi(0))\, \rho + \phi(0) = \phi'(0)\, \rho + \phi(0)$ and 
$K\, \rho + \Phi^M(\rho) \ge [K + \phi'(0)]\, \rho + \phi(0) \ge \phi(0)$, by \er{eq:exact-a} and \er{eq:phi-properties-1-2}.
As $\Phi^M(\rho)$ is non-decreasing in $M$ by Theorem \ref{thm:Phi-M-sup-of-quadratics} {\em (ii)}, the same inequalities hold with $M\rightarrow\infty$.
\end{proof}


\subsection{Approximate regulator problem and its convergence to the exact problem}
The sup-of-quadratics representation \er{eq:Phi-sup-of-quadratics} for the convex barrier function $\Phi$ of \er{eq:barrier}, and its convergent approximation \er{eq:Phi-M-sup-of-quadratics}, can be used to formulate an approximate regulator problem for \er{eq:value-W}. Given $t\in\R_{\ge 0}$, $M\in\R_{\ge -\phi(0)}$, the value function $\ol{W}_t^M:\R^n\rightarrow\R$ for this approximate problem is defined by
\begin{align}
	\ol{W}_t^M(x)
	& \doteq \inf_{u\in\cU[0,t]} \bar J_t^M(x,u)
	\label{eq:value-W-M}
\end{align}
for all $x\in\R^n$, with $\bar J_t^M:\R^n\times\cU[0,t]\rightarrow\R$ defined with respect to $I_t^\kappa$ and $\Psi$ of \er{eq:cost-I-bar}, \er{eq:cost-Psi} and $\bar I_t^M:\R^n\times\cU[0,t]\rightarrow\R$ by
\begin{align}
	\bar J_t^M(x,u) & \doteq \bar I_t^M(x,u) + I_t^\kappa(u) + \Psi(\xi_t),
	\label{eq:cost-J-M}
	\\
	\bar I_t^M(x,u)
	& \doteq \int_0^t \ts{\frac{K}{2}}\, |\xi_s|^2 + \demi\, \Phi^M(|\xi_s|^2) \, ds,
	\label{eq:cost-I-J-M}
\end{align}
for all $x\in\R^n$, $u\in\cU[0,t]$, in which $\xi\doteq\chi(x,u)$ and $\Phi^M$ are as per \er{eq:dynamics} and \er{eq:Phi-M-from-A-and-rho-M}, \er{eq:Phi-M-explicit}, \er{eq:Phi-M-sup-of-quadratics} respectively, and $K\in\R_{\ge -\phi'(0)}$, $\kappa\in\R_{>0}$ are fixed. This approximate problem recovers the original problem of \er{eq:value-W} in the limit as $M\rightarrow\infty$, as formalized by the theorem below. For convenience, $\bar J_t^\infty:\R^n\times\cU[0,t]\rightarrow\ol{\R}^+$ and $\ol{W}_t^\infty:\R^n\rightarrow\ol{\R}^+$ are defined by
\begin{align}
	\bar J_t^\infty(x,u)
	& \doteq \sup_{M\in\R_{\ge -\phi(0)}} \bar J_t^M(x,u) 
	\label{eq:J-infty}
	\\
	\ol{W}_t^\infty(x)
	& \doteq \sup_{M\in\R_{\ge -\phi(0)}} \ol{W}_t^M(x)
	\label{eq:W-infty}
\end{align}
for all $x\in\R^n$, $u\in\cU[0,t]$.

\begin{theorem}
\label{thm:monotone}
Given $t\in\R_{\ge 0}$, the cost and value functions $\bar J_t^M$, $\ol{W}_t^M$ of \er{eq:cost-J-M}, \er{eq:value-W-M} are pointwise non-decreasing in $M\in\R_{\ge -\phi(0)}$, and satisfy
\begin{align}
	-\infty < \bar J_t^M(x,u)
	& \le \bar J_t^\infty(x,u) 
	= \lim_{M\rightarrow\infty} \bar J_t^M(x,u) = \bar J_t(x,u),
	\label{eq:order-cost}
	\\
	-\infty < \ol{W}_t^M(x)
	& \le \ol{W}_t^\infty(x) 
	=  \lim_{M\rightarrow\infty} \ol{W}_t^M(x) = \ol{W}_t(x),
	\quad
	\R^n = \dom\ol{W}_t^M\supset\dom\ol{W}_t\ne\emptyset\,,
	\label{eq:order-value}
\end{align}
for all $x\in\R^n$, $u\in\cU[0,t]$. where $\bar J_t,\, \bar J_t^\infty:\R^n\times\cU[0,t]\rightarrow\ol{\R}^+$ and
$\ol{W}_t,\, \ol{W}_t^\infty:\R^n\rightarrow\ol{\R}^+$ are defined by \er{eq:value-W}, \er{eq:cost-J-bar} and \er{eq:J-infty}, \er{eq:W-infty}.
\end{theorem}
\begin{proof}
Fix $t\in\R_{\ge 0}$, $x\in\R^n$.
{\em [Non-decreasing property]}
This is immediate by inspection of \er{eq:value-W-M}, \er{eq:cost-J-M}, \er{eq:cost-I-J-M}, and the non-decreasing property of $\Phi^M(|\cdot|^2)$ provided by Theorem \ref{thm:Phi-M-sup-of-quadratics} {\em (ii)}.

{\em [Left-hand inequalities in \er{eq:order-cost}, \er{eq:order-value}]} Immediate by the definition of $\bar J_t^\infty$, $\ol{W}_t^\infty$ in \er{eq:J-infty}, \er{eq:W-infty}. Also, Corollary \ref{cor:Phi-M-bound} implies that $-\infty < \ts{\frac{\phi(0)}{2}} \, t \le\bar J_t^M(x,u)$. Moreover, as $u$ is arbitrary here, $-\infty< \ts{\frac{\phi(0)}{2}} \, t \le\ol{W}_t^M(x) = \inf_{u\in\cU[0,t]} \bar J_t^M(x.u)$.

{\em [Domain properties in \er{eq:order-value}]} Fix $M\in\R_{\ge -\phi(0)}$. It is immediate by the left-hand inequality in \er{eq:order-value} and Lemma \ref{lem:dom-value-W-bar} that $\dom\ol{W}_t^M\supset\dom\ol{W}_t\ne\emptyset$ holds. For the remaining assertion, fix $u\in\cU[0,t]$, and recall that  $\chi(x,u)\in C([0,t];\R^n)$ via \er{eq:dynamics} and Lemma \ref{lem:chi}. Applying Lemma \ref{lem:approx-convex-duality} {\em (ii)}, $\Phi^M(|\chi(x,u)|^2)\in C([0,t];\R)$, so that $\bar J_t^M(x,u) < \infty$ by inspection of \er{eq:cost-J-M}. Hence, as $u$ is arbitrary, $\ol{W}_t^M(x) = \inf_{u\in\cU[0,t]} \bar J_t^M(x.u) <\infty$, and as $x\in\R^n$ is arbitrary, $\dom\ol{W}_t^M = \R^n$.

{\em [Limits in \er{eq:order-cost}, \er{eq:order-value}]} Immediate from the non-decreasing property above.

{\em [Right-hand equality in \er{eq:order-cost}]} Fix $u\in\cU[0,t]$. In view of Corollary \ref{cor:Phi-M-bound} and \er{eq:cost-I-J-M}, it follows that
$
	\bar I_t^{M}(x,u) = \int_0^t \nu_s^M\, ds + \ts{\frac{\phi(0)}{2}}\, t
$,
where $s\mapsto\nu_s^M \doteq \ts{\frac{K}{2}} \, |\xi_s|^2 + \demi\, \Phi^M(|\xi_s|^2) - \ts{\frac{\phi(0)}{2}}$ is nonnegative by Corollary \ref{cor:Phi-M-bound}, non-decreasing in $M\in\R_{\ge -\phi(0)}$ by Theorem \ref{thm:Phi-M-sup-of-quadratics} {\em (ii)}, and continuous (and hence measurable). Applying the monotone convergence theorem,
\begin{align}
	\sup_{M>-\phi(0)}  \bar I_t^{M}(x,u)
	& = \lim_{M\rightarrow\infty} \int_0^t \nu_s^M\, ds + \ts{\frac{\phi(0)}{2}}\, t
	\nn\\
	& = \int_0^t \lim_{M\rightarrow\infty} \nu_s^M\, ds + \ts{\frac{\phi(0)}{2}}\, t
	= \int_0^t \ts{\frac{K}{2}} \, |\xi_s|^2 + \demi\, \lim_{M\rightarrow\infty} \Phi^M(|\xi_s|^2) \, ds
	= \bar I_t(x,u),
	\nn
\end{align}
in which final equality follows by Lemma \ref{lem:approx-convex-duality} {\em (iii)}. Hence, recalling \er{eq:cost-J-M}, \er{eq:J-infty},
\begin{align}
	\bar J_t^\infty(x,u)
	& = \sup_{M\in\R_{\ge -\phi(0)}} \bar I_t^M(x,u) + I_t^\kappa(u) + \Psi(\xi_t)
	= \bar I_t(x,u) + I_t^\kappa(u) + \Psi(\xi_t) = \bar J_t(x,u)\,,
	\nn
\end{align}
in which it is noted that $x\in\R^n$ and $u\in\cU[0,t]$ are arbitrary.

{\em [Right-hand equality in \er{eq:order-value}]} Applying \er{eq:J-infty}, \er{eq:W-infty}, \er{eq:order-cost},
\begin{align}
	\ol{W}_t^\infty(x)
	& = \sup_{M\ge -\phi(0)} \inf_{u\in\cU[0,t]} \bar J_t^M(x,u) \le \inf_{u\in\cU[0,t]} \sup_{M\ge -\phi(0)} \bar J_t^M(x,u) 
	= \inf_{u\in\cU[0,t]} \bar J_t(x,u) = \ol{W}_t(x)\,.
	\label{eq:one-side}
\end{align}
It remains to demonstrate the opposite inequality. To this end, fix an arbitrary $\eps\in\R_{>0}$, and select any non-decreasing sequence $\{M_k\}_{k\in\N}\subset\R_{\ge -\phi(0)}$ such that $\lim_{k\rightarrow\infty} M_k = \infty$. Define a sequence $\{u_k^\eps\}_{k\in\N}\subset\cU[0,t]$ by
\begin{align}
	\bar J_t^{M_k}(x,u_k^\eps) 
	& < \ol{W}_t^{M_k}(x) + \eps,
	\label{eq:near-optimal-sequence}
\end{align}
and note by definition \er{eq:value-W-M} of $\ol{W}_t^{M_k}(x)$ that this is always possible. Suppose that $\{u_k^\eps\}_{k\in\N}$ is unbounded. Applying Corollary \ref{cor:Phi-M-bound} in the definition \er{eq:cost-I-J-M} of $\bar I_t^{M_k}(x,\cdot)$, note that $\bar I_t^{M_k}(x,u) \ge \ts{\frac{\phi(0)}{2}} \, t$ for all $u\in\cU[0,t]$. Combining this bound with \er{eq:cost-J-M} and \er{eq:near-optimal-sequence} yields
\begin{align}
	\ol{W}_t^\infty(x)
	& = \lim_{k\rightarrow\infty} \ol{W}_t^{M_k}(x)
	\ge \lim_{k\rightarrow\infty} 
	\bar J_t^{M_k}(x,u_k^\eps) - \eps
	\ge \ts{\frac{\phi(0)}{2}} \, t - \eps + \ts{\frac{\kappa}{2}}\, \lim_{k\rightarrow\infty} \|u_k^\eps\|_{\cU[0,t]}^{2} = \infty,
	\nn
\end{align}
which yields $\ol{W}_t^\infty(x) \ge \ol{W}_t(x)$, as required to complete the proof in that unbounded case.

Alternatively, suppose that $\{u_k^\eps\}_{k\in\N}$ is bounded, i.e. there exists $\ol{U}\in\R_{\ge 0}$ such that $\{u_k^\eps\}_{k\in\N}\in\cB_{\cU[0,t]}[0;\ol{U}]$. Lemma \ref{lem:chi} {\em (iii)} implies that there exists a subsequence $\{\tilde u_k^\eps\}_{k\in\N} \subset \{u_k^\eps\}_{k\in\N}$ such that $\tilde\xi_k^\eps\rightarrow\bar\xi^\eps$ uniformly as $k\rightarrow\infty$, where $\tilde\xi_k^\eps\doteq\chi(x,\tilde u_k^\eps)$. In view of \er{eq:cost-I-bar}, \er{eq:cost-I-J-M}, and Corollary \ref{cor:Phi-M-bound}, define a sequence $\{\tilde\nu_k^\eps\}_{k\in\N}$ of maps from $[0,t]$ to $\R_{\ge 0}$, and its candidate limit $\bar\nu^\eps:[0,t]\rightarrow\ol{\R}_{\ge 0}^+$, by
\begin{align}
	[\tilde\nu_k^\eps]_s
	& \doteq \ts{\frac{K}{2}} \, |[\tilde\xi_k^\eps]_s|^2 + \demi\, \Phi^{M_k}(|[\tilde\xi_k^\eps]_s|^2) - \ts{\frac{\phi(0)}{2}},
	\quad
	\bar\nu^\eps_s
	\doteq \ts{\frac{K}{2}} \, |[\bar\xi^\eps]_s|^2 + \demi\, \Phi(|[\bar\xi^\eps]_s|^2) - \ts{\frac{\phi(0)}{2}},
	\label{eq:nu-maps}
\end{align}
for all $s\in[0,t]$, $k\in\N$. Fix any $s\in[0,t]$, $j\in\N$. Note that by monotonicity of $\{\Phi^{M_k}\}_{k\in\N}$, see Lemma \ref{lem:approx-convex-duality} {\em (iii)} or Theorem \ref{thm:Phi-M-sup-of-quadratics} {\em (ii)},
$
	\Phi^{M_k}(|[\tilde\xi_k^\eps]_s|^2) 
	= [\Phi^{M_k}(|[\tilde\xi_k^\eps]_s|^2) - \Phi^{M_{j}}(|[\tilde\xi_k^\eps]_s|^2)] + \Phi^{M_{j}}(|[\tilde\xi_k^\eps]_s|^2) 
	\ge \Phi^{M_{j}}(|[\tilde\xi_k^\eps]_s|^2)
$
for all $k\ge j$. Hence, as $\Phi^{M_j}$ is continuous, $\liminf_{k\rightarrow\infty} \Phi^{M_k}(|[\tilde\xi_k^\eps]_s|^2) \ge \Phi^{M_{j}}(|\bar\xi^\eps_s|^2)$, so that $\liminf_{k\rightarrow\infty} \Phi^{M_k}(|[\tilde\xi_k^\eps]_s|^2) \ge \lim_{j\rightarrow\infty}\Phi^{M_{j}}(|\bar\xi^\eps_s|^2) = \Phi(|\bar\xi_s^\eps|^2)$. As $\lim_{k\rightarrow\infty} |[\tilde\xi_k^\eps]_s|^2 = |[\bar\xi^\eps]_s|^2$, \er{eq:nu-maps} subsequently yields that
\begin{align}
	&  \bar\nu_s^\eps
	\le \liminf_{k\rightarrow\infty} [\tilde\nu_k^\eps]_s\,.
	\label{eq:liminf-nu}
\end{align}

By inspection, $\bar\nu_s^\eps = \infty$ implies that $\lim_{k\rightarrow\infty} [\tilde\nu_k^\eps]_s = \infty = \bar\nu_s^\eps$. 

Alternatively, suppose that $\bar\nu_s^\eps<\infty$. In view of \er{eq:nu-maps}, define $[\hat\nu_k^\eps]_s \doteq \ts{\frac{K}{2}}\, |[\tilde\xi_k^\eps]_s|^2 + \demi\, \Phi(|[\tilde\xi_k^\eps]_s|^2)  - \ts{\frac{\phi(0)}{2}}$ for all $k\in\N$. As $\bar\nu_s^\eps<\infty$, there exists an open interval containing $|\bar\xi_s^\eps|^2$ on which $\Phi$ is continuous, and $\lim_{k\rightarrow\infty} [\tilde\xi_k^\eps]_s = [\bar\xi^\eps]_s$, so that $[\hat\nu_k^\eps]_s<\infty$ for all $k\in\N$ sufficiently large, and $\lim_{k\rightarrow\infty} [\hat\nu_k^\eps]_s = \bar\nu_s^\eps$. Note further that $[\tilde\nu_k^\eps]_s \le [\hat\nu_k^\eps]_s$ for all $k\in\N$, again by Lemma \ref{lem:approx-convex-duality} {\em (iii)}. Hence, 
\begin{align}
	\limsup_{k\rightarrow\infty}
	[\tilde\nu_k^\eps]_s
	& \le \limsup_{k\rightarrow\infty} [\hat\nu_k^\eps]_s 
	= \bar\nu_s^\eps\,.
	\label{eq:limsup-nu}
\end{align}
Consequently, combining \er{eq:liminf-nu} and \er{eq:limsup-nu}, and recalling the $\bar\nu_s^\eps = \infty$ case above, it may be concluded that $\lim_{k\rightarrow\infty} [\tilde\nu_k^\eps]_s = \bar\nu_s^\eps$ for both the $\bar\nu_s^\eps=\infty$ and the $\bar\nu_s^\eps<\infty$ cases.

Next, recall by definition \er{eq:nu-maps} and Corollary \ref{cor:Phi-M-bound}, that $\{\tilde v_k^\eps\}_{k\in\N}$ defines a non-negative sequence of functions in $C([0,t];\R)$. Consequently, every element of this sequence is measurable and non-negative, so that Fatou's lemma yields
$
	\int_0^t \bar\nu_s^\eps\, ds = \int_0^t \liminf_{k\rightarrow\infty} [\tilde\nu_k^\eps]_s
	\le \liminf_{k\rightarrow\infty} \int_0^t [\tilde\nu_k^\eps]_s\, ds
$.
Hence, recalling \er{eq:nu-maps}, the definitions of $\tilde\xi_k^\eps$, $\bar\xi^\eps$ prior, and \er{eq:cost-I-bar}, \er{eq:cost-I-J-M},
\begin{align}
	\bar I_t(x,\bar u^\eps)
	& = \int_0^t \ts{\frac{K}{2}}\, |\bar\xi_s^\eps|^2 + \demi\, \Phi(|\bar\xi_s^\eps|^2)\, ds
	=
	\int_0^t \bar\nu_s^\eps\, ds + \ts{\frac{\phi(0)}{2}}\, t
	\label{eq:cost-sequence-1}
	\\
	& \le  \liminf_{k\rightarrow\infty} \int_0^t [\tilde\nu_k^\eps]_s\, ds + \ts{\frac{\phi(0)}{2}}\, t
	= \liminf_{k\rightarrow\infty} \int_0^t \ts{\frac{K}{2}}\, |[\tilde\xi_k^\eps]_s|^2 + \demi\, \Phi^{M_k} (|[\tilde\xi_k^\eps]_s|^2) \, ds
	= \liminf_{k\rightarrow\infty} \bar I_t^{M_k}(x,\tilde u_k^\eps)\,.
	\nn
\end{align}
Meanwhile, by weak convergence of $\tilde u_k^\eps$ to $\bar u^\eps$, $\|\bar u^\eps\|_{\cU[0,t]} \le \liminf_{k\rightarrow\infty} \| \tilde u_k^\eps\|_{\cU[0,t]}$, so that \er{eq:cost-I-bar} implies
\begin{align}
	I_t^\kappa(\bar u^\eps)
	& = \ts{\frac{\kappa}{2}} \|\bar u^\eps\|_{\cU[0,t]}^2
	\le \liminf_{k\rightarrow\infty} \ts{\frac{\kappa}{2}} \, \|\tilde u_k^\eps\|_{\cU[0,t]}^2
	=
	\liminf_{k\rightarrow\infty} I_t^\kappa(\tilde u_k^\eps)\,.
	\label{eq:cost-sequence-2}
\end{align}
Moreover, continuity of $[\chi(x,\cdot)]_t$ by Lemma \ref{lem:chi} {\em (ii)}, along with continuity of $\Psi_t$ of \er{eq:cost-Psi}, imply that
\begin{align}
	\Psi_t(\bar\xi_t^\eps)
	& = \lim_{k\rightarrow\infty} \Psi_t([\tilde\xi_k^\eps]_t)\,.
	\label{eq:cost-sequence-3}
\end{align}
Combining \er{eq:cost-sequence-1}, \er{eq:cost-sequence-2}, \er{eq:cost-sequence-3} via \er{eq:cost-J-bar}, \er{eq:cost-J-M} yields
\begin{align}
	\bar J_t(x,\bar u^\eps)
	& = \bar I_t(x,\bar u^\eps) + I_t^\kappa(\bar u^\eps) + \Psi_t(\bar\xi_t^\eps)
	\le \liminf_{k\rightarrow\infty} \bar I_t^{M_k}(x,\tilde u_k^\eps) + \liminf_{k\rightarrow\infty} I_t^\kappa(\tilde u_k^\eps)
	+ \lim_{k\rightarrow\infty} \Psi_t([\tilde\xi_k^\eps]_t)
	\nn\\
	& \le \liminf_{k\rightarrow\infty} \left\{ \bar I_t^{M_k}(x,\tilde u_k^\eps) + I_t^\kappa(\tilde u_k^\eps) + \Psi_t([\tilde\xi_k^\eps]_t) \right\}
	= \liminf_{k\rightarrow\infty} \bar J_t^{M_k}(x,\tilde u_k^\eps).
	\label{eq:cost-sequence-4}
\end{align}
Hence, applying \er{eq:near-optimal-sequence} and \er{eq:cost-sequence-4} while recalling that $\{\tilde u_k^\eps\}_{k\in\N} \subset \{ u_k^\eps \}_{k\in\N}$ is a subsequence of the near-optimal inputs involved, and noting that $\bar u^\eps$ is suboptimal in the definition \er{eq:value-W} of $\ol{W}_t(x)$, yields
\begin{align}
	\ol{W}_t(x)
	& \le \bar J_t(x,\bar u^\eps) \le \liminf_{k\rightarrow\infty} \bar J_t^{M_k}(x,\tilde u_k^\eps)
	\le \liminf_{k\rightarrow\infty} \ol{W}_t^{M_k}(x) + \eps
	= \ol{W}^\infty(x)+ \eps\,.
	\nn
\end{align}
As $\eps\in\R_{>0}$ is arbitrary, it follows that $\ol{W}_t(x) \le \ol{W}^\infty(x)$. Recalling \er{eq:one-side} and the fact that $t\in\R_{\ge 0}$ and $x\in\R^n$ are also arbitrary completes the proof of the equality in \er{eq:order-value}.
\end{proof}


\section{Optimal trajectories and constraint satisfaction}
\label{sec:existence}

Existence and uniqueness of the optimal trajectories in \er{eq:value-W}, \er{eq:value-W-M} is demonstrated via analysis of the attendant cost functions \er{eq:cost-J-bar}, \er{eq:cost-I-J-M}. In particular, these cost functions are shown to be proper, lower semicontinuous, strictly convex, and coercive. These properties are demonstrated to be sufficient for the required existence and uniqueness of the optimal controls involved, and hence the corresponding trajectories. The behaviour of these optimal trajectories relative to the desired state constraint is subsequently determined. 


\subsection{Existence and uniqueness of the optimal controls}
In order to demonstrate that the cost functions $\bar J_t(x,\cdot), \, \bar J_t^M(x,\cdot): \cU[0,t]\rightarrow\ol{\R}^+$ of \er{eq:cost-J-bar}, \er{eq:cost-I-J-M} are proper, convex, and coercive for fixed $t\in\R_{\ge 0}$, $M\in\R_{\ge -\phi(0)}$, $x\in\R^n$, it is useful to consider the map $\gamma_x^{s,\alpha}:\cU[0,t]\rightarrow\R$ defined for fixed $x\in\R^n$, $s\in[0,t]$, $\alpha\in\R_{\ge -\phi(0)}$ by
\begin{align}
	\gamma_x^{s,\alpha}(u)
	& \doteq \demi\, [ K + a^{-1}(\alpha)]\, |[\chi(x,u)]_s|^2 - \ts{\frac{\alpha}{2}}
	\label{eq:gamma-fn}
\end{align}
for all $u\in\cU[0,t]$, in which $\chi$ is as per \er{eq:dynamics}.
\begin{lemma}
\label{lem:convex-gamma}
Given $t\in\R_{\ge 0}$, $x\in\R^n$, $s\in[0,t]$, $\alpha\in\R_{\ge -\phi(0)}$, $\gamma_x^{s,\alpha}:\cU[0,t]\rightarrow\R$ of \er{eq:gamma-fn} is convex.
\end{lemma}
\begin{proof}
Fix $t\in\R_{\ge 0}$, $x\in\R^n$, $s\in[0,t]$, $\alpha\in\R_{\ge -\phi(0)}$. As $u\mapsto[\chi(x,u)]_s$ is affine by \er{eq:dynamics}, convexity of $\gamma_x^{s,\alpha}$ follows by inspection of \er{eq:gamma-fn}, properties {\em (ii)} and {\em (iii)} of \er{eq:phi-properties-1-2}, and Lemma \ref{lem:a-properties}. 
\end{proof}

\begin{lemma}
\label{lem:convex-cost-J}
Given any $t\in\R_{>0}$, $M\in\R_{\ge -\phi(0)}$, the cost functions $\bar J_t(x,\cdot), \, \bar J_t^M(x,\cdot): \cU[0,t]\rightarrow\ol{\R}^+$ defined for $x\in\R^n$ by \er{eq:cost-J-bar}, \er{eq:cost-I-J-M} satisfy the following properties:
\begin{enumerate}[\it (i)] \itsep
\item $\bar J_t^M(x,\cdot)$ and $\bar J_t(x,\cdot)$ are respectively continuous and lower semicontinuous for all $x\in\R^n$;
\item Both are strictly convex and coercive for all $x\in\R^n$; and
\item $\bar J_t^M(x,\cdot)$ and $\bar J_t(y,\cdot)$ are proper for all $x\in\dom\ol{W}_t^M = \R^n$ and all $y\in\dom\ol{W}_t$.
\end{enumerate}
\end{lemma}

\newcommand{\ball}		{{\mathcal{B}}}

\begin{proof}
Fix $t\in\R_{>0}$, $M\in\R_{\ge -\phi(0)}$. {\em (i)} Fix $x\in\R^n$. 
{\em [Continuity of $\bar J_t^M(x,\cdot)$]} By inspection of \er{eq:cost-J-M}, \er{eq:cost-I-J-M}, and \er{eq:cost-I-bar}, \er{eq:cost-Psi}, continuity of $\bar J_t^M(x,\cdot)$ on $\cU[0,t]$ requires continuity of its constituent maps $\bar I_t^M(x,\cdot)$, $I_t^\kappa$, and $\Psi([\chi(x,\cdot)]_t)$ on $\cU[0,t]$. This is immediate for $I_t^\kappa$ and $\Psi([\chi(x,\cdot)]_t)$, by \er{eq:cost-I-bar}, \er{eq:cost-Psi}, and Lemma \ref{lem:chi} {\em (ii)}. The same conclusion follows for $\bar I_t^M(x,\cdot)$, by application of Lemma \ref{lem:chi} {\em (ii)} and Lemma \ref{lem:approx-convex-duality} {\em (ii)}.

{\em [Lower semicontinuity of $\bar J_t(x,\cdot)$]} Fix $u\in\cU[0,t]$, and any sequence $\{\tilde u_i\}_{i\in\N}\subset\cU[0,t]$ such that $\lim_{i\rightarrow\infty} \|u - \tilde u_i\|_{\cU[0,t]} = 0$. By continuity of $\bar J_t^M(x,\cdot)$, note that $\bar J_t^M(x,u) = \lim_{i\rightarrow\infty} \bar J_t^M(x,\tilde u_i)$. 
Hence, applying Theorem \ref{thm:monotone}, and in particular \er{eq:J-infty}, \er{eq:order-cost},
\begin{align}
	\bar J_t(x,u)
	& = \sup_{M\ge -\phi(0)} \bar J_t^M(x,u)
	= \sup_{M\ge -\phi(0)} \liminf_{i\rightarrow\infty} \bar J_t^M(x,\tilde u_i)
	=  \sup_{M\ge -\phi(0)} \sup_{j\in\N} \inf_{i>j} \bar J_t^M(x,\tilde u_i)
	\nn\\
	& \le \sup_{j\in\N} \inf_{i>j} \sup_{M\ge -\phi(0)} \bar J_t^M(x,\tilde u_i) = \liminf_{i\rightarrow\infty} \sup_{M\ge -\phi(0)} \bar J_t^M(x,\tilde u_i)
	= \liminf_{i\rightarrow\infty} \bar J_t(x,\tilde u_i)\,.
	\nn
\end{align}
As $u\in\cU[0,t]$ and $\{\tilde u_i\}_{i\in\N}\subset\cU[0,t]$ are arbitrary, the assertion follows.

{\em (ii)} Fix $x\in\R^n$.
{\em [Convexity of $\bar J_t^M(x,\cdot)$]} Fix $u\in\cU[0,t]$, and $\xi\doteq\chi(x,u)$ by \er{eq:dynamics}. By \er{eq:cost-I-bar}, \er{eq:Phi-M-sup-of-quadratics}, \er{eq:cost-I-J-M},
\begin{align}
	\bar I_t^M(x,u)
	& = \int_0^t \ts{\frac{K}{2}} \, |\xi_s|^2 + \demi\, \Phi^M(|\xi_s|^2) \, ds
	= \int_0^t \sup_{\alpha\in[-\phi(0),M]} \gamma_{x}^{s,\alpha}(u) \, ds, 
	\label{eq:I-M-1}
\end{align}
where $\gamma_{x}^{s,\alpha}$ is as per \er{eq:gamma-fn}. Recall by Lemma \ref{lem:convex-gamma} that $\gamma_{x}^{s,\alpha}:\cU[0,t]\rightarrow\R$ is convex for any $s\in[0,t]$, $\alpha\in\R_{\ge -\phi(0)}$. As convexity is preserved under suprema and integration, see \cite[Theorem 3 and (2.6), p.7]{R:74}, it follows by \er{eq:I-M-1} that $\bar I_t^M(x,\cdot):\cU[0,t]\rightarrow\R$ is convex. Consequently, as $\kappa\in\R_{>0}$ in \er{eq:cost-I-bar}, $\bar I_t^M(x,\cdot) + I_t^{\kappa}(\cdot):\cU[0,t]\rightarrow\R$ is strictly convex. Lastly, as $\Psi$ of \er{eq:cost-Psi} is convex by definition of $P_t\in\Sigma^n$, and $[\chi(x,\cdot)]_t:\cU[0,t]\rightarrow\R^n$ is affine, $\Psi([\chi(x,\cdot)]_t):\cU[0,t]\rightarrow\R$ is also convex. Hence, applying \er{eq:cost-J-M}, $\bar J_t^M(x,\cdot):\cU[0,t]\rightarrow\R$ is strictly convex.

{\em [Convexity of $\bar J_t(x,\cdot)$]}
Recalling the convexity argument immediately above, $\bar I_t^M(x,\cdot) + \Psi([\chi(x,\cdot)]_t) = \bar J_t^M(x,\cdot) - \ts{\frac{\kappa}{2}} \|\cdot\|_{\cU[0,t]}^2:\cU[0,t]\rightarrow\R$ is convex for all $M\in\R_{\ge -\phi(0)}$.
As convexity is preserved under suprema \cite[(2.6), p.7]{R:74}, and Theorem \ref{thm:monotone} implies that \er{eq:order-cost}, \er{eq:J-infty} hold, convexity of $\bar J_t(x,\cdot) - \ts{\frac{\kappa}{2}} \|\cdot\|_{\cU[0,t]}^2 = \sup_{M\ge -\phi(0)} \bar J_t^M(x,\cdot) - \ts{\frac{\kappa}{2}} \|\cdot\|_{\cU[0,t]}^2:\cU[0,t]\rightarrow\ol{\R}^+$ follows. Hence, $\bar J_t(x,\cdot)$ is strictly convex, as $\kappa\in\R_{>0}$. 

[Coercivity of $\bar J_t^M(x,\cdot)$] 
Recall by Corollary \ref{cor:Phi-M-bound} that $\ts{\frac{K}{2}} \, |\cdot|^2 + \demi\, \Phi^M(|\cdot|^2) \ge \ts{\frac{\phi(0)}{2}}$. Applying \er{eq:cost-J-bar}, \er{eq:cost-Psi},
\begin{align}
	\bar J_t^M(x,u)
	& = \int_0^t \ts{\frac{K}{2}} |\xi_s|^2 + \demi\, \Phi^M(|\xi_s|^2) \, ds + \ts{\frac{\kappa}{2}}\, \|u\|_{\cU[0,t]}^2 + \Psi(\xi_t)
	\ge \ts{\frac{\phi(0)}{2}}\, t + \ts{\frac{\kappa}{2}}\, \|u\|_{\cU[0,t]}^2,
	\label{eq:coercive-J-M-bound}
\end{align}
for all $x\in\R^n$, $u\in\cU[0,t]$. Hence, $\bar J_t^M(x,\cdot)$ is coercive, as $\kappa\in\R_{>0}$.

{\em [Coercivity of $\bar J_t(x,\cdot)$]} Follows by coercivity of $\bar J_t^M(x,\cdot)$ and \er{eq:order-cost} of Theorem \ref{thm:monotone}.

{\em (iii)} Lemma \ref{lem:dom-value-W-bar} demonstrates that $\dom\ol{W}_t\ne\emptyset$. Fix any $x\in\dom\ol{W}_t$. Select a near-optimal input $\tilde u\in\cU[0,t]$ in the definition \er{eq:value-W} of $\ol{W}_t(x)$, such that $\bar J_t^M(x,\tilde u) \le \bar J_t(x,\tilde u) < \ol{W}_t(x) + 1<\infty$, and note that this is always possible by Theorem \ref{thm:monotone}, i.e. \er{eq:order-cost}. Hence, $\dom\, \bar J_t^M(x,\cdot) \ne \emptyset \ne \dom\, \bar J_t(x,\cdot)$. Again
recalling \er{eq:order-cost}, along with \er{eq:coercive-J-M-bound}, note also that
$
	-\infty < \ts{\frac{\phi(0)}{2}}\, t + \ts{\frac{\kappa}{2}}\, \|u\|_{\cU[0,t]}^2
	\le 
	\bar J_t^M(x,u)
	\le 
	\bar J_t(x,u) 
$ for all $u\in\cU[0,t]$.
Hence, $\bar J_t(x,\cdot), \, \bar J_t^M(x,\cdot): \cU[0,t]\rightarrow\ol{\R}^+$ of \er{eq:cost-J-bar}, \er{eq:cost-I-J-M} are proper for any $x\in\dom\ol{W}_t$. Finally, recalling \er{eq:coercive-J-M-bound} yields that $\bar J_t^M(y,\cdot)$ is also proper for any $y\in\dom\ol{W}_t^M = \R^n$.
\end{proof}

With $t\in\R_{>0}$, existence and uniqueness of the optimal controls in \er{eq:value-W}, \er{eq:value-W-M} may now be established.

\begin{theorem}
\label{thm:existence}
Given any $t\in\R_{>0}$, $M\in\R_{\ge-\phi(0)}$, $x\in\dom\ol{W}_t$, $y\in\dom\ol{W}_t^M = \R^n$, there exist unique optimal controls $u^*, u^{M*}\in\cU[0,t]$ for the respective optimal control problems \er{eq:value-W}, \er{eq:value-W-M}, with
\begin{align}
	& u^*=\argmin_{u\in\cU[0,t]} \bar J_t(x,u),
	\qquad
	u^{M*}=\argmin_{u\in\cU[0,t]} \bar J_t^M(y,u).
	\label{eq:existence}
	\\[-3mm]
	& \nn
\end{align}
\end{theorem}
\begin{proof}
As the existence and uniqueness arguments for the two optimal controls in \er{eq:existence} are analogous, only the first is included. Fix any $t\in\R_{>0}$, and recall that $\dom\ol{W}_t\ne\emptyset$ by Lemma \ref{lem:dom-value-W-bar}. Fix any $x\in\dom\ol{W}_t$, and recall by Lemma \ref{lem:convex-cost-J} that $\bar J_t(x,\cdot):\cU[0,t]\rightarrow\ol{\R}^+$ is proper, lower semicontinuous, strictly convex, and coercive. Given $\ell\in\ol{\R}^+$, define the level set $\Lambda_\ell\subset\cU[0,t]$ by
\begin{align}
	\Lambda_\ell
	& \doteq \left\{ u\in\cU[0,t] \, \bigl| \, \bar J_t(x,u) \le \ell \right\}.
	\label{eq:level-sets}
\end{align}
As $\bar J_t(x,\cdot):\cU[0,t]\rightarrow\ol{\R}^+$ is proper and coercive, and \er{eq:coercive-J-M-bound} holds, there exists $\hat u\in\cU[0,t]$ such that
$-\infty < \ts{\frac{\phi(0)}{2}}\, t + \ts{\frac{\kappa}{2}}\, \|\hat u\|_{\cU[0,t]}^2 \le \bar J_t(x,\hat u) < \infty$. Consequently, $\ell_0\doteq \inf_{u\in\cU[0,t]} \bar J_t(x,u)$ is finite, i.e. $\ell_0\in\R$, and $\Lambda_\ell$ of \er{eq:level-sets} is guaranteed to be non-empty for all $\ell>\ell_0$. Moreover, as $\kappa\in\R_{>0}$, \er{eq:coercive-J-M-bound} implies that $\Lambda_\ell$ is bounded for all $\ell>\ell_0$, with $\Lambda_\ell \subset \cB_{\cU[0,t]}[0,r_\ell]$, $r_\ell \doteq [\ell - \ts{\frac{\phi(0)}{2}}\, t]^\frac{1}{2}$ (and note by inspection that $\ell_0\ge\ts{\frac{\phi(0)}{2}}\, t$). Define a decreasing sequence $\{\ell_k\}_{k\in\N}\subset\R$ such that $\lim_{k\rightarrow\infty} \ell_k = \ell_0$, and a corresponding sequence $\{u_k\}_{k\in\N}\subset\cU[0,t]$ such that $u_k\in\Lambda_{\ell_k}$. Note in particular that $u_k\in\Lambda_{\ell_1}\subset\cB_{\cU[0,t]}[0;r_{\ell_1}]$ as $\Lambda_{\ell_k}\supset\Lambda_{\ell_{k+1}}\ne\emptyset$, for all $k\in\N$. That is, $\{u_k\}_{k\in\N}$ is bounded. As per the proof of Theorem \ref{thm:monotone}, this implies the existence of a subsequence $\{ \hat u_k \}_{k\in\N} \subset \{ u_k \}_{k\in\N}$ and a $\bar u\in\cU[0,t]$ such that $\hat\xi_k\rightarrow\bar\xi$ uniformly as $k\rightarrow\infty$, where $\hat\xi_k\doteq\chi(x,\hat u_k)$, $\bar\xi \doteq \chi(x,\bar u)$. Define a sequence of maps $\{\hat\nu_k\}_{k\in\N}$ from $[0,t]$ to $\ol{\R}^+$ and its candidate limit $\bar \nu:[0,t]\rightarrow\ol{\R}^+$ by
\begin{align}
	[\hat\nu_k]_s
	& \doteq \ts{\frac{K}{2}}\, |[\hat\xi_k]_s|^2 + \demi\, \Phi(|[\hat\xi_k]_s|^2) - \ts{\frac{\phi(0)}{2}}\,,
	\quad
	\bar\nu_s
	\doteq \ts{\frac{K}{2}}\, |\bar\xi_s|^2 + \demi\, \Phi(|\bar\xi_s|^2) - \ts{\frac{\phi(0)}{2}}\,,
	\label{eq:nu-hat}
\end{align}
for all $k\in\N$, $s\in[0,t]$. By inspection, note that $\lim_{k\rightarrow\infty} [\hat\nu_k]_s = \bar\nu_s$, irrespective of finiteness of $\bar\nu_s$, for all $s\in[0,t]$. Repeating the Fatou's Lemma argument of the proof of Theorem \ref{thm:monotone},
$
	\int_0^t \bar\nu_s \, ds
	= \int_0^t \liminf_{k\rightarrow\infty} [\hat\nu_k]_s \, ds
	\le \liminf_{k\rightarrow\infty} \int_0^t [\hat\nu_k]_s \, ds
$.
Hence, \er{eq:cost-I-bar}, \er{eq:nu-hat} imply that
$
	\bar I_t(x,\bar u)
	= \int_0^t \bar\nu_s \, ds + \ts{\frac{\phi(0)}{2}}\, t
	\le \liminf_{k\rightarrow\infty} \int_0^t [\hat\nu_k]_s \, ds  + \ts{\frac{\phi(0)}{2}}\, t
	= \liminf_{k\rightarrow\infty} \bar I_t(x,\hat u_k)$,
which, again following the proof of Theorem \ref{thm:monotone}, yields $\bar J_t(x,\bar u) \le \liminf_{k\rightarrow\infty} \bar J_t(x,\hat u_k)$.
Abuse notation by relabelling $\{\ell_k\}_{k\in\N}$ to match the subsequence $\{\hat u_k\}_{k\in\N}$ of $\{u_k\}_{k\in\N}$, and note that $\hat u_k\in\Lambda_{\ell_k}$. Hence, 
$\bar J_t(x,\bar u) \le \liminf_{k\rightarrow\infty} \bar J_t(x,\hat u_k) \le \liminf_{k\rightarrow\infty} \ell_k = \ell_0$.
Consequently, by definition of $\ell_0$, $\bar J_t(x,\bar u) = \ell_0 = \inf_{u\in\cU[0,t]} \bar J_t(x,u)$, so that $\bar u\in\argmin_{u\in\cU[0,t]} \bar J_t(x,u)$ and the argmin is non-empty. Suppose there exists a $\tilde u\in\argmin_{u\in\cU[0,t]} \bar J_t(x,u)$ such that $\tilde u \ne \bar u$, and define $\breve u\doteq \demi\, (\bar u + \tilde u)\in\cU[0,t]$. By strict convexity of $\bar J_t(x,\cdot)$, $\bar J_t(x,\breve u) < \demi\, \bar J_t(x,\bar u) + \demi\, \bar J_t(x,\tilde u) = \bar J_t(x,\bar u)$, contradicting $\bar u\in\argmin_{u\in\cU[0,t]} \bar J_t(x,u)$. Hence, the argmin is a singleton, with $\{ u^* \} = \argmin_{u\in\cU[0,t]} \bar J_t(x,u)$, as $u^* \doteq \bar u = \tilde u$.
\end{proof}

\begin{remark}
\label{rem:convex-t-0}
In the $t=0$ special case, note by \er{eq:cost-I-bar} that $I_t^\kappa:\cU[0,t]\rightarrow\R$ is identically zero, and so is not strictly convex. Consequently, the strict convexity assertion {\em (ii)} of Lemma \ref{lem:convex-cost-J} fails in that case, and uniqueness of the optimal controls in \er{eq:value-W}, \er{eq:value-W-M} cannot be established via Theorem \ref{thm:existence}.
\end{remark}

\begin{remark}
Given any non-decreasing sequence $\{M_k\}_{k\in\N}\subset\R_{\ge -\phi(0)}$ such that $\lim_{k\rightarrow\infty} M_k = \infty$, it may be shown that there exists a subsequence $\{M_j\}_{j\in\N}\subset\{M_k\}_{k\in\N}$ such that the corresponding optimal controls $\{u^{M_j*}\}_{j\in\N}\subset\cU[0,t]$ and $u^*\in\cU[0,t]$ of \er{eq:existence} satisfy
$u^{M_j *} \weakly u^*$ in $\cU[0,t]$, as $j\rightarrow\infty$. The details follow a similar argument to the proof of Theorem \ref{thm:existence} above, and are omitted.
\end{remark}


\subsection{Constraint satisfaction}
With existence of the optimal controls in \er{eq:value-W}, \er{eq:value-W-M} guaranteed by Theorem \ref{thm:existence}, the corresponding state trajectories can be examined to determine their compliance with the intended state constraint. To this end,
given $t\in\R_{>0}$, $x\in\R^n$, $\eps\in\R_{>0}$, define the sets of $\eps$-optimal inputs in the definitions \er{eq:value-W}, \er{eq:value-W-M} of $\ol{W}_t(x)$, $\ol{W}_t^M(x)$, $M\in\R_{\ge -\phi(0)}$, respectively by
\begin{align}
	& \cU_x^\eps[0,t]
	\doteq \left\{ u\in\cU[0,t] \, \biggl| \, \ol{W}_t(x) + \eps > \bar J_t(x,u) \right\},
	\label{eq:near-optimal-inputs}\\
	& \cU_{x}^{M,\eps}[0,t]
	\doteq \left\{ u\in\cU[0,t] \, \biggl| \, \ol{W}_t^M(x) + \eps > \bar J_t^M(x,u) \right\}.
	\label{eq:near-optimal-inputs-M}
\end{align}
Define the sets of times for which the desired state constraint is violated, as a function of the initial state and control, via a map $\Delta_t:\R^n\times\cU\rightarrow\cup_{I\subset[0,t]} I$, where
\begin{align}
	\Delta_{t}(x,u)
	\doteq \bigcup_{r\in[0,t], s\in[r,t]} \left\{ [r,s] \left| \ba{c}
			|[\chi(x,u)]_\sigma|\ge b \ \forall \ \sigma\in[r,s]
	\ea \right.\!\right\}\!
	\label{eq:time-intervals}
\end{align}
for all $x\in\R^n$, $u\in\cU[0,t]$, in which $\chi$ is as per \er{eq:dynamics}.

\newcommand{\betat}	{{\Xi_t}}
\newcommand{\barbetat}	{{{\ol{\Xi}_t}}}

\begin{theorem}
\label{thm:excursions}
The following properties concerning the map $\Delta_t$ of \er{eq:time-intervals} hold for any $t\in\R_{>0}$:
\begin{enumerate}[\it (i)] \itsep
\item There exist constants $M_1\in\R_{>-\phi(0)}$ and $\eta_t, \, \lambda_t,\, \betat\in\R_{>0}$ 
and non-increasing $\beta:\R_{>M_1}\rightarrow\R_{>0}$ satisfying $\lim_{M\rightarrow\infty} \beta(M) = 0$, such that for any $M\in\R_{>M_1}$, $\eps\in\R_{>0}$,
\begin{align}
	\sup_{u\in\cU_x^{M,\eps}[0,t]} \mu( \Delta_t(x,u) )
	& \le
	\beta(M) \left[\eta_t \left( \ol{W}_t^M(x) + \eps \right)
		+ \lambda_t  + \betat\, |x|^2 	
	\right]
	\label{eq:set-measure-bound}
\end{align}
for all $x\in\R^n$, in which $\mu$ denotes the Lebesgue measure, and $\cU_x^{M,\eps}$ is as per \er{eq:near-optimal-inputs-M};

\item Given any $x\in\dom\ol{W}_t$, and any $\eps\in\R_{>0}$, 
\begin{align}
	& \lim_{M\rightarrow\infty} \sup_{u\in\cU_x^{M,\eps}[0,t]} \mu( \Delta_t(x,u) ) = 0 =  \!\!\! \sup_{u\in\cU_x^\eps[0,t]} \mu(\Delta_t(x,u) )\,,
	\label{eq:set-measure-zero}
\end{align}
in which $\cU_x^\eps[0,t]$ is as per \er{eq:near-optimal-inputs}; and

\item Given any $x\in\dom\ol{W}_t$, and any strictly increasing sequence $\{M_k\}_{k\in\N}\subset\R_{>-\phi(0)}$, there exist a unique $u^*\in\cU[0,t]$ and a unique sequence $\{u^{M_k*}\}_{k\in\N}\subset\cU[0,t]$, specified by \er{eq:existence}, such that
\begin{align}
	& \lim_{k\rightarrow\infty} \mu( \Delta_t(x,u^{M_k*}) ) = 0 = \mu( \Delta_t(x,u^*) )\,.
	\nn
\end{align}
\end{enumerate}
\end{theorem}

\begin{proof}
Fix any $t\in\R_{>0}$. Select $M_1\in\R_{\ge -\phi(0)}$, $c\in\R$ as per Lemma \ref{lem:approx-convex-duality} {\em (iv)}. Fix any $M\in\R_{>M_1}$. By definition of $M$, $M_1$, and $c$, Lemma \ref{lem:approx-convex-duality} {\em (iii)} and {\em (iv)} imply that
\begin{align}
	c & < \inf_{M\ge M_1} \inf_{\rho\in\R} \Phi^{M}(\rho) \le \Phi^{M_1}(b^2) \le \Phi^M(b^2)\,.
	\label{eq:Phi-M-boundary}
\end{align}
Motivated by \er{eq:Phi-M-boundary}, define $\beta:\R_{>M_1}\rightarrow\R_{>0}$ by
\begin{align}
	\beta(M)
	& \doteq \frac{2}{\Phi^M(b^2) - c}
	\label{eq:mu-M}
\end{align}
for all $M>M_1$, and note that it is non-increasing by Lemma \ref{lem:approx-convex-duality} {\em (iii)}. Furthermore, \er{eq:Phi-M-sup-of-quadratics} and Lemma \ref{lem:gamma-properties} 
imply that $\lim_{M\rightarrow\infty} \Phi^M(b^2) \ge \lim_{M\rightarrow\infty} \{ a^{-1}(M)\, b^2 - M \} = \lim_{M\rightarrow\infty} \gamma_{b^2}(M) = \infty$, where $\gamma_{b^2}(M)$ is as per \er{eq:gamma}. Hence, by inspection of \er{eq:mu-M}, $\lim_{M\rightarrow\infty} \beta(M) = 0$.

{\em (i)} 
Fix any $x\in\R^n$, $\eps\in\R_{>0}$, and $u\in\cU_x^{M,\eps}[0,t]$, and denote the corresponding near-optimal trajectory by $\xi \doteq \chi(x,u)$ as per \er{eq:dynamics}. Applying \er{eq:near-optimal-inputs-M} and Corollary \ref{cor:Phi-M-bound}, note that
$	\ol{W}_t^M(x) + \eps 
	> \bar J_t^M(x,u)
	\ge \int_0^t \ts{\frac{K}{2}}\, |\xi_s|^2 + \demi\, \Phi^M(|\xi_s|^2) \, ds + \ts{\frac{\kappa}{2}}\, \|u\|_{\cU[0,t]}^2
	\ge \ts{\frac{\phi(0)}{2}}\, t + \ts{\frac{\kappa}{2}}\, \|u\|_{\cU[0,t]}^2$, 
so that
$
	\|u\|_{\cU[0,t]}^2 
	\le \ts{\frac{2}{\kappa}} \, [ \ol{W}_t^M(x) + \eps - \ts{\frac{\phi(0)}{2}}\, t ]
$.
Recalling \er{eq:dynamics}, there exist $\barbetat,\bar\gamma_t\in\R_{>0}$ such that $\|\xi\|_{\Ltwo([0;t];\R^n)}^2 \le \barbetat\, |x|^2 + \bar\gamma_t\, \|u\|_{\cU[0,t]}^2$, so that
\begin{align}
	& \|\xi\|_{\Ltwo([0;t];\R^n)}^2 \le \barbetat\, |x|^2 + \bar\gamma_t\, (\ts{\frac{2}{\kappa}}) \, [ \ol{W}_t^M(x) + \eps - \ts{\frac{\phi(0)}{2}}\, t ]\,.
	\nn
\end{align}
Consequently, returning to the definition \er{eq:near-optimal-inputs-M} of near optimality, and applying \er{eq:time-intervals}, Corollary \ref{cor:Phi-M-bound}, and the bound $K+\phi'(0)\ge 0$ adopted in \er{eq:phi-properties-1-2} {\em (ii)},
\begin{align}
	& \ol{W}_t^M(x) + \eps 
	> \bar J_t^M(x,u)
	\ge \int_0^t \ts{\frac{K}{2}}\, |\xi_s|^2 + \demi\, \Phi^M(|\xi_s|^2) \, ds 
	\nn\\
	& = \int_{[0,t]\setminus\Delta_t(x,u)} \ts{\frac{K}{2}}\, |\xi_s|^2 + \demi\, \Phi^M(|\xi_s|^2) \, ds
	+ \int_{\Delta_t(x,u)} \ts{\frac{K}{2}}\, |\xi_s|^2 + \demi\, \Phi^M(|\xi_s|^2) \, ds 
	\nn\\
	& \ge \int_{[0,t]\setminus\Delta_t(x,u)} -\ts{\frac{\phi'(0)}{2}}\, ds 
	+ \int_{\Delta_t(x,u)} - \ts{\frac{\phi'(0)}{2}}\, |\xi_s|^2 + \ts{\frac{c}{2}}\, ds
	+ \int_{\Delta_t(x,u)} \ts{\frac{K + \phi'(0)}{2}}\, |\xi_s|^2 + \demi\, [\Phi^M(|\xi_s|^2) - c] \, ds
	\nn\\
	& 
	\ge -\ts{\frac{|\phi'(0)|+|c|}{2}} \, t - \ts{\frac{|\phi'(0)|}{2}}\, \|\xi\|_{\Ltwo([0;t];\R^n)}^2
	+ \int_{\Delta_t(x,u)} \demi\, [\Phi^M(b^2) - c] \, ds 
	\nn\\
	& \ge -\ts{\frac{|\phi'(0)|+|c|}{2}} \, t 
	- \ts{\frac{|\phi'(0)|}{2}}\, \left[  \barbetat\, |x|^2 + \bar\gamma_t\, (\ts{\frac{2}{\kappa}}) \, [ \ol{W}_t^M(x) + \eps - \ts{\frac{\phi'(0)}{2}}\, t ] \right]
	+ \demi\, [\Phi^M(b^2) - c]\, \mu( \Delta_t(x,u) )\,.
	\label{eq:pre-excursion-bound}
\end{align}
That is, with $\beta$ as per \er{eq:mu-M},
\begin{align}
	& \mu( \Delta_t(x,u) )
	\le
	\beta(M) \left( \ol{W}_t^M(x) + \eps 
		+\ts{\frac{|\phi'(0)|+|c|}{2}} \, t 
		+ \ts{\frac{|\phi'(0)|}{2}}\, \left[  \barbetat\, |x|^2 + \bar\gamma_t\, (\ts{\frac{2}{\kappa}}) \, [ \ol{W}_t^M(x) + \eps - \ts{\frac{\phi(0)}{2}}\, t ] \right]	
	\right)
	\nn\\
	& = 	\beta(M) \left( [1 + \ts{\frac{\bar\gamma_t\, |\phi'(0)|}{\kappa}}] ( \ol{W}_t^M(x) + \eps )
		+ \demi\, [|\phi'(0)|+|c| - \ts{\frac{\bar\gamma_t}{\kappa}}\, |\phi'(0)|\, \phi(0) ] \, t  + \demi\, \barbetat\, |\phi'(0)|\, |x|^2 	\right),
	\nn
\end{align}
from which \er{eq:set-measure-bound} immediately follows by selecting $\eta_t\doteq 1 + \ts{\frac{\bar\gamma_t\, |\phi'(0)|}{\kappa}}$, $\lambda_t\doteq \demi\, [ |\phi'(0)|+|c| - \ts{\frac{\bar\gamma_t}{\kappa}}\, |\phi'(0)|\, \phi(0) ] \, t$, and $\betat\doteq \demi\, \barbetat\, |\phi'(0)|$.

{\em (ii)} Fix any $x\in\dom\ol{W}_t$, $\eps\in\R_{>0}$. The left-hand equality of \er{eq:set-measure-zero} holds follows by \er{eq:Phi-M-explicit}, \er{eq:order-value}, and assertion {\em (i)}, i.e. \er{eq:set-measure-bound}. In particular, 
\begin{align}
	\lim_{M\rightarrow\infty} \sup_{u\in\cU_x^{M,\eps}[0,t]} \mu( \Delta_t(x,u) )
	& \le \lim_{M\rightarrow\infty} \beta(M) \left[\eta \left( \ol{W}_t^M(x) + \eps \right)
		+ \lambda_t  + \betat\, |x|^2 	
	\right] = 0.
	\nn
\end{align}

It remains to show that the right-hand equality in \er{eq:set-measure-zero} holds. Fix any $u\in\cU_x^\eps[0,t]$. Suppose there exists $\delta\in\R_{>0}$ such that $\mu( \Delta_t(x,u) )\ge \delta > 0$. An analogous calculation to \er{eq:pre-excursion-bound}, with $\ol{W}_t^M$ and $\Phi^M$ replaced with $\ol{W}_t$ and $\sup_{M\ge -\phi(0)} \Phi^M$, yields
\begin{align}
	& \ol{W}_t(x) + \eps + \ts{\frac{|\phi(0)|+|c|}{2}} \, t +
	\ts{\frac{|\phi(0)|}{2}}\, \left[  \barbetat\, |x|^2 + \bar\gamma_t\, (\ts{\frac{2}{\kappa}}) \, [ \ol{W}_t(x) + \eps - \ts{\frac{\phi(0)}{2}}\, t ] \right]
	\nn\\
	& \qquad > \demi\, \sup_{M\ge -\phi(0)} [ \Phi^M(b^2) - c ]\, \delta = \demi\, \sup_{M\ge -\phi(0)} [ \gamma_{b^2}(M) - c ]\, \delta = \infty\,,
	\nn
\end{align}
in which the equalities follow as $\delta\in\R_{>0}$, and by \er{eq:Phi-M-sup-of-quadratics} and Lemma \ref{lem:gamma-properties}. 
Hence, $\ol{W}_t(x) = \infty$, which contradicts the definition of $x\in\dom\ol{W}_t$. Consequently, no such $\delta\in\R_{>0}$ exists, so that $\mu( \Delta_t(x,u) ) = 0$. As $u\in\cU_x^\eps[0,t]$ is arbitrary, the right-hand equality in \er{eq:set-measure-zero} follows as required.

{\em (iii)} Immediate by assertion {\em (ii)} and Theorem \ref{thm:existence}.
\end{proof}

\begin{remark}
Theorem \ref{thm:excursions} indicates that the regulator problem defined by $\ol{W}_t$ of \er{eq:value-W} implements the required state constraint for almost every time for those initial states $x\in\R^n$ for which $\ol{W}_t(x)<\infty$, and that the approximating regulator problem defined by $\ol{W}_t^M$ of \er{eq:value-W-M} implements the same constraint in the limit as $M\rightarrow\infty$. 
\end{remark}


\section{Equivalent unconstrained game representations} 
\label{sec:game}
The sup-of-quadratics representation \er{eq:Phi-sup-of-quadratics} for the convex barrier function $\Phi$ in \er{eq:barrier} is used to demonstrate equivalence of the value function defining the state constrained regulator problem \er{eq:value-W} with the upper  value of an unconstrained two player game. Similarly, the approximate sup-of-quadratics representation \er{eq:Phi-M-sup-of-quadratics} is used to demonstrate an equivalence between the value function defining the approximate regulator problem \er{eq:value-W-M} with the corresponding upper value of an approximate two player game. It is further demonstrated that this approximate game has equivalent upper and lower values, which in turn is used to demonstrate the corresponding equivalence for the exact game, via the convergence results of Theorem \ref{thm:monotone}. The lower value is subsequently exploited to examine solutions of the state constrained regulator problem \er{eq:value-W} via families of DREs.


\subsection{Exact unconstrained game and its upper value}
Given a horizon $t\in\R_{\ge 0}$, define a function space by
$
	\cA[0,t]
	\doteq \left\{ \alpha:[0,t]\rightarrow\R_{\ge -\phi(0)} \, |\, \text{measurable} \right\}
$.
Motivated by \er{eq:cost-J-bar}, \er{eq:cost-I-bar}, \er{eq:Phi-sup-of-quadratics}, define the upper value $W_t:\R^n\rightarrow\ol{\R}^+$ of a two player unconstrained linear quadratic game by
\begin{align}
	W_t(x)
	& \doteq \inf_{u\in\cU[0,t]} \sup_{\alpha\in\cA[0,t]} J_t(x,u,\alpha)
	\label{eq:game-W}
\end{align}
for all $x\in\R^n$, in which $J_t$ is a cost function defined with respect to a new integrated running cost function $I_t$ motivated by \er{eq:cost-I-bar}, and the existing integrated running cost $I_t^\kappa$ of \er{eq:cost-I-bar} and terminal cost $\Psi$ of \er{eq:cost-Psi}. In particular, define
$J_t, I_t:\R^n\times\cU[0,t]\times\cA[0,t]\rightarrow\R$ and $\nu:\R^n\times\R_{\ge -\phi(0)}\rightarrow\R$ by
\begin{align}
	J_t(x,u,\alpha)
	& \doteq I_t(x,u,\alpha) + I_t^\kappa(u) + \Psi(\xi_t),
	\label{eq:cost-J-hat}
	\\
	I_t(x,u,\alpha)
	& \doteq \int_0^t \nu(\xi_s,\alpha_s)\, ds,
	\quad \xi\doteq\chi(x,u),
	\label{eq:cost-I-hat}
	\\
	\nu(x,\hat\alpha)
	& \doteq
	\ts{\frac{K}{2}}\, |x|^2 + \demi\, [ a^{-1}(\hat\alpha)\, |x|^2 - \hat\alpha ],
	\label{eq:running-cost-nu}
\end{align}
for all $x\in\R^n$, $u\in\cU[0,t]$, $\alpha\in\cA[0,t]$, $\hat\alpha\in\R_{\ge -\phi(0)}$.

The value functions \er{eq:value-W} and \er{eq:game-W} defining the exact regulator problem and the exact unconstrained game are in fact equivalent, as stated in the following theorem. 
\begin{theorem}
\label{thm:game}
Given $t\in\R_{\ge 0}$, the value functions $\ol{W}_t$, $W_t$ of \er{eq:value-W}, \er{eq:game-W} are equivalent, with $\ol{W}_t(x) = W_t(x)$ for all $x\in\R^n$.
\end{theorem}

The proof of Theorem \ref{thm:game} follows as a consequence of a measurable selection lemma.

\begin{lemma}
\label{lem:measurable-selection}
Given $t\in\R_{\ge 0}$, $x\in\R^n$, $u\in\cU[0,t]$, $\xi\doteq\chi(x,u)\in C([0,t];\R^n)$, the following hold:
\begin{enumerate}[\it (i)] \itsep
\item The cost functions $\bar I_t$, $I_t$ of \er{eq:cost-I-bar}, \er{eq:cost-I-hat} associated with the exact regulator problem \er{eq:value-W} and game upper value \er{eq:game-W} satisfy
\begin{align}
	& \hspace{-3mm} \bar I_t(x,u)
	= \int_0^t \sup_{\hat\alpha\ge -\phi(0)} \nu(\xi_s,\hat\alpha) \, ds
	= \sup_{\alpha\in\cA[0,t]} I_t(x,u,\alpha),
	\label{eq:measurable-selection}
\end{align}
in which $\nu$ is as per \er{eq:running-cost-nu};

\item If $\mu(\Delta_t(x,u))=0$, see \er{eq:time-intervals}, then $\alpha^*\in\cA[0,t]$ given for any $M\in\R_{\ge -\phi(0)}$ by
\begin{align}
	\alpha_s^{*}
	= \hat\alpha^{\Delta*}(|\xi_s|^2)
	& \doteq \left\{ \ba{rl}
		a\circ\phi'(|\xi_s|^2),
		& s\in[0,t]\setminus\Delta_t(x,u),
		\\
		M,
		& s\in\Delta_t(x,u),
	\ea \right.
	\label{eq:alpha-star-cost-J-hat}
\end{align}
satisfies
\begin{align}
	\bar I_t(x,u)
	& = I_t(x,u,\alpha^*)\, ;
	\label{eq:alpha-star-maximizes}
\end{align}

\item If $x\in\dom\ol{W}_t$ and $u\in\cU_x^\eps[0,t]$, $\eps\in\R_{>0}$, see \er{eq:near-optimal-inputs}, then \er{eq:alpha-star-maximizes} holds with $\alpha^*\in\cA[0,t]$ as per \er{eq:alpha-star-cost-J-hat}, for arbitrary $M\in\R_{\ge -\phi(0)}$.
\end{enumerate}
\end{lemma}
\begin{proof}
Fix $t\in\R_{\ge 0}$, $x\in\R^n$, $u\in\cU[0,t]$, and $\xi\doteq\chi(x,u)\in C([0,t];\R^n)$.
{\em (i)}
The left-hand equality in \er{eq:measurable-selection} is immediate by \er{eq:cost-I-bar}, \er{eq:cost-I-hat}, \er{eq:running-cost-nu}, and Theorem \ref{thm:Phi-sup-of-quadratics}, in particular \er{eq:Phi-sup-of-quadratics}. For the right-hand equality, first fix any $\alpha\in\cA[0,t]$, and note that it is pointwise suboptimal in the supremum over $\hat\alpha\ge -\phi(0)$. That is,
$\int_0^t \sup_{\hat\alpha\ge -\phi(0)} \nu(\xi_s, \hat\alpha) \, ds \ge \int_0^t \nu(\xi_s, \alpha_s) \, ds$ for all $\alpha\in\cA[0,t]$. Hence, 
\begin{align}
	\int_0^t \sup_{\hat\alpha\ge -\phi(0)} \nu(\xi_s, \hat\alpha) \, ds
	& \ge \sup_{\alpha\in\cA[0,t]} \int_0^t \nu(\xi_s, \alpha_s) \, ds
	= \sup_{\alpha\in\cA[0,t]} I_t(x,u,\alpha).
	\label{eq:selection-ge}
\end{align}
In order to prove the opposite inequality, fix $\delta\in\R_{>0}$, and suppose that $u\in\cU[0,t]$ is such that $|\Delta_t(x,u)|\ge \delta>0$, see \er{eq:time-intervals}. Given any $\alpha^-,\alpha^+\in\R_{\ge -\phi(0)}$, define piecewise constant $\breve\alpha\in\cA[0,t]$ by 
\begin{align}
	\breve\alpha_s
	& = \left\{ \ba{rl}
		\alpha^-, 
		& s\in[0,t]\setminus\Delta_t(x,u),
		\\
		\alpha^+,
		& s\in\Delta_t(x,u),
	\ea \right.
	\nn
\end{align}
for all $s\in[0,t]$. Note that $\breve\alpha$ is suboptimal insofar as
\begin{align}
	& \sup_{\alpha\in\cA[0,t]} I_t(x,u,\alpha)
	\ge I_t(x,u,\breve\alpha)
	= I^- + 
	\int_{\Delta_t(x,u)} \nu(\xi_s,\alpha^+)\, ds,
	\quad
	I^- \doteq
	\int_{[0,t]\setminus\Delta_t(x,u)} \nu(\xi_s,\alpha^-)\, ds\,.
	\nn
\end{align}
As this is true for any $\alpha^+\in\R_{\ge -\phi(0)}$, it follows immediately that
\begin{align}
	&  \sup_{\alpha\in\cA[0,t]} I_t(x,u,\alpha)
	\ge I^- + \sup_{\alpha^+\ge -\phi(0)} \int_{\Delta_t(x,u)} \nu(\xi_s,\alpha^+)\, ds.
	\nn
\end{align}
Lemma \ref{lem:gamma-properties} 
implies that $\nu(\xi_s,\cdot) = \gamma_{|\xi_s|^2}(\cdot)$ is strictly increasing for any $s\in\Delta_t(x,u)$ fixed, as $|\xi_s|^2\ge b^2$ by \er{eq:time-intervals}, with $\lim_{\alpha^+\rightarrow\infty} \nu(\xi_s,\alpha^+) = \ts{\frac{K}{2}}\, |\xi_s|^2 + \lim_{\alpha^+\rightarrow\infty} \gamma_{|\xi_s|^2}(\alpha^+) = \infty$. Consequently, there exists an $M_0\in\R_{\ge -\phi(0)}$ such that
$\nu(\xi_s,\alpha^+) \ge a^{-1}(\alpha^+)\, b^2 - \alpha^+ > 0$ for all $\alpha^+>M_0$.
Hence, the monotone convergence theorem implies that
\begin{align}
	\sup_{\alpha\in\cA[0,t]} \! I_t(x,u,\alpha)
	& \ge I^- + \lim_{\alpha^+\rightarrow\infty} \int_{\Delta_t(x,u)} \nu(\xi_s,\alpha^+)\, ds 
	= I^- + \int_{\Delta_t(x,u)} \ts{\lim_{\alpha^+\rightarrow\infty}} \, \nu(\xi_s,\alpha^+)\, ds 
	= \infty.
	\nn
\end{align}
As the left-hand side here is the right-hand side of \er{eq:selection-ge}, it follows immediately that the opposite inequality to \er{eq:selection-ge} always holds where $|\Delta_t(x,u)|\ge \delta>0$.

Alternatively, suppose $u\in\cU[0,t]$ is such that $\mu( \Delta_t(x,u) )=0$, and let $\alpha^*\in\cA[0,t]$ be defined by \er{eq:alpha-star-cost-J-hat}. Recalling the left-hand equality of \er{eq:measurable-selection}, and the definition \er{eq:alpha-star} of $\hat\alpha^*(\cdot)$ in Theorem \ref{thm:Phi-sup-of-quadratics},
\begin{align}
	\bar I_t(x,u)
	& = \int_0^t \sup_{\hat\alpha\ge-\phi(0)} \nu(\xi_s, \hat\alpha) \, ds
	= \int_{[0,t]\setminus\Delta_t(x,u)}\!\!\!\!\!\nu(\xi_s,\hat\alpha_s^*(|\xi_s|^2))\, ds
	= \int_{[0,t]\setminus\Delta_t(x,u)}\!\!\!\!\!\nu(\xi_s,\alpha_s^*)\, ds
	\nn\\
	&  = \int_0^t \nu(\xi_s,\alpha_s^*)\, ds
	\le \sup_{\alpha\in\cA[0,t]} \int_0^t \nu(\xi_s,\alpha_s)\, ds = \sup_{\alpha\in\cA[0,t]} I_t(x,u,\alpha)\,.
	\label{eq:item-ii-proof}
\end{align}
Combining this inequality with \er{eq:selection-ge} yields \er{eq:measurable-selection}. 

{\em (ii)} Immediate by the fourth equality of \er{eq:item-ii-proof}.

{\em (iii)} Fix $x\in\dom\ol{W}_t$, $\eps\in\R_{>0}$, $u\in\cU_x^\eps[0,t]$. Theorem \ref{thm:excursions} {\em (ii)} implies that $\mu( \Delta(x,u) ) = 0$, so that assertion {\em (ii)} above applies.
\end{proof}

Theorem \ref{thm:game} then follows by Lemma \ref{lem:measurable-selection} {\em (i)} and comparison of \er{eq:value-W}--\er{eq:cost-I-bar} and \er{eq:game-W}--\er{eq:running-cost-nu}.


\subsection{Approximate game and its upper and lower values}
Given $M\in\R_{\ge -\phi(0)}$, $t\in\R_{\ge 0}$, define $\cA^M[0,t]\doteq C([0,t];[-\phi(0),M])$. Analogous to the exact game defined by \er{eq:game-W}, define the upper value $W_t^M:\R^n\rightarrow\R$ of an approximating two player unconstrained linear quadratic game by
\begin{align}
	W_t^M(x)
	& \doteq \inf_{u\in\cU[0,t]} \sup_{\alpha\in\cA^M[0,t]} J_t(x,u,\alpha)
	\label{eq:game-W-M}
\end{align}
for all $x\in\R^n$, where cost $J_t$ is as per \er{eq:cost-J-hat}. As in the exact case, the value function \er{eq:value-W-M} of the approximating regulator problem and the upper value \er{eq:game-W-M} of the approximating game are equivalent.
\begin{theorem}
\label{thm:game-W-M}
Given $t\in\R_{\ge 0}$, $M\in\R_{\ge -\phi(0)}$, the value functions $\ol{W}_t^M$, $W_t^M$ of \er{eq:value-W-M}, \er{eq:game-W-M} are equivalent, with $\ol{W}_t^M(x) = W_t^M(x)$ for all $x\in\R^n$.
\end{theorem}

The proof of Theorem \ref{thm:game-W-M} follows as a consequence of a corresponding measurable selection lemma. 

\begin{lemma}
\label{lem:measurable-selection-M}
Given any $t\in\R_{\ge 0}$, $M\in\R_{\ge -\phi(0)}$, $x\in\R^n$, $u\in\cU[0,t]$, and $\xi\doteq\chi(x,u)\in C([0,t];\R^n)$,
the cost functions $\bar I_t^M$, $I_t$ and $J_t^M$, $J_t$ of \er{eq:cost-I-J-M}, \er{eq:cost-I-hat} and \er{eq:cost-J-M}, \er{eq:cost-J-hat} satisfy
\begin{gather}
	\bar I_t^M(x,u)
	= \int_0^t \sup_{\hat\alpha\in[-\phi(0),M]} \nu(\xi_s,\hat\alpha) \, ds 
	= \sup_{\alpha\in\cA^M[0,t]} I_t(x,u,\alpha) = I_t(x,u,\alpha^{M*})\,,
	\label{eq:measurable-selection-M}
	\\
	\bar J_t^M(x,u)
	= \sup_{\alpha\in\cA^M[0,t]} J_t(x,u,\alpha) = J_t(x,u,\alpha^{M*})\,,
	\label{eq:cost-I-J-hat-sup}
\end{gather}
in which $\alpha^{M*}\in\cA^M[0,t]$ is defined via \er{eq:alpha-M-star} by, and satisfies,
\begin{gather}
	\alpha_s^{M*}
	\doteq \hat\alpha^{M*}(|\xi_s|^2)\,,
	\quad s\in[0,t]\,,
	\label{eq:alpha-star-in-cost-J-hat-M}
	\\
	\alpha^{M*}\in\argmax_{\alpha\in\cA^M[0,t]} I_t(x,u,\alpha)\equiv\argmax_{\alpha\in\cA^M[0,t]} J_t(x,u,\alpha)\,.
	\label{eq:alpha-star-M-maximizes}
\end{gather}
\end{lemma}

The proof of Lemma \ref{lem:measurable-selection-M} follows analogously to that of Lemma \ref{lem:measurable-selection} and is omitted.
%
\if{false}

\begin{proof}
Fix any $t\in\R_{\ge 0}$, $M\in\R_{\ge -\phi(0)}$, $x\in\R^n$, $u\in\cU[0,t]$, and define $\xi\doteq\chi(x,u)\in C([0,t];\R^n)$.
Define $\alpha^{M*}$ as per \er{eq:alpha-star-in-cost-J-hat-M}, and note in particular that $\alpha^{M*}\in\cA^M[0,t]$ by Lemma \ref{lem:a-properties} and \er{eq:phi-properties-1-2}, \er{eq:Phi-M-from-A-and-rho-M}.

{\em [\er{eq:measurable-selection-M} and the left-hand argmax in \er{eq:alpha-star-M-maximizes}]}
The first equality in \er{eq:measurable-selection-M} is immediate by Theorem \ref{thm:Phi-M-sup-of-quadratics} {\em (i)}, i.e. \er{eq:Phi-M-sup-of-quadratics}. For the remaining equalities, any $\alpha\in\cA^M[0,t]$ is pointwise suboptimal in the supremum over $\hat\alpha\in[-\phi(0),M]$, so that
\begin{align}
	\bar I_t^M(x,u) = \int_0^t \sup_{\hat\alpha\in[-\phi(0),M]} \nu(\xi_s,\hat\alpha) \, ds
	& \ge \int_0^t \nu(\xi_s,\alpha_s)\, ds
	\nn
\end{align}
for all $\alpha\in\cA^M[0,t]$. Hence, 
\begin{align}
	\bar I_t^M(x,u) = \int_0^t \sup_{\hat\alpha\in[-\phi(0),M]} \nu(\xi_s,\hat\alpha) \, ds
	& \ge \sup_{\alpha\in\cA^M[0,t]} \int_0^t \nu(\xi_s,\alpha_s) \, ds 
	= \sup_{\alpha\in\cA^M[0,t]} I_t(x,u,\alpha)\,.
	\label{eq:measure-M-ineq-1}
\end{align}
In order to prove the opposite inequality, recall by \er{eq:alpha-star-in-cost-J-hat-M} and Theorem \ref{thm:Phi-M-sup-of-quadratics}, i.e. \er{eq:alpha-M-star}, that $\alpha_s^{M*}$ is the pointwise maximizer of $\nu(\xi_s,\cdot)$. That is,
\begin{align}
	\bar I_t^M(x,u)
	& =
	\int_0^t \sup_{\hat\alpha\in[-\phi(0),M]} \nu(\xi_s,\hat\alpha) \, ds
	= \int_0^t \nu(\xi_s,\alpha_s^{M*}) \, ds = I_t(x,u,\alpha^{M*})
	\nn\\
	& 
	\le \sup_{\alpha\in\cA^M[0,t]} \int_0^t \nu(\xi_s,\alpha_s) \, ds
	=  \sup_{\alpha\in\cA^M[0,t]} I_t(x,u,\alpha).
	\label{eq:measure-M-ineq-2}
\end{align}
Hence, combining inequalities \er{eq:measure-M-ineq-1} and \er{eq:measure-M-ineq-2} yields \er{eq:measurable-selection-M}, and the left-hand argmax in \er{eq:alpha-star-M-maximizes}.

{\em [\er{eq:cost-I-J-hat-sup} and the right-hand argmax in \er{eq:alpha-star-M-maximizes}]} Immediate by definitions \er{eq:cost-J-M}, \er{eq:cost-I-J-M}, \er{eq:cost-J-hat}, \er{eq:cost-I-hat} of $\bar J_t^M$, $\bar I_t^M$, $J_t$, $I_t$, \er{eq:measurable-selection-M}, \er{eq:alpha-star-in-cost-J-hat-M}, and the left-hand argmax in \er{eq:alpha-star-M-maximizes}. (Note in particular that the dependence of $J_t(\cdot,\cdot,\alpha)$ on $\alpha\in\cA^M[0,t]$ comes only through $I_t(\cdot,\cdot,\alpha)$.)
\end{proof}

\fi

Theorem \ref{thm:game-W-M} subsequently follows by Lemma \ref{lem:measurable-selection-M} and comparison of \er{eq:value-W-M}--\er{eq:cost-I-J-M} and \er{eq:cost-J-hat}--\er{eq:running-cost-nu}, \er{eq:game-W-M}.

With a view to addressing computation, the remaining objective is to demonstrate equivalence of the upper and lower values for the game \er{eq:game-W-M}. To this end, a number of useful properties of the cost function $J_t$ of \er{eq:cost-J-hat} are summarised via the following two lemmas.

\begin{lemma}
\label{lem:convex-J-hat}
Given any $t\in\R_{>0}$, $x\in\R^n$, $M\in\R_{\ge -\phi(0)}$, $\alpha\in\cA^M[0,t]$, the cost function $J_t(x,\cdot,\alpha):\cU[0,t]\rightarrow\R$ defined by \er{eq:cost-J-hat} is {\Frechet} differentiable, strictly convex, and coercive.
\end{lemma}

\begin{proof}
The proof follows analogous arguments to that of Lemma \ref{lem:convex-cost-J}, and the details are omitted.
\end{proof}

\if{false}

\rev{??}

\begin{proof}
Fix $t\in\R_{>0}$, $x\in\R^n$, $M\in\R_{\ge -\phi(0)}$, $\alpha\in\cA^M[0,t]$.

{\em [Differentiability]} By inspection of \er{eq:cost-J-hat}, it is sufficient to show that the maps $I_t(x,\cdot,\alpha)$, $I_t^\kappa$, and $\Psi([\chi(x,\cdot)]_t)$ are ({\Frechet}) differentiable on $\cU[0,t]$. Fix $u\in\cU[0,t]$. It is immediate by \er{eq:cost-I-bar} that $I_t^\kappa$ is differentiable on $\cU[0,t]$, with derivative $\langle\kappa\, u, \cdot \rangle_{\cU[0,t]}$. Similarly, by Lemma \ref{lem:chi} {\em (ii)}, $[\chi(x,\cdot)]_t$ is differentiable on $\cU[0,t]$, with derivative $D_u [\chi(x,u)]_t\, h = [\op{A}\, h]_t$ for all $h\in\cU[0,t]$. As $\Psi$ of \er{eq:cost-Psi} is differentiable on $\R^n$, the chain rule subsequently implies that $\Psi([\chi(x,\cdot)]_t)$ is differentiable on $\cU[0,t]$, with derivative $h\mapsto x'\, P_t\, [\op{A}\, h]_t\in\bo(\cU[0,t];\R)$.

It remains to be shown that $I_t(x,\cdot,\alpha)$ is differentiable on $\cU[0,t]$. To this end, motivated by \er{eq:cost-I-hat}, define $\iota^\alpha:C([0,t];\R^n)\rightarrow\R$ and $\op{I}^\alpha_\xi:C([0,t];\R^n)\rightarrow\R$ by
\begin{align}
	\iota^\alpha(\xi)
	& \doteq \int_0^t \demi\, [K + a^{-1}(\alpha_s)] \, |\xi_s|^2 - \ts{\frac{\alpha_s}{2}}\, ds\,,
	\quad
	\op{I}^\alpha_\xi\, h
	\doteq \int_0^t  [K + a^{-1}(\alpha_s)] \, \langle \xi_s, h_s \rangle\, ds\,,
	\label{eq:iota-frechet}
\end{align}
for all $\xi,h\in C([0,t];\R^n)$. With a view to applying the chain rule, observe in particular that $I_t(x,\cdot,\alpha) = \iota^\alpha\circ\chi(x,\cdot)$, with $\chi$ as per \er{eq:dynamics}. Attending first to differentiability of $\iota^\alpha$, note by Lemma \ref{lem:a-properties}, \er{eq:phi-properties-1-2}, and the definition of $\alpha\in\cA^M[0,t]$, that 
$
	K + a^{-1}(M) \ge K + a^{-1}(\alpha_s) \ge K+ a^{-1}(-\phi(0)) = K + \phi'(0) \ge 0
$.
Consequently,
\begin{align}
	|\iota^\alpha(\xi)|
	& \le \left( [K + a^{-1}(M)]\, \|\xi\|_{C([0,t];\R^n)}^2 + \max(|\phi(0)|,|M|) \right) \ts{\frac{t}{2}}\,,
	\nn\\
	|\op{I}_\xi^\alpha(h)|
	& \le t\, [K + a^{-1}(M)]\, \|\xi\|_{C([0,t];\R^n)}\, \|h\|_{C([0,t];\R^n)} \,,
	\nn
\end{align}
for all $\xi,h\in C([0,t];\R^n)$. Note in particular that $\op{I}_\xi^\alpha\in\bo(C([0,t];\R^n);\R)$ for all $\xi\in C([0,t];\R^n)$.
Fix $\xi\in C([0,t];\R^n)$. Recalling \er{eq:iota-frechet},
\begin{align}
	|\iota^\alpha(\xi+h) - \iota^\alpha(\xi) - \op{I}_\xi^\alpha\, h|
	& = \int_0^t \demi\, [K + a^{-1}(\alpha_s)] \, |h_s|^2\, ds
	\le \demi\, t\, [K + a^{-1}(M)]\, \|h\|_{C([0,t];\R^n)}^2
	\nn
\end{align}
for all $h\in C([0,t];\R^n)$. Hence, $\iota^\alpha$ is differentiable on $C([0,t];\R^n)$, with derivative $\op{I}_\xi^\alpha$ as per \er{eq:iota-frechet}. Meanwhile, Lemma \ref{lem:chi} {\em (ii)} implies that $\chi(x,\cdot)$ is differentiable on $\cU[0,t]$ with derivative as per \er{eq:derivatives}, i.e. $\op{A}\in\bo(\cU[0,t];C([0,t];\R^n))$. Consequently, $\iota^\alpha:C([0,t];\R^n)\rightarrow\R$ and $\chi(x,\cdot):\cU[0,t]\rightarrow C([0,t];\R^n)$ are differentiable, so that the chain rule implies that $ \iota^\alpha\circ\chi(x,\cdot)\equiv I_t(x,\cdot,\alpha):\cU[0,t]\rightarrow\R$ is differentiable, with derivative $\op{I}_{\chi(x,u)}^\alpha\,  \op{A}\in\bo(\cU[0,t];\R)$ at $u\in\cU[0,t]$. Consequently, $J_t(x,\cdot,\alpha)$ is ({\Frechet}) differentiable.

{\em [Strict convexity]}
Recall by \er{eq:gamma-fn} and \er{eq:cost-J-hat}--\er{eq:running-cost-nu} that
\begin{align}
	J_t(x,u,\alpha)
	& = I_t(x,u,\alpha) + I_t^\kappa(u) + \Psi(\xi_t),
	\quad
	I_t(x,u,\alpha) = \int_0^t \gamma_{x}^{s,\alpha_s}(u)\, ds,
	\nn
\end{align}
for all $u\in\cU[0,t]$, in which $\gamma_x^{s,\hat\alpha}\in C(\cU[0,t];\R)$ is convex for every $s\in[0,t]$, $\hat\alpha\in\R_{\ge -\phi(0)}$ by Lemma \ref{lem:convex-gamma}, including for $\hat\alpha \doteq\alpha_s\in[-\phi(0),M]$. Note in particular that $I_t(x,\cdot,\alpha)$ is convex on $\cU[0,t]$, as convexity is preserved under integration, and $I_t^\kappa(\cdot)$ and $\Psi([\chi(x,\cdot)]_t)$ are respectively strictly convex, and convex on $\cU[0,t]$ by inspection of \er{eq:cost-I-bar} and \er{eq:cost-Psi}. Hence, $J_t(x,\cdot,\alpha)$ is strictly convex on $\cU[0,t]$.

{\em [Coercivity]} Fix $u\in\cU[0,t]$. By inspection of \er{eq:cost-J-hat}, \er{eq:cost-I-hat}, \er{eq:running-cost-nu}, and noting that $K + a^{-1}(\alpha_s)\ge 0$ and $-\alpha_s \ge -M$ for all $s\in[0,t]$, as $\alpha\in\cA^M[0,t]$,
\begin{align}
	J_t(x,u,\alpha)
	& \ge 
	\int_0^t \demi\, [K + a^{-1}(\alpha_s) ]\, |\xi_s|^2 - \ts{\frac{\alpha_s}{2}}\, ds + \ts{\frac{\kappa}{2}}\, \|u\|_{\cU[0,t]}^2
	\ge -\ts{\frac{M}{2}}\, t  + \ts{\frac{\kappa}{2}}\, \|u\|_{\cU[0,t]}^2\,,
	\nn
\end{align}
which yields the required coercivity.
\end{proof}

\rev{??}

\fi

\begin{remark}
The strict convexity assertion in Lemma \ref{lem:convex-J-hat} requires $t\in\R_{>0}$, see also Remark \ref{rem:convex-t-0}.
\end{remark}

\begin{lemma}
\label{lem:u-star-is-optimal-in-cost-J-hat}
Given $t\in\R_{>0}$, $M\in\R_{\ge -\phi(0)}$, and $x\in\R^n$, let $u^{M*}\in\cU[0,t]$ be defined as per \er{eq:existence}, and let $\alpha^{M*}\doteq \hat\alpha^{M*}(|\chi(x,u^{M*})|^2)\in \cA^M[0,t]$ be defined via \er{eq:dynamics}, \er{eq:alpha-star-in-cost-J-hat-M}. Then, $u^{M*}$ and $\alpha^{M*}$ are unique, and together satisfy 
\begin{align} 
	&
	\argmin_{u\in\cU[0,t]} \bar J_t^M(x,u) = 
	u^{M*} 
	= \argmin_{u\in\cU[0,t]} J_t(x,\cdot,\alpha^{M*})\,.
	\nn
\end{align}
\end{lemma}
\begin{proof}
Fix $t\in\R_{>0}$, $M\in\R_{\ge -\phi(0)}$. Define $u^{M*}\in\cU[0,t]$ uniquely as per \er{eq:existence}, i.e. as per the left-hand equality in the lemma statement. Given this $u^{M*}$, define $\alpha^{M*}\in\cA^M[0,t]$ as per the lemma statement, and note by Lemma \ref{lem:measurable-selection-M} that $\alpha^{M*}$ is unique by definition. Recall by Lemma \ref{lem:convex-J-hat} that $J_t(x,\cdot,\alpha^{M*}):\cU[0,t]\rightarrow\R$ is {\Frechet} differentiable and strictly convex. Hence, in order to show that $J_t(x,\cdot,\alpha^{M*})$ is (uniquely) minimized at $u^{M*}$, it is sufficient to show that the directional derivative of $J_t(x,\cdot,\alpha^{M*})$ is nonnegative in all directions when evaluated at $u^{M*}$. The details follow.

Fix any $\tilde u\in\cU[0,t]$ with $\|\tilde u\|_{\cU[0,t]} = 1$. The {\Frechet} derivative $D_u J_t(x,u^{M*},\alpha^{M*})\in\bo(\cU[0,t];\R)$, and corresponding Riesz representation $\ggrad_u J_t(x,u^{M*},\alpha^{M*})\in\cU[0,t]$, at $u^{M*}\in\cU[0,t]$, satisfy
\begin{align}
	& D_u J_t(x,u^{M*}, \alpha^{M*}) (\tilde u) 
	= \langle \ggrad_u J_t(x,u^{M*},\alpha^{M*}), \, \tilde u \rangle_{\cU[0,t]}
	\nn\\
	& \qquad = \lim_{\delta\rightarrow 0^+} \left\{ \frac{J_t(x, u^{M*} + \delta\, \tilde u, \alpha^{M*}) - J_t(x,u^{M*}, \alpha^{M*})}{\delta} \right\}.
	\label{eq:Frechet-J}
\end{align}
Fix any $\eps\in\R_{>0}$ with $\eps^2<\min(1,\hat\rho(M))$, and $\hat\rho(M)$ as per \er{eq:Phi-M-from-A-and-rho-M}. Let $L_t\doteq \|\op{A}\|_{\bo(\cU[0,t];C([0,t];\R^n))} \in\R_{>0}$, with $\op{A}$ as per \er{eq:derivatives}, and $\bar\delta^\eps\doteq \eps / (4\, L_t) \in\R_{>0}$. Fix any $\delta\in(0,\bar\delta^\eps]$. Define 
\begin{gather}
	\tilde u^{M*} \doteq u^{M*} + \delta\, \tilde u\,, \quad 
	{\tilde \xi}^{M*}\doteq\chi(x,\tilde u^{M*})\,, \quad 
	\xi^{M*}\doteq\chi(x,u^{M*})\,,
	\nn\\
	\tilde\alpha_s^{M*} \doteq \hat\alpha^{M*}(|\tilde\xi_s^{M*}|^2)\,, \quad
	\alpha_s^{M*} \doteq \hat\alpha^{M*}(|\xi_s^{M*}|^2),
	\label{eq:u-alpha}
\end{gather}
for all $s\in[0,t]$, with $\chi$, $\hat\alpha^{M*}$ as per \er{eq:dynamics}, \er{eq:alpha-M-star}. Recalling \er{eq:dynamics} and Lemma \ref{lem:chi}, note in particular that
\begin{align}
	& \|\tilde\xi^{M*} - \xi^{M*}\|_{C([0,t];\R^n)}
	\le L_t\, \delta\, \|\tilde u\|_{\cU[0,t]} = L_t\, \delta \le L_t\, \bar\delta^\eps \le \ts{\frac{\eps}{4}}.
	\label{eq:closeness}
\end{align}
%
%
By \er{eq:cost-J-hat}, \er{eq:cost-I-hat}, \er{eq:running-cost-nu}, and Lemma \ref{lem:measurable-selection-M}, 
\begin{align}
	J_t(x,u^{M*},\alpha^{M*})
	& = \bar J_t(x,u^{M*}),
	\nn\\
	J_t(x,\tilde u^{M*},\alpha^{M*})
	& = \int_0^t \nu(\tilde\xi_s^{M*},\alpha_s^{M*}) \, ds +  I_t^\kappa(\tilde u^{M*}) + \Psi(\tilde\xi_t^{M*})
	\nn\\
	& = \int_0^t \nu(\tilde\xi_s^{M*},\tilde\alpha_s^{M*}) 
			+ [\nu(\tilde\xi_s^{M*}, \alpha_s^{M*}) - \nu(\tilde\xi_s^{M*}, \tilde\alpha_s^{M*})] \, ds 
			+  I_t^\kappa(\tilde u^{M*}) + \Psi(\tilde\xi_t^{M*})
	\nn\\
	& = \bar J_t(x,\tilde u^{M*}) + \int_0^t \nu(\tilde\xi_s^{M*}, \alpha_s^{M*}) - \nu(\tilde\xi_s^{M*}, \tilde\alpha_s^{M*}) \, ds, 
	\nn
\end{align}
so that, by subtraction,
\if{false}

??

\begin{align}
	& J_t(x,\tilde u^{M*},\alpha^{M*}) - J_t(x,u^{M*},\alpha^{M*})  
	\nn\\
	& = \int_0^t \nu(\tilde\xi_s^{M*},\alpha_s^{M*}) - \nu(\xi_s^{M*},\alpha_s^{M*})\, ds
	+ [I_t^\kappa(\tilde u^{M*}) - I_t^\kappa(u^{M*})]
	+ [\Psi(\tilde\xi_t^{M*}) - \Psi(\xi_t^{M*})]
	\label{eq:pre-cost-difference}
	\\
	& = \int_0^t [\nu(\tilde\xi_s^{M*}, \tilde\alpha_s^{M*}) -  \nu(\xi_s^{M*},\alpha_s^{M*})]\, ds
	+ [I_t^\kappa(\tilde u^{M*}) - I_t^\kappa(u^{M*})]
	+ [\Psi(\tilde\xi_t^{M*}) - \Psi(\xi_t^{M*})]
	\nn\\
	& \qquad\qquad
	+ \int_0^t [\nu(\tilde\xi_s^{M*}, \alpha_s^{M*}) - \nu(\tilde\xi_s^{M*}, \tilde\alpha_s^{M*})] \, ds
	\nn
\end{align}

??

Adding and subtracting a term involving $\tilde\alpha_s^{M*}$ in the integrand yields
$\nu(\tilde\xi_s^{M*}, \alpha_s^{M*}) -  \nu(\xi_s^{M*},\alpha_s^{M*})
	= [\nu(\tilde\xi_s^{M*}, \tilde\alpha_s^{M*}) -  \nu(\xi_s^{M*},\alpha_s^{M*})]
	+ [\nu(\tilde\xi_s^{M*}, \alpha_s^{M*}) - \nu(\tilde\xi_s^{M*}, \tilde\alpha_s^{M*})]$.
Applying \er{eq:cost-I-hat}, \er{eq:running-cost-nu}, and Lemma \ref{lem:measurable-selection-M}, 
\begin{align}
	& \int_0^t [\nu(\tilde\xi_s^{M*}, \tilde\alpha_s^{M*}) -  \nu(\xi_s^{M*},\alpha_s^{M*})]\, ds
	= \bar I_t^M(x,\tilde u^{M*}) - \bar I_t^M(x,u^{M*})\,.
	\nn
\end{align}
Hence, \er{eq:cost-J-M}, \er{eq:cost-J-hat}, \er{eq:pre-cost-difference} imply that

\fi
\begin{align}
	& \hspace{-10mm} 
	J_t(x,\tilde u^{M*},\alpha^{M*}) - J_t(x,u^{M*},\alpha^{M*})
	\nn\\
	&
	= \bar J_t^M(x,\tilde u^{M*}) - \bar J_t^M(x,u^{M*})
	+ \! \int_0^t \! \nu(\tilde\xi_s^{M*}, \alpha_s^{M*}) - \nu(\tilde\xi_s^{M*}, \tilde\alpha_s^{M*})\, ds.
	\label{eq:cost-difference}
\end{align}
As $u^{M*}$ is the minimizer of $\bar J_t^M(x,\cdot)$, see Theorem \ref{thm:existence} and \er{eq:existence}, a lower bound for the integral term in the right-hand side of \er{eq:cost-difference} is sought, as a function of $\delta$, using Taylor's theorem. To this end, it may be shown with some calculation via \er{eq:alpha-M-star}, \er{eq:running-cost-nu}, \er{eq:u-alpha}, \er{eq:gamma} 
that
\begin{align}
	\alpha_s^{M*} - \tilde\alpha_s^{M*}
	& = \left\{ \ba{cl}
			a\circ\phi'(|\xi_s^{M*}|^2) - a\circ\phi'(|\tilde\xi_s^{M*}|^2)\,, 
			& |\tilde\xi_s^{M*}|^2 < \hat\rho(M)\,,  |\xi_s^{M*}|^2 < \hat\rho(M)\,,
			\\
			M -  a\circ\phi'(|\tilde\xi_s^{M*}|^2)\,, & |\tilde\xi_s^{M*}|^2 < \hat\rho(M) \le  |\xi_s^{M*}|^2\,, 
			\\
			a\circ\phi'(|\xi_s^{M*}|^2) - M\,,
			& |\tilde\xi_s^{M*}|^2 \ge \hat\rho(M) > |\xi_s^{M*}|^2 
			\\
			0\,,
			&  |\tilde\xi_s^{M*}|^2 \ge \hat\rho(M), \ |\xi_s^{M*}|^2 \ge \hat\rho(M)\,,
		\ea \right.
	\nn\\
	\ts{\pdtone{\nu}{\hat\alpha}} (\tilde\xi_s^{M*}, \tilde\alpha_s^{M*})
	& = \frac{|\tilde\xi_s^{M*}|^2}{a'\circ a^{-1}(\tilde\alpha_s^{M*})} - 1
	= \left\{ \ba{cl}
		0\,, & |\tilde\xi_s^{M*}|^2 < \hat\rho(M)\,,
		\\
		\frac{|\tilde\xi_s^{M*}|^2}{\hat\rho(M)} - 1\,,
		& |\tilde\xi_s^{M*}|^2 \ge \hat\rho(M)\,,
	\ea \right.	
	\label{eq:alpha-difference}
	\\
	\ts{\pdttwo{\nu}{\alpha}}(\tilde\xi_s^{M*},\tilde\alpha_s^{M*}) 
	& = \frac{-|\tilde\xi_s^{M*}|^2}{2\, \phi''(a'\circ a^{-1}(\tilde\alpha_s^{M*})) \, [a'\circ a^{-1}(\tilde\alpha_s^{M*})]^3}
	= \left\{ \ba{cl}
		\frac{-1}{2\, \phi''(|\tilde\xi_s^{M*}|^2)\, |\tilde\xi_s^{M*}|^2},
		& |\tilde\xi_s^{M*}|^2 < \hat\rho(M)\,,
		\\[3mm]
		\frac{-|\tilde\xi_s^{M*}|^2}{2\, \phi''(\hat\rho(M))\, [\hat\rho(M)]^3},
		& |\tilde\xi_s^{M*}|^2 \ge \hat\rho(M)\,,
	\ea \right.
	\nn	
\end{align}
in which the derivatives follow by Lemma \ref{lem:gamma-properties} and the identity $\nu(\tilde\xi_s^{M*},\alpha) = \frac{K}{2}\, | \tilde\xi_s^{M*}|^2 + \demi\, \gamma_{| \tilde\xi_s^{M*}|^2}(\alpha)$ for all $\alpha\in\R_{\ge -\phi(0)}$, with $\gamma_{(\cdot)}$ as per \er{eq:gamma}.
Observe that the second partial derivative in \er{eq:alpha-difference} is unbounded if $|\tilde\xi_s^{M*}|\rightarrow 0$. Two cases are thus considered, {\em (i)} $s\in\Delta_0^\eps$, and {\em (ii)} $s\in[0,t]\setminus\Delta_0^\eps$, in which 
\begin{align}
	\Delta_0^\eps
	& \doteq \left\{ s\in [0,t] \ \bigl| \ |\xi_s^{M*}|\le \ts{\frac{\eps}{2}} \right\} \subset[0,t].
	\label{eq:Delta-0}
\end{align}

{\em (i)} Fix $s\in\Delta_0^\eps$. The triangle inequality, \er{eq:closeness}, \er{eq:Delta-0} imply that
$|\tilde\xi_s^{M*}| \le |\tilde\xi_s^{M*} - \xi_s^{M*}| + |\xi_s^{M*}|\le \ts{\frac{\eps}{4}} + \ts{\frac{\eps}{2}} = \ts{\frac{3\, \eps}{4}}$, so that $\max( |\tilde\xi_s^{M*}|^2, |\xi_s^{M*}|^2 ) \le \eps^2 < \hat\rho(M)$ by definition of $\eps$. Hence, \er{eq:alpha-M-star}, \er{eq:u-alpha} yield $\tilde\alpha_s^{M*} = a\circ\phi'(|\tilde\xi_s^{M*}|^2)$, $\alpha_s^{M*} = a\circ\phi'(|\xi_s^{M*}|^2)$, so that $a^{-1}(\tilde\alpha_s^{M*}) = \phi'(|\tilde\xi_s^{M*}|^2)$, $a^{-1}(\alpha_s^{M*}) = \phi'(|\xi_s^{M*}|^2)$. Note also that $a\circ\phi'$ is differentiable by Lemma \ref{lem:a-properties}, with $(a\circ\phi')'(\rho) = [a'\circ\phi'(\rho)]\, \phi''(\rho) = \rho\, \phi''(\rho)$ for all $\rho\in[0,\eps]$. Hence, \er{eq:running-cost-nu}, the triangle inequality, Taylor's theorem, and \er{eq:closeness} together imply that the integrand in \er{eq:cost-difference} satisfies
\begin{align}
	& |\nu(\tilde\xi_s^{M*}, \alpha_s^{M*}) - \nu(\tilde\xi_s^{M*}, \tilde\alpha_s^{M*})|
	\nn\\
	& 
	\le \demi\, | a^{-1}(\alpha_s^{M*}) - a^{-1}(\tilde\alpha_s^{M*}) | \,  |\tilde\xi_s^{M*}|^2
	+ \demi\, | \tilde\alpha_s^{M*} - \alpha_s^{M*} |
	\nn\\
	& \le \ts{\frac{\eps^2}{2}}\, | \phi'(|\xi_s^{M*}|^2) - \phi'(|\tilde\xi_s^{M*}|^2) | 
			+ \demi\, | a\circ\phi'(|\xi_s^{M*}|^2) - a\circ\phi'(|\tilde\xi_s^{M*}|^2)|
	\nn\\
	& \le \demi \left( \eps^2\, | \phi''(\mu_s) |  
	+ |\rho_s\, \phi''(\rho_s)| \right) \left| |\xi_s^{M*}|^2 - |\tilde\xi_s^{M*}|^2 \right| 
	\le \ts{\frac{\eps}{2}}\, (\eps + 1) \sup_{\rho\in[0,\eps]} |\phi''(\rho)| \left| |\xi_s^{M*}|^2 - |\tilde\xi_s^{M*}|^2 \right| 
	\nn\\
	& \le \eps \sup_{\rho\in[0,\hat\rho(M)]} |\phi''(\rho)|\, ( |\xi_s^{M*}| + |\tilde\xi_s^{M*}| )\, | \xi_s^{M*} - \tilde\xi_s^{M*} |
	\le 2\, \eps^2\, \sup_{\rho\in[0,\hat\rho(M)]} |\phi''(\rho)| \, | \xi_s^{M*} - \tilde\xi_s^{M*} |
	\le K_0^M(\eps)\, \delta\,,
	\nn
\end{align}
in which $\eps<1$ by definition, and $\mu_s,\rho_s\in[0,\eps]\subset[0,\hat\rho(M)]$ lie in an interval defined by the end points $|\xi_s^{M*}|^2$ and $|\tilde\xi_s^{M*}|^2$, and
\begin{align}
	& K_0^M(\eps) \doteq L_1^M\, \eps^2\,,
	\quad L_1^M\doteq \sup_{\rho\in[0,\hat\rho(M)]} 2\, |\phi''(\rho)|\,.
	\label{eq:K-L-eps}
\end{align}
As $s\in\Delta_0^\eps$ is arbitrary, integration yields
\begin{align}
	\int_{\Delta_0^\eps} 
	\nu(\tilde\xi_s^{M*}, \alpha_s^{M*}) - \nu(\tilde\xi_s^{M*}, \tilde\alpha_s^{M*}) \, ds
	& \ge 
	\int_{\Delta_0^\eps} 
	-|\nu(\tilde\xi_s^{M*}, \alpha_s^{M*}) - \nu(\tilde\xi_s^{M*}, \tilde\alpha_s^{M*})| \, ds
	\ge -t\, K_0^M(\eps)\, \delta\,.
	\label{eq:nu-diff-bound-1}
\end{align}

{\em (ii)} Fix any $s\in [0,t]\setminus\Delta_0^\eps$. As $|\xi_s^{M*}| > \ts{\frac{\eps}{2}}$, by definition of $\Delta_0^\eps$, the triangle inequality and \er{eq:closeness} imply that $|\tilde\xi_s^{M*}| \ge |\xi_s^{M*}| - |\tilde\xi_s^{M*} - \xi_s^{M*}|  >  \ts{\frac{\eps}{2}} - \ts{\frac{\eps}{4}} = \ts{\frac{\eps}{4}}$, so that $\tilde\xi_s^{M*}, \xi_s^{M*} \not\in\cB_{\R^n}[0;\ts{\frac{\eps}{4}}]$. Consequently, \er{eq:u-alpha}, \er{eq:alpha-difference} imply that $\pdtone{\nu}{\alpha}(\tilde\xi_s^M,\tilde\alpha_s^{M*})$ and $\pdttwo{\nu}{\alpha}(\tilde\xi_s^M,\bar\alpha_s)$ exist and are uniformly bounded for $s\in[0,t]$, given any $\bar\alpha_s$ contained in the interval defined by the end points $\tilde\alpha_s^{M*}$ and $\alpha_s^{M*}$. By Taylor's theorem, such an $\bar\alpha_s$ exists, and satisfies
\begin{align}
	\nu(\tilde\xi_s^{M*}, \alpha_s^{M*}) - \nu(\tilde\xi_s^{M*}, \tilde\alpha_s^{M*})
	& = \ts{\pdtone{\nu}{\hat\alpha}} (\tilde\xi_s^{M*}, \tilde\alpha_s^{M*}) \, ( \alpha_s^{M*} - \tilde\alpha_s^{M*} ) 
			+ \demi\, \ts{\pdttwo{\nu}{\hat\alpha}} (\tilde\xi_s^{M*}, \bar\alpha_s)\, ( \alpha_s^{M*} - \tilde\alpha_s^{M*} )^2\,.
	\label{eq:nu-series}
\end{align}
Note by inspection of the various cases in \er{eq:alpha-difference} that the first order term is equivalently given by
\begin{align}
	\ts{\pdtone{\nu}{\hat\alpha}} (\tilde\xi_s^{M*}, \tilde\alpha_s^{M*})\, (\alpha_s^{M*} - \tilde\alpha_s^{M*})
	& = \left\{ \ba{cl}
		( a\circ\phi'(|\xi_s^{M*}|^2) - M )\, ( \frac{|\tilde\xi_s^{M*}|^2}{\hat\rho(M)} - 1 ),
		& |\tilde\xi_s^{M*}|^2 \ge \hat\rho(M) > |\xi_s^{M*}|^2\,,
		\nn\\
		0\,, & \text{otherwise}.
	\ea \right.
	\nn
\end{align}
Let $R^M\doteq \|\xi^{M*}\|_{C([0,t];\R^n)}>\ts{\frac{\eps}{2}}$, and note that $|\tilde\xi_s^{M*}| \le |\xi_s^{M*}| + |\tilde\xi_s^{M*} - \xi_s^{M*}| \le R^M + \ts{\frac{\eps}{4}}$, by \er{eq:closeness}. In the non-zero case above, as $M = a\circ\phi'(\hat\rho(M))$, a second application of Taylor's theorem yields
\begin{align}
	& |a\circ \phi'(|\xi_s^{M*}|^2) - M|
	= |a\circ \phi'(\hat\rho(M) + [|\xi_s^{M*}|^2 - \hat\rho(M)] ) - a\circ\phi'(\hat\rho(M))|
	\nn\\
	& 
	= |(a\circ\phi')'(\mu_s)| \, \left| |\xi_s^{M*}|^2 - \hat\rho(M) \right| = |\mu_s\, \phi''(\mu_s)| \, \left| |\xi_s^{M*}|^2 - \hat\rho(M) \right|
	\nn\\
	&
	\le \sup_{\mu\in[0,\hat\rho(M)]} |\mu\, \phi''(\mu)|\, \left| |\xi_s^{M*}|^2 -  |\tilde\xi_s^{M*}|^2 \right|
	\le \hat\rho(M) \sup_{\mu\in[0,\hat\rho(M)]} |\phi''(\mu)|\, \left( |\xi_s^{M*}| + |\tilde\xi_s^{M*}| \right) |\xi_s^{M*} -  \tilde\xi_s^{M*}|
	\nn\\
	& \le \hat\rho(M) \sup_{\mu\in[0,\hat\rho(M)]} |\phi''(\mu)|\, \left( 2\, R^M + \ts{\frac{\eps}{4}} \right) L_t\, \delta
	= \hat\rho(M)\, L_1^M\, (R^M + \ts{\frac{\eps}{8}} )\, L_t\, \delta\,.
	\nn
\end{align}
in which $\mu_s\in[|\xi_s^{M*}|^2, \hat\rho(M)]$. Similarly,
\begin{align}
	\left| \ts{\frac{|\tilde\xi_s^{M*}|^2}{\hat\rho(M)}} - 1 \right|
	& \le \ts{\frac{1}{\hat\rho(M)}} \left| |\xi_s^{M*}|^2 -  |\tilde\xi_s^{M*}|^2 \right|
	\le \ts{\frac{2}{\hat\rho(M)}} \, \left( R^M + \ts{\frac{\eps}{8}} \right) L_t\, \delta\,.
	\nn
\end{align}
Hence, combining these inequalities yields a lower bound for the first order term, with
\begin{align}
	\ts{\pdtone{\nu}{\hat\alpha}} (\tilde\xi_s^{M*}, \tilde\alpha_s^{M*})\, (\alpha_s^{M*} - \tilde\alpha_s^{M*})
	& \ge -|\ts{\pdtone{\nu}{\hat\alpha}} (\tilde\xi_s^{M*}, \tilde\alpha_s^{M*})\, (\alpha_s^{M*} - \tilde\alpha_s^{M*})|
	\ge -K_1^M(\eps)\, \delta^2\,,
	\label{eq:nu-first-order-bound}
\end{align}
with $K_1^M(\eps) \doteq 2\, L_1^M \left( R^M + \ts{\frac{\eps}{8}} \right)^2 L_t^2$.


The second order term in \er{eq:nu-series} has the same form as \er{eq:alpha-difference}, with
\begin{align}
	\ts{\pdttwo{\nu}{\alpha}}(\tilde\xi_s^{M*},\bar\alpha_s) 
	& = \frac{-|\tilde\xi_s^{M*}|^2}{2\, \phi''(a'\circ a^{-1}(\bar\alpha_s)) \, [a'\circ a^{-1}(\bar\alpha_s)]^3}
	= \frac{-|\tilde\xi_s^{M*}|^2}{2\, \phi''\circ\hat\rho(\bar\alpha_s)\, [\hat\rho(\bar\alpha_s)]^3}\,,
	\nn
\end{align}
in which $\bar\alpha_s$ is in the interval defined by the end points $\tilde\alpha_s^{M*}$ and $\alpha_s^{M*}$. 
As $\tilde\xi_s^{M*}, \xi_s^{M*}\not\in\cB_{\R^n}[0;\ts{\frac{\eps}{4}}]$, \er{eq:u-alpha} implies that $\alpha_s^{M*}, \tilde\alpha_s^{M*}\in[a\circ\phi'(\frac{\eps^2}{16}), M]$, so that $\bar\alpha_s\in[a\circ\phi'(\frac{\eps^2}{16}), M]$. Hence, Lemma \ref{lem:a-properties}, i.e. \er{eq:exact-a-dot}, yields
\begin{align}
	|\ts{\pdttwo{\nu}{\alpha}}(\tilde\xi_s^{M*},\bar\alpha_s)|
	& \le \frac{|\tilde\xi_s^{M*}|^2}{\displaystyle \inf_{\rho\in[\eps^2/16, \hat\rho(M)]} \{ 2\, \rho^3\, \phi''(\rho) \} }
	\le L_2^M\, (\ts{\frac{16}{\eps^2}})^3\, (R^M+\ts{\frac{\eps}{4}})^2\,,
	\quad
	L_2^M \doteq \sup_{\rho\in[0,\hat\rho(M)]} [2\, \phi''(\rho)]^{-1}.
	\nn
\end{align}
Furthermore, in each of the four cases listed for $\alpha_s^{M*} - \tilde\alpha_s^{M*}$ in \er{eq:alpha-difference}, Taylor's theorem again yields
\begin{align}
	|\alpha_s^{M*} - \tilde\alpha_s^{M*}|
	& \le (a\circ\phi)'(\rho_s)\, \left| |\xi_s^{M*}|^2 - |\tilde\xi_s^{M*}|^2 \right|
	\le
	\sup_{\rho\in[0,\hat\rho(M)]} \rho\, \phi''(\rho)\, (|\xi_s^{M*}| + |\tilde\xi_s^{M*}|) \, |\xi_s^{M*} - \tilde\xi_s^{M*}|
	\nn\\
	& \le \hat\rho(M)\, L_1^M\, (R^M + \ts{\frac{\eps}{8}})\, L_t\, \delta\,,
	\nn
\end{align}
in which $\rho_s\in[0,\hat\rho(M)]$ in every case. Hence, a lower bound for the second order term is
\begin{align}
	\demi\, \ts{\pdttwo{\nu}{\hat\alpha}} (\tilde\xi_s^{M*}, \bar\alpha_s)\, ( \alpha_s^{M*} - \tilde\alpha_s^{M*} )^2
	& \ge - \demi\, |\ts{\pdttwo{\nu}{\hat\alpha}} (\tilde\xi_s^{M*}, \bar\alpha_s)|\, |\alpha_s^{M*} - \tilde\alpha_s^{M*} |^2
	\ge - K_2^M(\eps)\, \delta^2\,,
	\label{eq:nu-second-order-bound}
\end{align}
with $K_2^M(\eps)\doteq \demi\, L_2^M\, (\ts{\frac{16}{\eps^2}})^3\, (R^M+\ts{\frac{\eps}{4}})^2\, [\hat\rho(M)\, L_1^M\, (R^M + \ts{\frac{\eps}{8}})\, L_t]^2$. Thus, integrating \er{eq:nu-series} via \er{eq:nu-first-order-bound}, \er{eq:nu-second-order-bound},
\begin{align}
	& \int_{[0,t]\setminus \Delta_0^\eps}
	\nu(\tilde\xi_s^{M*}, \alpha_s^{M*}) - \nu(\tilde\xi_s^{M*}, \tilde\alpha_s^{M*})\, ds
	\nn\\
	& = \int_{[0,t]\setminus \Delta_0^\eps} 
	\ts{\pdtone{\nu}{\hat\alpha}} (\tilde\xi_s^{M*}, \tilde\alpha_s^{M*}) \, ( \alpha_s^{M*} - \tilde\alpha_s^{M*} ) 
			+ \demi\, \ts{\pdttwo{\nu}{\hat\alpha}} (\tilde\xi_s^{M*}, \bar\alpha_s)\, ( \alpha_s^{M*} - \tilde\alpha_s^{M*} )^2
			\, ds
	\nn\\
	& \ge \int_{[0,t]\setminus \Delta_0^\eps}  -K_1^M(\eps)\, \delta^2 - K_2^M(\eps)\, \delta^2 \, ds
	\ge - t\, [K_1^M(\eps) + K_2^M(\eps) ] \, \delta^2\,.
	\label{eq:nu-difference-case-ii}
\end{align}

Cases {\em (i)} and {\em (ii)} may now be combined, via \er{eq:nu-diff-bound-1} and \er{eq:nu-difference-case-ii}, in \er{eq:cost-difference}, with
\begin{align}
	& J_t(x,\tilde u^{M*},\alpha^{M*}) - J_t(x,u^{M*},\alpha^{M*})
	= \bar J_t^M(x,\tilde u^{M*}) - \bar J_t^M(x,u^{M*}) 
	\nn\\
	& \hspace{20mm}
		+ \int_{\Delta_0^\eps} \nu(\tilde\xi_s^{M*}, \alpha_s^{M*}) - \nu(\tilde\xi_s^{M*}, \tilde\alpha_s^{M*})\, ds
		+ \int_{[0,t]\setminus \Delta_0^\eps} \nu(\tilde\xi_s^{M*}, \alpha_s^{M*}) - \nu(\tilde\xi_s^{M*}, \tilde\alpha_s^{M*})\, ds
	\nn\\
	& \ge  \bar J_t^M(x,\tilde u^{M*}) - \bar J_t^M(x,u^{M*}) - t\, K_0^M(\eps)\, \delta 
									-  t\, [K_1^M(\eps) + K_2^M(\eps) ] \,\delta^2\,.
	\nn
\end{align}
Recalling \er{eq:u-alpha}, a lower bound for the directional derivative \er{eq:Frechet-J} can subsequently be evaluated, with
\begin{align}
	D_u J_t(x,u^{M*},\alpha^{M*}(\tilde u)
	& =
	\lim_{\delta\rightarrow 0^+}  \left\{ \frac{J_t(x,u^{M*}+\delta\, \tilde u,\alpha^{M*}) - J_t(x,u^{M*},\alpha^{M*})}{\delta} \right\}
	\nn\\
	& \ge \liminf_{\delta\rightarrow 0^+} \left\{ \frac{\bar J_t^M(x,u^{M*}+\delta\, \tilde u) - \bar J_t(x,u^{M*})}{\delta} \right\}
	- t\, K_0^M(\eps)
	\ge -t\, K_0^M(\eps)\,,
	\nn
\end{align}
in which the second inequality follows by Theorem \ref{thm:existence}, i.e. \er{eq:existence}. Furthermore, as $\eps\in\R_{>0}$ can be selected arbitrarily small, cf. its definition prior to \er{eq:u-alpha}, and $K_0^M(0) = 0$ by \er{eq:K-L-eps}, it follows that
\begin{align}
	D_u J_t(x,u^{M*},\alpha^{M*})(\tilde u)
	& \ge \liminf_{\delta\rightarrow 0^+} \left\{ \frac{\bar J_t^M(x,u^{M*}+\delta\, \tilde u) - \bar J_t(x,u^{M*})}{\delta} \right\},
	\nn
\end{align}
in which $\tilde u\in\cU[0,t]$, $\|\tilde u\|_{\cU[0,t]} = 1$, is arbitrary. Hence, $u^{M*}\in\cU[0,t]$ minimizes $J_t(x,\cdot,\alpha^{M*})$.
\end{proof}

\begin{theorem}
\label{thm:inf-sup-swap-M}
Given $t\in\R_{>0}$, $M\in\R_{\ge -\phi(0)}$, $x\in\R^n$, and $u^{M*}$, $\alpha^{M*}$ as per Lemma \ref{lem:u-star-is-optimal-in-cost-J-hat},
\begin{align}
	& \ol{W}_t^M(x) = W_t^M(x)
	= \inf_{u\in\cU[0,t]} \sup_{\alpha\in\cA^M[0,t]} J_t(x,u,\alpha) 
	\nn\\
	& = \min_{u\in\cU[0,t]} \max_{\alpha\in\cA^M[0,t]} J_t(x,u,\alpha) 
	= \max_{\alpha\in\cA^M[0,t]} \min_{u\in\cU[0,t]} J_t(x,u,\alpha) 
	= J_t(x,u^{M*},\alpha^{M*})
	\label{eq:games-M}
\end{align}
\end{theorem}
\begin{proof}
Fix $t\in\R_{> 0}$, $M\in\R_{\ge -\phi(0)}$, $x\in\R^n$, and $\alpha^{M*}$, $u^{M*}$ as per Lemma \ref{lem:u-star-is-optimal-in-cost-J-hat}. Recalling Theorem \ref{thm:game-W-M},
\begin{align}
	\ol{W}_t^M(x) = W_t^M(x)
	& = \inf_{u\in\cU[0,t]} \sup_{\alpha\in\cA^M[0,t]} J_t(x,u,\alpha) 
	\ge \sup_{\alpha\in\cA^M[0,t]}  \inf_{u\in\cU[0,t]} J_t(x,u,\alpha).
	\nn
\end{align}
For the opposite inequality, and existence of the minimizer and maximzer as per the final equality in \er{eq:games-M}, note by Theorem \ref{thm:existence}, Lemma \ref{lem:measurable-selection-M}, the definition of $\alpha^{M*}$, and finally Lemma \ref{lem:u-star-is-optimal-in-cost-J-hat}, that
\begin{align}
	& \inf_{u\in\cU[0,t]} \sup_{\alpha\in\cA^M[0,t]} J_t(x,u,\alpha) 
	= \inf_{u\in\cU[0,t]} \bar J_t^M(x, u)
	= \bar J_t^M(x,u^{M*}) = J_t(x,u^{M*},\alpha^{M*}),
	\nn\\
	& = \min_{u\in\cU[0,t]} J_t(x,u,\alpha^{M*})
	\le \max_{\alpha\in\cA^M[0,t]} \min_{u\in\cU[0,t]} J_t(x,u,\alpha)
	=  \sup_{\alpha\in\cA^M[0,t]} \inf_{u\in\cU[0,t]} J_t(x,u,\alpha)\,.
	\nn\\[-17mm]
	& \nn
\end{align}
\end{proof}

\vspace{2mm}
\begin{corollary}
\label{cor:inf-sup-swap}
Given $t\in\R_{>0}$ and $x\in\R^n$, the game upper value defined by $W_t$ of \er{eq:game-W} and the corresponding game lower value are equivalent, with
\begin{align}
	 & W_t(x) \doteq \inf_{u\in\cU[0,t]} \sup_{\alpha\in\cA[0,t]} J_t(x,u,\alpha)
	 = \sup_{\alpha\in\cA[0,t]} \inf_{u\in\cU[0,t]} J_t(x,u,\alpha).
	 \label{eq:games}
\end{align}
\end{corollary}
\begin{proof}
Fix $t\in\R_{>0}$, $x\in\R^m$. Applying Theorem \ref{thm:game}, followed by Theorems \ref{thm:monotone} and \ref{thm:inf-sup-swap-M},
\begin{align}
	W_t(x) = \ol{W}_t(x)
	& = \sup_{M\ge -\phi(0)} \ol{W}_t^M(x) = \sup_{M\ge -\phi(0)} W_t^M(x)
	= \sup_{M\ge -\phi(0)} \inf_{u\in\cU[0,t]} \sup_{\alpha\in\cA^M[0,t]} J_t(x,\alpha,u)
	\nn\\
	& = \sup_{M\ge -\phi(0)} \sup_{\alpha\in\cA^M[0,t]} \inf_{u\in\cU[0,t]} J_t(x,u,\alpha)
	= \sup_{\alpha\in\cA[0,t]} \inf_{u\in\cU[0,t]} J_t(x,u,\alpha).
	\nn\\[-17mm]
	& \nn
\end{align}
\end{proof}
\vspace{2mm}


\subsection{Computation via the lower value}
Theorems \ref{thm:game}, \ref{thm:game-W-M}, \ref{thm:inf-sup-swap-M}, and Corollary \ref{cor:inf-sup-swap}, together establish equivalences of the exact and approximate regulator problems \er{eq:value-W} and \er{eq:value-W-M} with the corresponding exact and approximate games \er{eq:game-W}, \er{eq:game-W-M}, \er{eq:games-M}, \er{eq:games}, and that the upper and lower values of these games are equivalent in both cases. With a view to computation, via the value function and optimal trajectories corresponding to the approximate regulator problem \er{eq:value-W-M}, it is useful to explicitly consider the lower value of the approximate game.
%
%
To this end, given $t\in\R_{\ge 0}$, $M\in\R_{\ge -\phi(0)}$, $\alpha\in\cA^M[0,t]$, define an auxiliary value function $\wh{W}_t^{\alpha}$ by
\begin{align}
	\wh{W}_t^{\alpha}(x)
	& \doteq \inf_{u\in\cU[0,t]} J_t(x,u,\alpha)
	\label{eq:value-W-hat-M-alpha}
\end{align}
for all $t\in\R_{\ge 0}$, $x\in\R^n$. The following is then immediate.

\begin{lemma}
\label{lem:game}
Given any $t\in\R_{\ge 0}$, $M\in\R_{\ge -\phi(0)}$,  the value functions 
$W_t^M,\wh{W}_t^{\alpha}:\R^n\rightarrow\ol{\R}^+$, $\alpha\in\cA^M[0,t]$, of \er{eq:game-W}, \er{eq:value-W-hat-M-alpha} satisfy $W_t^M(x) = \sup_{\alpha\in\cA^M[0,t]} \wh{W}_t^{\alpha}(x)$ for all $x\in\R^n$.
\end{lemma}

\newcommand{\br}[1]	{{\breve {#1}}}

By inspection of \er{eq:cost-J-hat}, \er{eq:cost-I-hat}, \er{eq:running-cost-nu}, $\wh{W}_t^{\alpha}$ of \er{eq:value-W-hat-M-alpha} defines the value of an LQR problem, parameterized by $\alpha\in\cA^M[0,t]$. In order to demonstrate that $\wh{W}_t^{\alpha}$ has an explicit quadratic representation, it is convenient to consider the final value problem (FVP)
\begin{align}
	-\dot {\hat P}_s^\alpha
	& = \hat A' \, \hat P_s^\alpha + \hat P_s^\alpha\, \hat A - \ts{\frac{1}{\kappa}} \, \hat P_s^\alpha\, \hat B\, \hat B' \, \hat P_s^\alpha 
	+ \hat V_s^\alpha,
		\quad  
	&& \hat P_t^\alpha = \hat P_t,
	\label{eq:DRE}
\end{align}
for all $s\in [0,t]$, in which $\hat A, \hat P_t, \hat V_s\in\R^{(n+1)\times(n+1)}$, $s\in[0,t]$, $\hat B\in\R^{(n+1)\times m}$, $\hat C\in\R^{n\times (n+1)}$ are defined in terms of $\kappa$, $K$ of \er{eq:cost-J-bar}, $A$, $B$ of \er{eq:dynamics}, and $P_t$ of \er{eq:cost-Psi} by
\begin{gather}
	\hat A \doteq \left( \ba{l|l}
				A & 0_{n} \\\hline
				0_{n}' & 0 
			\ea \right),
	\ \hat B \doteq \left( \ba{c}
						B \\\hline 0_m'
					\ea \right),
	\ \hat P_t \doteq \left( \ba{c|c}
			P_t & Q_t \\\hline
			Q_t' & R_t
		\ea \right),
	\ \hat V_s^\alpha
	\doteq \left( \ba{c|c}
			[K + a^{-1}(\alpha_s)]\, I_n & 0_n
			\\\hline
			0_n' & -\alpha_s
		\ea \right),
	\nn\\
	Q_t \doteq - P_t\, z, \quad R_t \doteq \langle z,\, P_t \, z\rangle,
	\label{eq:data-and-BC}
\end{gather}
in which $I_n\in\R^{n\times n}$ and $0_n\in\R^{n\times 1}$ denote the identity matrix and zero vector respectively. Solutions to FVP \er{eq:DRE}, where they exist, take the compatibly partitioned form
\begin{align}
	\hat P_s^\alpha
	& \doteq \left( \ba{c|c}
			P_s^\alpha & Q_s^\alpha \\\hline
			(Q_s^\alpha)' & R_s^\alpha
		\ea \right),
	\label{eq:P-hat}
\end{align}
for all $s\in[0,t]$. 

\begin{remark}
FVP \er{eq:DRE} may be equivalently expressed as three component FVPs
\begin{align}
	& - \dot P_s^\alpha
	= A'\, P_s^\alpha + P_s^\alpha \, A - \ts{\frac{1}{\kappa}}\, P_s^\alpha\, B\, B' \, P_s^\alpha 
	+ [ K + a^{-1}(\alpha_s) ] I_n,
	&& P_t^\alpha = P_t,
	\label{eq:DRE-P}
	\\
	& -\dot Q_s^\alpha
	= (A - \ts{\frac{1}{\kappa}}\, B \, B' \, P_s^\alpha )'\, Q_s^\alpha,
	&& Q_t^\alpha = Q_t,
	\label{eq:DRE-Q}
	\\
	& -\dot R_s^\alpha
	= -\alpha_s - \ts{\frac{1}{\kappa}}\, (Q_s^\alpha)' \, B\, B'\, Q_s^\alpha,
	\quad\  s\in(0,t),
	&& R_t^\alpha = R_t,
	\label{eq:DRE-R}
\end{align}
in which the respective boundary data follows by \er{eq:cost-Psi}, \er{eq:data-and-BC}.
\if{false}

 by
\begin{align}
	& 
	P_t^\alpha = P_t,
	\quad Q_t^\alpha = Q_t,
	\quad R_t^\alpha = R_t.
	\label{eq:BC-PQR}
\end{align}

\fi
\end{remark}

\begin{lemma}
\label{lem:existence}
Given fixed $t\in\R_{>0}$, $M\in\R_{\ge -\phi(0)}$, and any $\alpha\in\cA^M[0,t]$, there exists a unique $\hat P^\alpha\in C([0,t];\Sigma^{n+1})\cap C^1((0,t);\Sigma^{n+1})$ of the form \er{eq:P-hat} that satisfies FVP \er{eq:DRE}.
\end{lemma}
\begin{proof}
See for example \cite[Theorem 37, p.364]{S-book:98} or \cite[7, Lemma 2.2, p. 391]{BDDM:07}.
\end{proof}

\begin{theorem}
\label{thm:inner-LQR}
Given any $t\in\R_{>0}$, $M\in\R_{\ge -\phi(0)}$, $\alpha\in\cA^M[0,t]$, the auxiliary value function $\wh{W}_t^{\alpha}$ of \er{eq:value-W-hat-M-alpha} satisfies $\wh{W}_t^{\alpha}(x) = \br{W}_t^{\alpha}(0,x)$ for all $x\in\R^n$, where $\br{W}_t^{\alpha}:[0,t]\times\R^n\rightarrow\ol{\R}^+$ is given by
\begin{align}
	\br{W}_t^{\alpha}(s,x)
	& \doteq  \demi \left\langle \left( \ba{c}
					x \\ 1
			\ea \right),\, [\hat P_{t}^{\alpha}]_s  \left( \ba{c}
					x \\ 1
			\ea \right) \right\rangle
	\nn
\end{align}
for all $s\in[0,t]$, $x\in\R^n$, in which $\hat P_t^{\alpha}\in C([0,t];\Sigma^{n+1})\cap C^1((0,t);\Sigma^{n+1})$ is the unique solution of FVP \er{eq:DRE}. Furthermore, the optimal input $u^\alpha\in\cU[0,t]$ in \er{eq:value-W-hat-M-alpha} has the state feedback characterization
\begin{equation}
	\begin{aligned}
	& \dot\xi_s^{\alpha}
	= A\, \xi_s^{\alpha} + B\, u_s^{\alpha},
	&& \xi_0^\alpha = x,
	\\
	&
	u_s^{\alpha} \doteq -\ts{\frac{1}{\kappa}}\, B' \, [ P_s^\alpha\, \xi_s^{\alpha} + Q_s^\alpha],
	&& s\in(0,t),
	\end{aligned}
	\label{eq:optimal-control-fixed-alpha}
\end{equation}
for any $x\in\R^n$, where $P_s^\alpha$, $Q_s^\alpha$ are as per \er{eq:DRE-P}, \er{eq:DRE-Q}. 
\end{theorem}
\begin{proof}
Fix arbitrary $t\in\R_{>0}$, $M\in\R_{\ge -\phi(0)}$, and $\alpha\in\cA^M[0,t]$. Applying Lemma \ref{lem:existence}, there exists a unique $\hat P^\alpha\in C([0,t];\Sigma^{n+1})\cap C^1((0,t);\Sigma^{n+1})$ of the form \er{eq:P-hat} that satisfies FVP \er{eq:DRE}. Consequently, given any $s\in(0,t)$, $x\in\R^n$, \er{eq:DRE}, \er{eq:P-hat}, \er{eq:DRE-P}, \er{eq:DRE-Q}, \er{eq:DRE-R} imply that 
\begin{align}
	\ts{\pdtone{\br{W}_t^\alpha}{s}}(s,x)
	& = \demi\, \langle x,\, \dot P_s^\alpha\, x\rangle + \langle x,\, \dot Q_s^\alpha \rangle + \demi \, \dot R_s^\alpha
	\nn\\
	& 
	= -\demi\, \langle x, ( A' \, P_s^\alpha + P_s^\alpha \, A - \ts{\frac{1}{\kappa}}\, P_s^\alpha\, B \, B' \, P_s^\alpha 
					+ [K + a^{-1}(\alpha_s) ] I_n )\, x \rangle
	\nn\\
	& \hspace{10mm}
	- \langle x,\, (A - \ts{\frac{1}{\kappa}}\, B\, B' \, P_s^\alpha)'\, Q_s^\alpha \rangle
	+ \demi\, [ \alpha_s + \ts{\frac{1}{\kappa}}\, (Q_s^\alpha)' \, B\, B' \, Q_s^\alpha ],
	\nn\\
	\ggrad_x \br{W}_t^\alpha(s,x)
	& = P_s^\alpha \, x + Q_s^\alpha\,.
	\label{eq:derivatives-W-breve}
\end{align}
Define the Hamiltonian $H^\alpha:[0,t]\times\R^n\times\R^n\rightarrow\R$ by
\begin{align}
	H^\alpha(s,x,p)
	& \doteq \langle p, \, A\, x \rangle - \ts{\frac{1}{2\, \kappa}} \, \langle p,\, B\, B'\, p \rangle
		+ \demi\, [K + a^{-1}(\alpha_s) ]\, |x|^2 - \ts{\frac{\alpha_s}{2}}
	\nn\\
	& = \inf_{u\in\R^m} \{ 
			\langle p,\, A\, x + B\, u \rangle + \ts{\frac{\kappa}{2}}\, |u|^2 
			+ \demi\, [K + a^{-1}](\alpha_s)]\, |x|^2
			- \ts{\frac{\alpha_s}{2}}
	\}
	\label{eq:H}
\end{align}
for all $x,p\in\R^n$, $s\in[0,t]$. Combining \er{eq:derivatives-W-breve}, \er{eq:H}, note that
$-\pdtone{\br{W}_t^\alpha}{s}(s,x) =  H^\alpha(s,x,\ggrad_x \br{W}_t^\alpha(s,x))$.
Fix any $\bar u\in\cU[0,t]$. Define $\bar \xi \doteq\chi(x,\bar u)$ via \er{eq:dynamics}, and observe via \er{eq:H} that $\bar u_s$ is pointwise suboptimal in $H^\alpha(s,\bar\xi_s,\ggrad_x \br{W}_t^\alpha(s,\bar\xi_s))$ for any $s\in[0,t]$. Consequently, 
\begin{align}
	& 0 \le \ts{\pdtone{\br{W}_t^\alpha}{s}}(s,\bar\xi_s) 
	+ \langle \ggrad_x \br{W}_t^\alpha(s,\bar\xi_s),\, A\, \bar\xi_s + B\, \bar u_s \rangle + \ts{\frac{\kappa}{2}}\, |\bar u_s|^2 
			+ \demi\, [K + a^{-1}](\alpha_s)]\, |\bar\xi_s|^2
			- \ts{\frac{\alpha_s}{2}}
	\nn\\
	& = \ts{\ddtone{}{s}} \br{W}_t^\alpha(s,\bar\xi_s) 
			+ \ts{\frac{\kappa}{2}}\, |\bar u_s|^2 
 			+ \demi\, [K + a^{-1}](\alpha_s)]\, |\bar\xi_s|^2
			- \ts{\frac{\alpha_s}{2}}.
	\nn
\end{align}
Integrating with respect to $s\in[0,t]$, and observing that $\breve W_t^\alpha(t,x) = \Psi(x)$, yields
\begin{align}
	\br{W}_t^\alpha(0,x) 
	& \le 
	\int_0^t \ts{\frac{\kappa}{2}}\, |\bar u_s|^2 
			+ \demi\, [K + a^{-1}](\alpha_s)]\, |\bar\xi_s|^2
			- \ts{\frac{\alpha_s}{2}}\, ds
	+ \Psi(\bar\xi_t) 
	= J_t(x,\bar u,\alpha).
	\nn
\end{align}
As $\bar u\in\cU[0,t]$ is arbitrary, it follows by \er{eq:value-W-hat-M-alpha} that
\begin{align}
	& \br{W}_t^\alpha(0,x) \le \wh{W}_t^\alpha(x)\,,
	\label{eq:verify-lower-bound}
\end{align}
for all $x\in\R^n$. 
Consider the initial value problem \er{eq:optimal-control-fixed-alpha}.
By Lemma \ref{lem:existence}, note that $\xi^\alpha\in\Ltwo([0,t];\R^n)$ and $u^\alpha\in\cU[0,t]$. Note further that $u_s^\alpha\in\R^m$ is pointwise optimal in $H^\alpha(s,\xi_s^\alpha,\ggrad_x \br{W}_t^\alpha(s,\xi_s^\alpha))$ for any $s\in[0,t]$. Hence, repeating the above argument and applying \er{eq:value-W-hat-M-alpha}, \er{eq:verify-lower-bound},
\begin{align}
	& \wh{W}_t^\alpha(x)
	\le J_t(x,u^\alpha,\alpha) 
	= 
	\br{W}_t^\alpha(0,x) \le \wh{W}_t^\alpha(x).
	\nn
\end{align}
Recalling that $x\in\R^n$ is arbitrary completes the proof.
\end{proof}

\begin{theorem}
\label{thm:optimal-inputs}
Given any $t\in\R_{>0}$, $M\in\R_{\ge -\phi(0)}$, suppose 
there exists a unique solution 
\begin{align}
	& P^*\in C([0,t];\Sigma^n)\cap C^1((0,t);\Sigma^n),  \quad
	Q^*, \xi^*\in C([0,t];\R^n)\cap C^1((0,t);\R^n)
	\nn
\end{align}
of the two point boundary value problem (TPBVP)
\begin{align}
	\hspace{0mm}
	-\dot P_s^*
	& = A'\, P_s^* + P_s^*\, A - \ts{\frac{1}{\kappa}} \, P_s^* \, B \, B' \, P_s^* + [K + a^{-1}\circ\hat\alpha^{M*}(|\xi_s^*|^2)]\, I_n\,,
	&& P_t^* = P_t,
	\nn\\
	-\dot Q_s^*
	& = (A - \ts{\frac{1}{\kappa}}\, B\, B' \, P_s^*)'\, Q_s^*\,,
	&& Q_t^* = Q_t = -P_t\, z,
	\nn\\
	\dot \xi_s^*
	& = (A - \ts{\frac{1}{\kappa}}\, B\, B' \, P_s^*)\, \xi_s^*  - \ts{\frac{1}{\kappa}}\, B\, B' \, Q_s^*\,,
	&& \xi_0 = x,
	\label{eq:TPBVP}
\end{align}
for all $s\in(0,t)$, 
\if{false}

defined with respect to the boundary data
\begin{gather}
	P_t^* = P_t,
	\quad 
	Q_t^* = Q_t = -P_t\, z,
	\quad
	\xi_0 = x,
	\label{eq:TPBVP-BC}
\end{gather}

\fi
where $P_t$, $z$ are as per \er{eq:cost-Psi}. Then, the optimal inputs $u^* = u^{M*}\in\cU[0,t]$, $\alpha^* = \alpha^{M*}\in\cA^M[0,t]$ in \er{eq:games-M} are given by the state feedback characterizations
\begin{align}
	& u_s^{*}
	= -\ts{\frac{1}{\kappa}}\, B' (P_s^*\, \xi_s^* + Q_s^*)\,,
	\quad
	\alpha_s^{*}
	= \hat\alpha^{M*}(|\xi_s^*|^2)\,,
	\label{eq:optimal-inputs}
\end{align}
for all $s\in[0,t]$, in which $\hat\alpha^{M*}$ is as per \er{eq:alpha-M-star}, \er{eq:alpha-star-in-cost-J-hat-M}.
\end{theorem}
\begin{proof}
Fix $t\in\R_{>0}$, $M\in\R_{\ge -\phi(0)}$, and suppose that a unique solution of TPBVP \er{eq:TPBVP} exists as per the theorem statement. Inputs $u^*\in\cU[0,t]$ and $\alpha^*\in\cA^M[0,t]$ of \er{eq:optimal-inputs} thus exist and are uniquely defined by \er{eq:TPBVP}, 
\er{eq:optimal-inputs}. Applying Theorem \ref{thm:inner-LQR}, i.e. \er{eq:optimal-control-fixed-alpha}, with $\alpha\doteq\alpha^*$ yields the optimal control in \er{eq:value-W-hat-M-alpha}, with 
$u_s^{M*} = u_s^\alpha = -\ts{\frac{1}{\kappa}}\, B' \, (P_s^{\alpha}\, \xi_s^{\alpha} + Q_s^{\alpha})$,
in which $P^{\alpha}$, $Q^{\alpha}$, $\xi^{\alpha}$ are as per \er{eq:optimal-control-fixed-alpha}. Note by inspection that $P^{\alpha}$, $Q^{\alpha}$ satisfy the FVPs defined by \er{eq:DRE-P}, \er{eq:DRE-Q}, 
which are precisely the FVPs defined by the first two equations of \er{eq:TPBVP}. 
By assertion, these FVPs exhibit a unique solution given by $P^*$, $Q^*$, so that $P^{\alpha} \equiv P^*$, $Q^{\alpha}\equiv Q^*$. Consequently, \er{eq:optimal-control-fixed-alpha} and \er{eq:optimal-inputs} imply that
\begin{align}
	u_s^{\alpha}
	& = -\ts{\frac{1}{\kappa}}\, B'\, (P_s^{\alpha}\, \xi_s^{\alpha} + Q_s^{\alpha}) 
	= -\ts{\frac{1}{\kappa}}\, B'\, (P_s^*\, \xi_s^* + Q_s^*) 
	= u_s^*,
	\label{eq:u-alpha-star}
\end{align}
for all $s\in[0,t]$, so that $\xi^{\alpha} \equiv \xi^*$ by \er{eq:optimal-control-fixed-alpha}. Consequently,
$\alpha_s = \alpha_s^* = \hat\alpha^{M*}(|\xi_s^*|^2) =  \hat\alpha^{M*}(|[\chi(x,u^*)]_s|^2)$ for all $s\in[0,t]$. Hence, Lemma \ref{lem:measurable-selection-M} implies that
\begin{align}
	& \bar J_t^M(x,u^*) = \sup_{\alpha\in\cA^M[0,t]} J_t(x, u^*, \alpha) = J_t(x, u^*, \alpha^*)
	\nn\\
	& \ge \ol{W}_t^M(x) = \sup_{\alpha\in\cA^M[0,t]} \wh{W}_t^\alpha(x)
	= \sup_{\alpha\in\cA^M[0,t]} J_t(x, u^\alpha, \alpha) = \sup_{\alpha\in\cA^M[0,t]} J_t(x, u^*, \alpha)
	\ge J_t(x, u^*, \alpha^*)\,,
	\nn
\end{align}
in which the first inequality is immediate by \er{eq:value-W-M}, and the last inequality follows by \er{eq:u-alpha-star}.
Hence, $\ol{W}_t^M(x) = J_t(x, u^*, \alpha^*)$, and uniqueness via Theorem \ref{thm:inf-sup-swap-M} and Lemma \ref{lem:u-star-is-optimal-in-cost-J-hat} imply that $u^* = u^{M*}$ and $\alpha^* = \alpha^{M*}$ as required.
\end{proof}

\begin{remark}
Theorem \ref{thm:optimal-inputs} implies that the unique optimal inputs $u^{M*}$ and $\alpha^{M*}$ of Theorem \ref{thm:inf-sup-swap-M} and Lemma \ref{lem:u-star-is-optimal-in-cost-J-hat} can be computed via the state feedback characterizations \er{eq:optimal-inputs}, which depend on the solution of TPBVP \er{eq:TPBVP}. Consequently, as expected, a shooting method applied to TPBVP \er{eq:TPBVP} will yield numerical approximations of these optimal inputs in specific examples.
\end{remark}

\begin{remark}
Recent work \cite{DC1:17} by the authors has further generalized the approach described in this paper to include linear time-varying dynamics, and convex constraints defined by the intersection of a finite collection of $p\in\N$ ellipses. The latter generalization involves an increase in the dimension of the range of the actions of the barrier penalty negotiating player, i.e. $\alpha_s\in\R^p$, $s\in[0,t]$. Crucially, the dimension of the DRE \er{eq:DRE}, or equivalently the DREs \er{eq:DRE-P}, \er{eq:DRE-Q}, \er{eq:DRE-R}, does not change, so that the dimension of the dynamics underlying the TPBVP involved does not increase beyond that presented here. 
The interested reader is referred to \cite{DC1:17} for some preliminary details and examples.
\end{remark}


\section{Illustrative example}
\label{sec:examples}

In illustrating an application of Theorems \ref{thm:game}, \ref{thm:inf-sup-swap-M}, and \ref{thm:optimal-inputs} the approximate  solution of a state constrained regulator problem \er{eq:value-W} via the approximate problem \er{eq:value-W-M} and corresponding game \er{eq:game-W-M}, a simple example is considered. The linear dynamics \er{eq:dynamics} and barrier \er{eq:barrier} are specified by 
\begin{gather}
	A \doteq \left[ \ba{cc}
			-1 & 2 \\ -1 & 1
		\ea \right],
	\quad
	B \doteq \left[ \ba{c} 1 \\ 0 \ea \right],
	\qquad
	\begin{gathered}
	\phi:[0,b^2)\rightarrow\R, \ b\doteq 3, \\
	\phi(\rho)
	\doteq - \log(1 - \rho/9), \ \rho\in[0,b^2),
	\end{gathered}
	\label{eq:example}
\end{gather}
while the running cost \er{eq:cost-I-bar} and its approximation \er{eq:cost-I-J-M} are specified by $t \doteq 4$, $\kappa\doteq 1$, $K\doteq 0.1$, and $M\doteq 50$. The sup-of-quadratics representation for $\Phi^M$ provided by Theorem \ref{thm:Phi-M-sup-of-quadratics} is illustrated in Figure \ref{fig:sup-of-quad-no-dip}. The trajectory defined by TPBVP \er{eq:TPBVP} is computed using a standard shooting method, which integrates the state dynamics \er{eq:dynamics} and FVP \er{eq:DRE} 
backward in time from the known terminal cost $\wh{P}_t^* = \wh{P}_t\in\Sigma^3$, and a candidate terminal state $\xi_t^* = \xi_t\in\R^2$. The error in the obtained initial state $|x - \xi_0^*|$ is subsequently iteratively minimized by varying $\xi_t$ using a Nelder-Mead simplex method. 

{\bf Case I:} {\em Terminal cost \er{eq:cost-Psi} with $z\doteq 0$ and $P_t \doteq I_2$.}
A pair of optimal trajectories for this terminal cost case is illustrated in Figure \ref{fig:case-I-traj}, corresponding to the barrier cost being active or inactive, i.e. included or excluded, in the cost \er{eq:cost-J-bar}, \er{eq:cost-I-J-M}. The circle included identifies the boundary of the state constraint imposed. An initial state of $x \doteq [\ba{cc} 1.6 & -1.6 \ea]'$ for dynamics \er{eq:dynamics} is assumed. Figures \ref{fig:case-I-alpha} and \ref{fig:case-I-control} illustrate the optimal inputs $\tilde\alpha^*$ and $\tilde u^*$ of \er{eq:optimal-inputs} respectively. By inspection of the unconstrained case, $\tilde\alpha^*$ attains its maximum value of $M=50$ where the constraint is violated. However, as $\tilde\alpha^*$ does not influence the control in the unconstrained case, the trajectory is not adjusted accordingly. In contrast, in the active constraint case, $\tilde\alpha^*$ attains a maximum of approximately $35$ as the trajectory approaches the constraint. By inspection, the state constraint is not violated, due to the intervention evident in the large actuated control $\tilde u^*$ that ensues.

{\bf Case II:} {\em Terminal cost \er{eq:cost-Psi} with $z\doteq [\, 1 \ \ 1 \,]'$ and $P_t \doteq 10\, I_2$.}
The terminal cost is adjusted in this case so as to encourage the trajectory to move towards the non-zero terminal state $\xi_t = z = [\, 1 \ \ 1 \,]'$, while respecting the state constraint. Figures \ref{fig:case-II-traj}, \ref{fig:case-II-alpha}, and \ref{fig:case-II-control} illustrate respectively the corresponding state trajectories, the optimal input $\tilde\alpha^*$, and the optimal control $\tilde u^*$ obtained, by solving TPBVP \er{eq:TPBVP}, with the constraint inactive and active.


Note that in both cases, the log barrier function $\phi$ of the form specified in \er{eq:example} yields
\begin{gather}
	\phi(\rho) = -\log(1 - \frac{\rho}{b^2}), \quad
	\phi'(\rho) = \frac{1}{b^2 - \rho}, \quad
	\phi''(\rho) = \frac{1}{(b^2 - \rho)^2},
	\nn\\
	a(\beta) = b^2\, \beta - \log(b^2\, \beta) - 1, \quad
	a'(\beta) = (\phi')^{-1}(\beta) = b^2 - \frac{1}{\beta}, \quad
	a''(\beta) = \frac{1}{\beta^2},
	\nn\\
	(\phi')^{-1}(\beta) = b^2 - \frac{1}{\beta}, \quad
	a^{-1}(\alpha) = -\frac{1}{b^2}\, W_{-1}(-\exp(-1-\alpha)), \quad
	(a')^{-1}(\rho) = \phi'(\rho) = \frac{1}{b^2 - \rho},
	\nn
\end{gather}
in which $W_{-1}$ is the $-1$ branch of the {\em Lambert-W function}. In practice, it was found that evaluating the inverses numerically was sufficiently accurate and fast, e.g. solving $\alpha = a(\beta)$ for $\beta$ given $\alpha$.


\section{Conclusions}
\label{sec:conc}
A sup-of-quadratics representation is developed for a class of convex barrier functions of interest in implementing state constraints in linear regulator problems. Using this representation, an equivalent unconstrained two player linear quadratic game is constructed. By demonstrating equivalence of its upper and lower values, an approach to computation is presented, and illustrated by example.

\begin{figure}[h]
\vspace{-1mm}
\begin{center}
\begin{subfigure}[t]{0.45\textwidth}
\psfrag{x1}{$x_1$}
\psfrag{x2}{$x_2$}
\includegraphics[width=\textwidth,height=60mm]{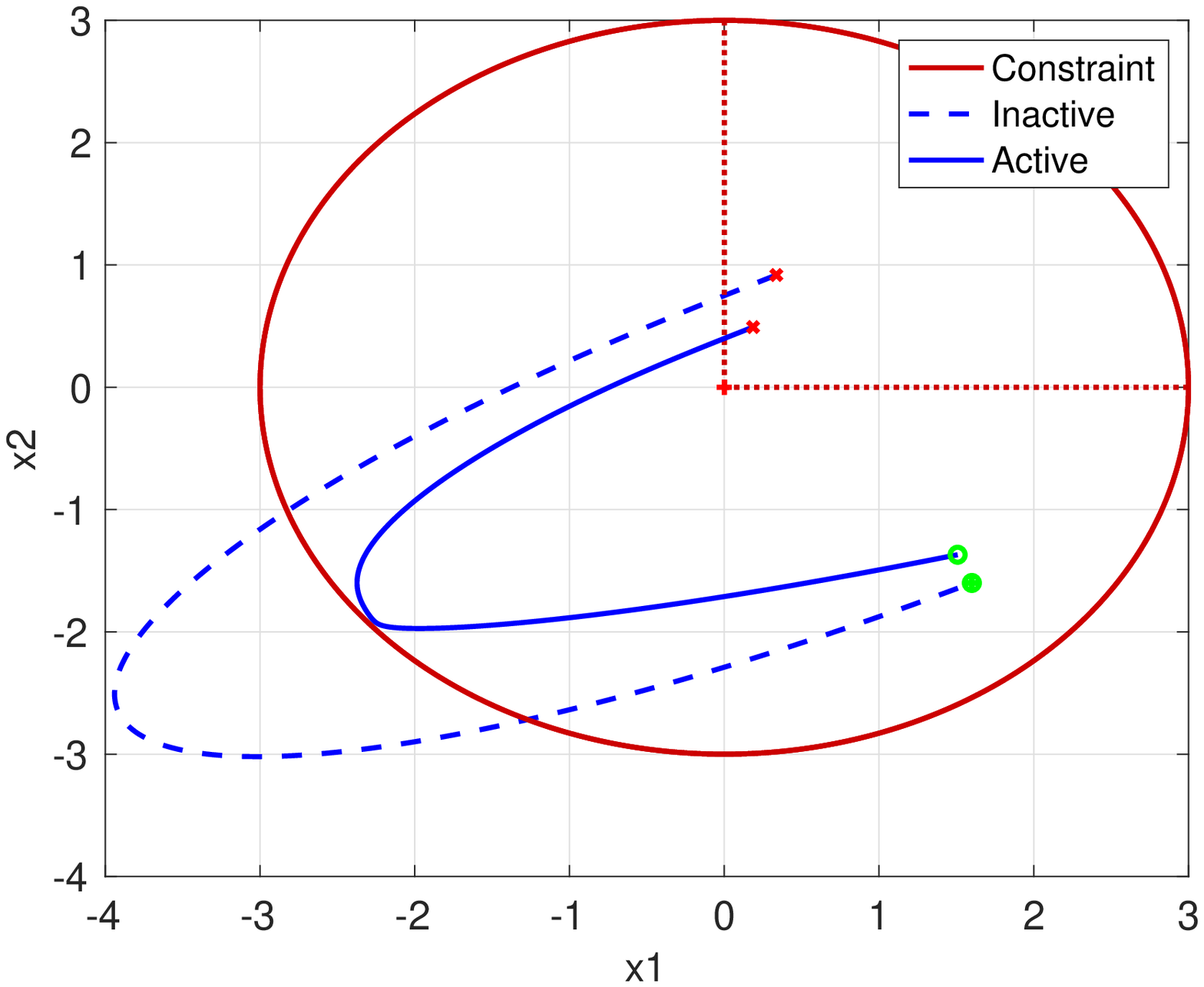}
\vspace{-6mm}
\caption{Case I.}
\label{fig:case-I-traj}
\end{subfigure}
\quad
\begin{subfigure}[t]{0.45\textwidth}
\psfrag{x1}{$x_1$}
\psfrag{x2}{$x_2$}
\includegraphics[width=\textwidth,height=60mm]{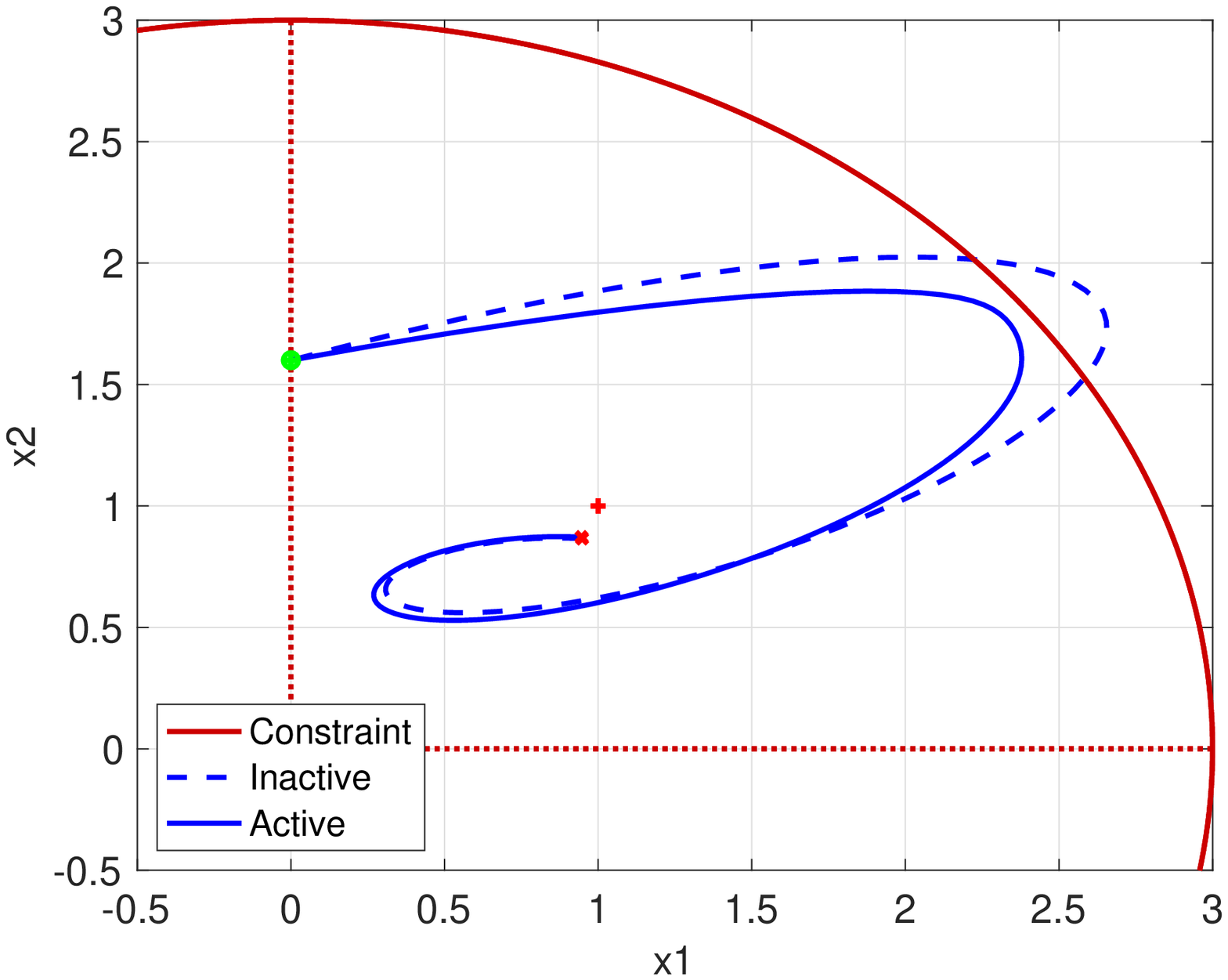}
\vspace{-6mm}
\caption{Case II.}
\label{fig:case-II-traj}
\end{subfigure}
\vspace{-3mm}
\caption{State trajectory with state constraint active and inactive (Cases I and II).}
\label{fig:state-traj}
\end{center}
\vspace{-3mm}
\end{figure}

\begin{figure}[h]
\begin{center}
\vspace{-2mm}
\begin{subfigure}[t]{0.45\textwidth}
\psfrag{s}{$s$}
\psfrag{alpha}{$\tilde\alpha_s^*$}
\includegraphics[width=\textwidth,height=60mm]{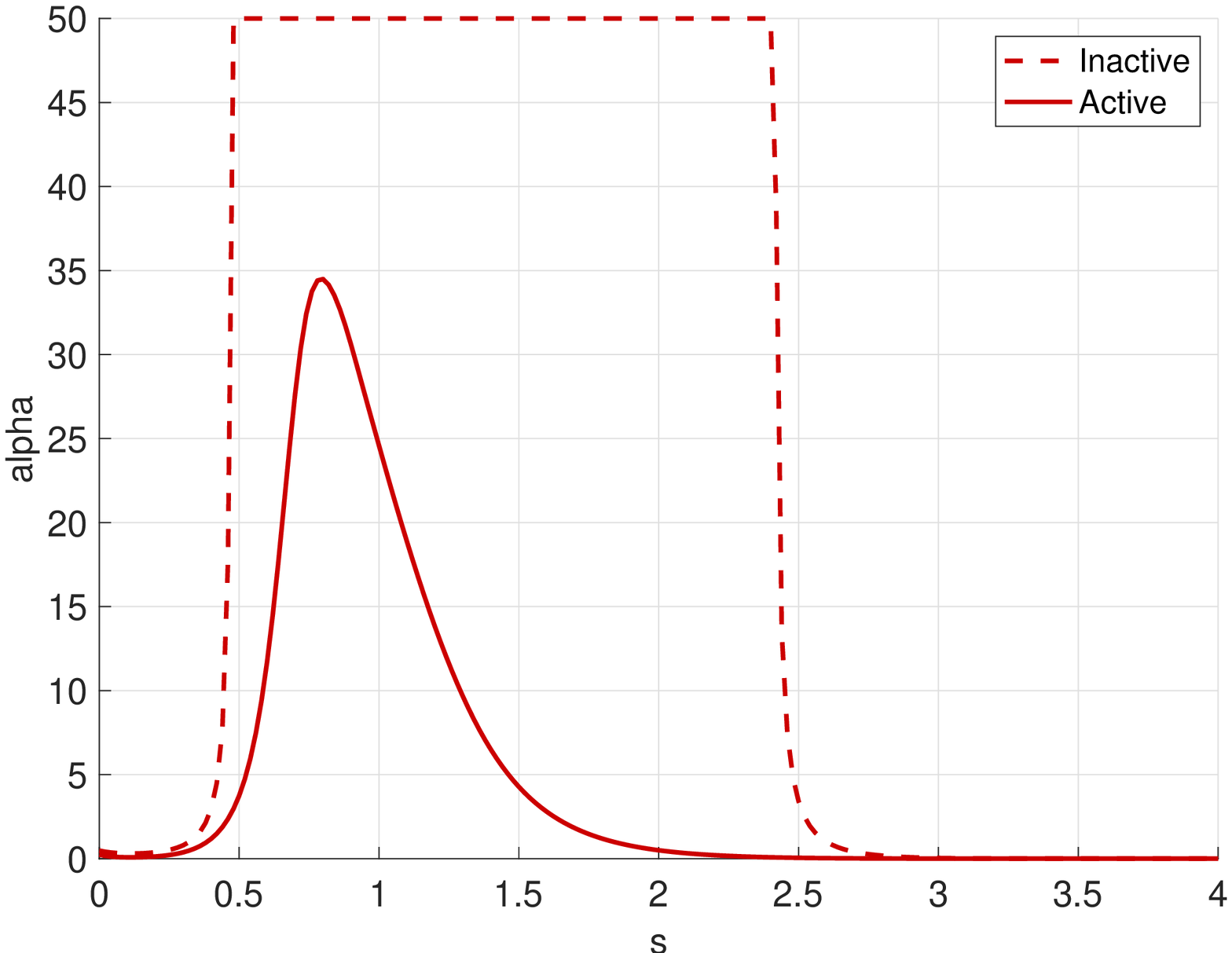}
\vspace{-6mm}
\caption{Case I.}
\label{fig:case-I-alpha}
\end{subfigure}
\quad
\begin{subfigure}[t]{0.45\textwidth}
\psfrag{s}{$s$}
\psfrag{alpha}{$\tilde\alpha_s^*$}
\includegraphics[width=\textwidth,height=60mm]{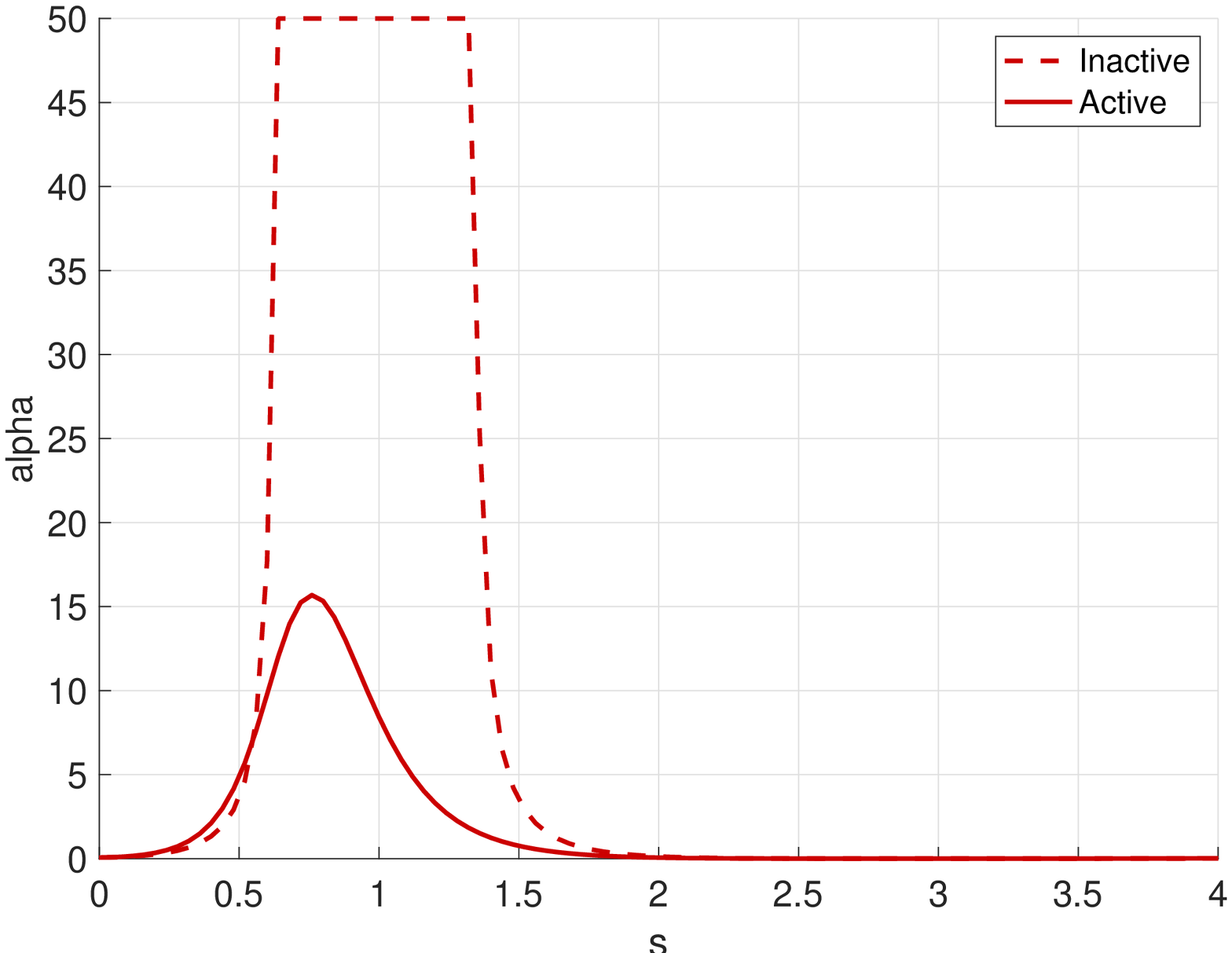}
\vspace{-6mm}
\caption{Case II.}
\label{fig:case-II-alpha}
\end{subfigure}
\vspace{-3mm}
\caption{Optimal input $\tilde\alpha^*$ (Cases I and II).}
\label{fig:alpha}
\end{center}
\vspace{-3mm}
\end{figure}

\begin{figure}[h]
\begin{center}
\vspace{-2mm}
\begin{subfigure}[t]{0.45\textwidth}
\psfrag{s}{$s$}
\psfrag{u}{$\tilde u_s^*$}
\includegraphics[width=\textwidth,height=60mm]{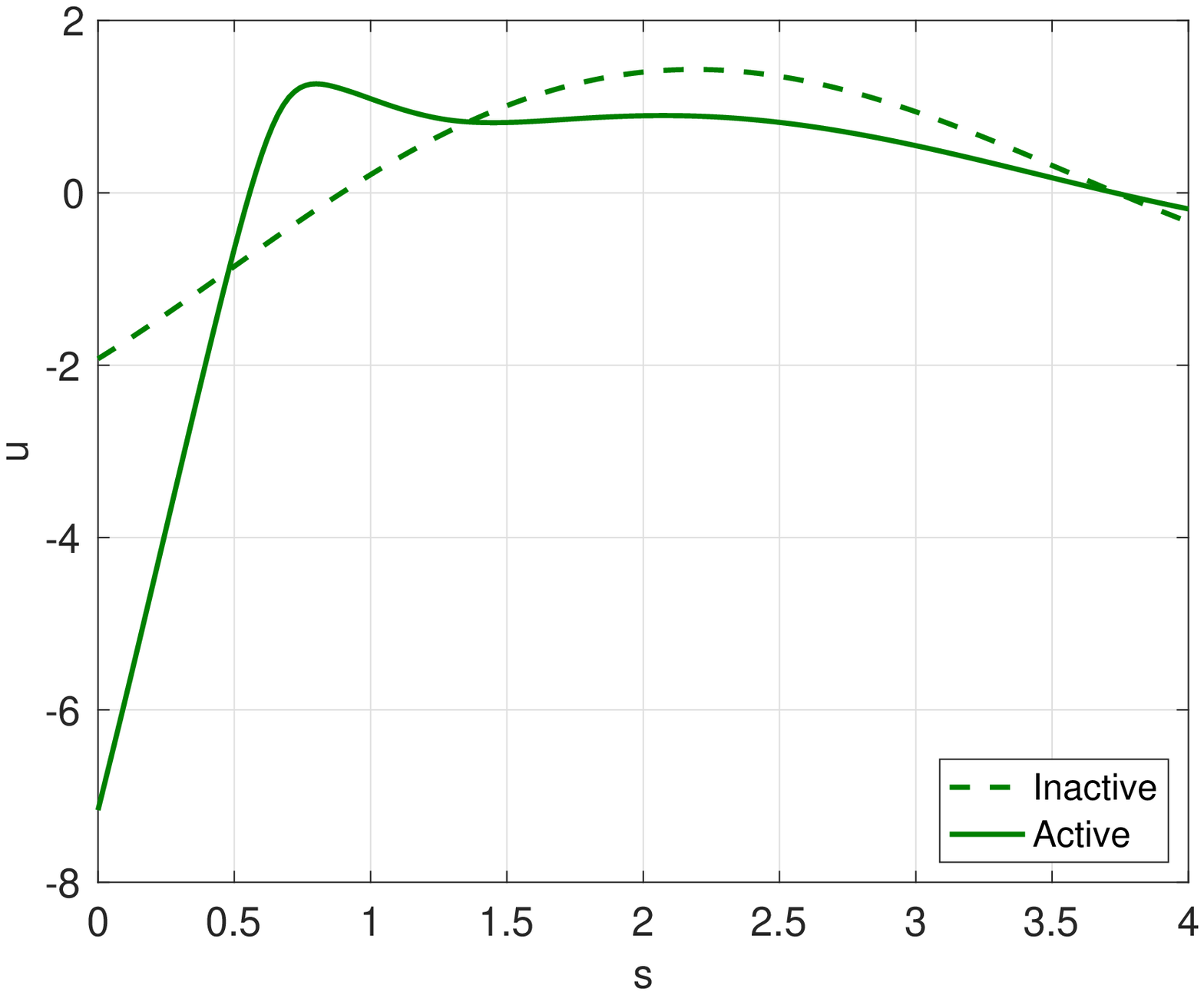}
\vspace{-6mm}
\caption{Case I.}
\label{fig:case-I-control}
\end{subfigure}
\quad
\begin{subfigure}[t]{0.45\textwidth}
\psfrag{s}{$s$}
\psfrag{u}{$\tilde u_s^*$}
\includegraphics[width=\textwidth,height=60mm]{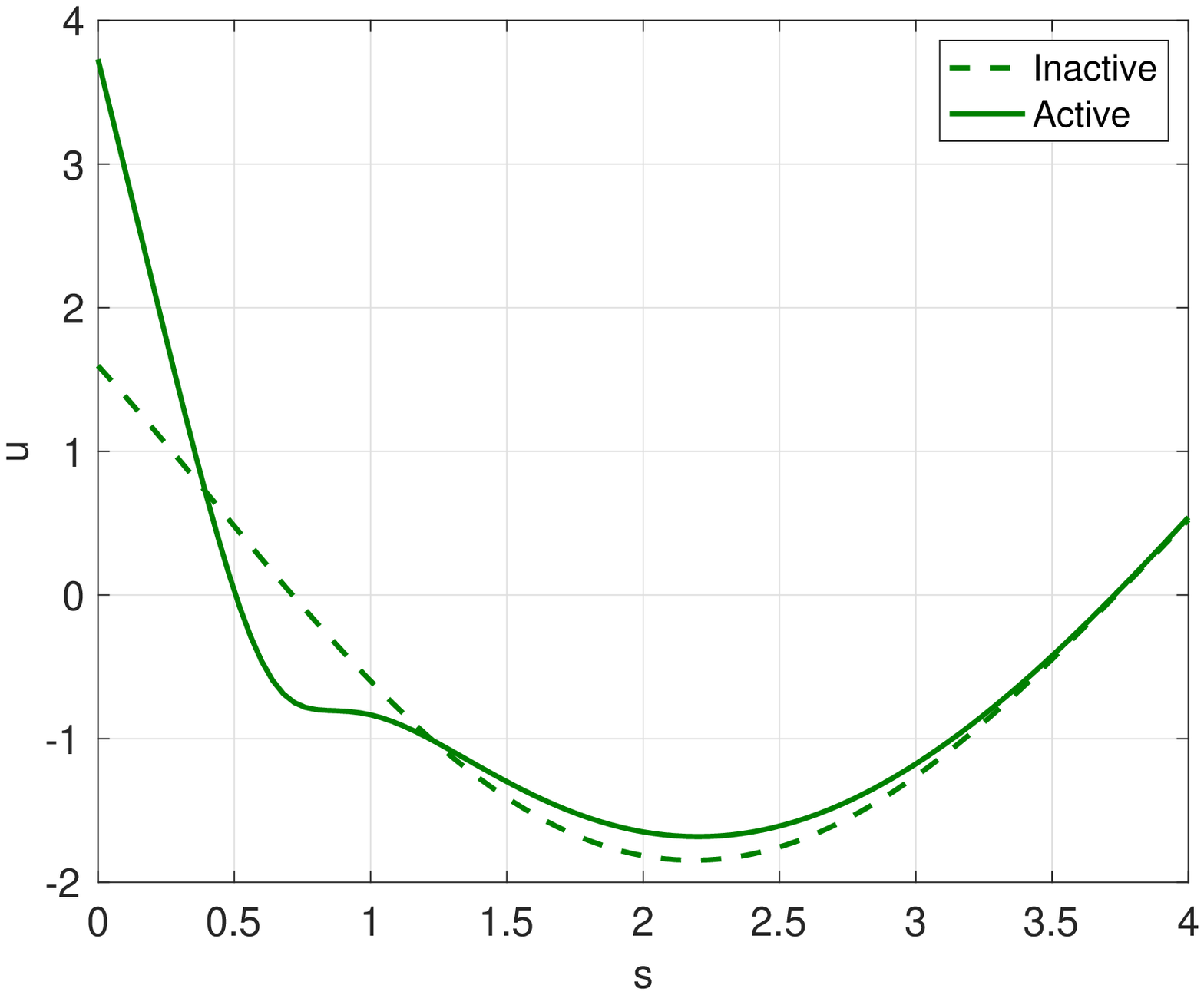}
\vspace{-6mm}
\caption{Case II.}
\label{fig:case-II-control}
\end{subfigure}
\vspace{-3mm}
\caption{Optimal control $\tilde u^*$ (Cases I and II).}
\label{fig:control}
\end{center}
\vspace{-3mm}
\end{figure}

\clearpage




\appendix


\section{Some useful properties of the barrier and its dual}
\label{app:properties}

\begin{lemma}
\label{lem:a-properties}
Given $\phi$ satisfying \er{eq:phi-properties-1-2}, the function $a$ of \er{eq:exact-a} is well-defined, differentiable, and strictly increasing, and has a well-defined, differentiable, and strictly increasing first derivative $a'$ and inverse $a^{-1}$, and a well-defined, strictly positive second derivative $a''$, satisfying
\begin{align}
	& a' : \R_{\ge \phi'(0)} \rightarrow [0,b^2), 
	\quad a'(\beta) = (\phi')^{-1}(\beta),
	\label{eq:exact-a-dot}
	\\
	& a'' : \R_{\ge \phi'(0)}\rightarrow\R_{>0},
	\quad a''(\beta) = \frac{1}{\phi''\circ (\phi')^{-1}(\beta)}.
	\label{eq:a-ddot-positive}
	\\
	& a^{-1} : \R_{\ge -\phi(0)} \rightarrow\R_{\ge \phi'(0)},
	\quad a^{-1}(-\phi(0)) = \phi'(0),
	\label{eq:exact-a-inv}
	\\
	& (a^{-1})' : \R_{\ge -\phi(0)} \rightarrow\ol{\R}_{>1/b^2}^+,
	\ (a^{-1})'(\alpha) = \frac{1}{a'\circ a^{-1}(\alpha)},
	\label{eq:exact-a-inv-dot}
	\\
	& \lim_{\alpha\uparrow\infty} (a^{-1})'(\alpha) = 1/b^2,
	\
	\lim_{\alpha\downarrow -\phi(0)} (a^{-1})'(\alpha) = +\infty.
	\label{eq:exact-a-inv-dot-limits}
	\\[-5mm]
	\nn
\end{align}
\end{lemma}
%
%
\begin{proof} 
By inspection of \er{eq:exact-a}, and the properties of $(\phi')^{-1}$ provided by \er{eq:phi-properties-1-2} {\em (v)}, it is evident that $a$ is well-defined on $\R_{\ge\phi'(0)}$. Note further that 
$
	[(\phi')^{-1}]'(\beta)
	= 1/[\phi''\circ(\phi')^{-1}(\beta)]
$
for all $\beta\in\R_{\ge\phi'(0)}$, in which the denominator is strictly positive by \er{eq:phi-properties-1-2} {\em (i)}, {\em (v)}. Hence, $a$ is differentiable by inspection of \er{eq:exact-a}, and the chain rule yields
$
	a'(\beta)
	= (\phi')^{-1}(\beta) + \beta\, [(\phi')^{-1}]'(\beta) - [\phi\circ(\phi')^{-1}(\beta)] \, [(\phi')^{-1}]'(\beta) = (\phi')^{-1}(\beta)
$
for all $\beta\in\R_{\ge\phi'(0)}$. Consequently, $a' = (\phi')^{-1}$ is well-defined and strictly increasing, with $a':\R_{\ge\phi'(0)}\rightarrow[0,b^2)$, by \er{eq:phi-properties-1-2} {\em (v)}. That is, \er{eq:exact-a-dot} holds. As $a'(\phi'(0)) = 0$ (by substitution), the strict increase property of $a'$ implies that $a'(\beta)\in\R_{>0}$ for all $\beta\in(\phi'(0),\infty)$. Hence, $a$ is also strictly increasing, and so \er{eq:exact-a} implies that $a(\beta) \ge a(\phi'(0)) = -\phi(0)$ for all $\beta\in[\phi'(0),\infty)$. By the same strict increase property of $a'$, note further that there exists an $\eps>0$ and $\beta_\eps>\phi'(0)$ such that
$a'(\beta) \ge \eps>0$ for all $\beta\ge\beta_\eps$. Consequently, $\lim_{\beta\rightarrow\infty} a(\beta) \ge \lim_{\beta\rightarrow\infty} [ (\beta - \beta_\eps)\, \eps + a(\beta_\eps)] = \infty$. Hence, $a(\beta)\in[-\phi(0),\infty)$ for all $\beta\in[\phi'(0),\infty)$, which confirms the range of $a$ specified in \er{eq:exact-a}. 

By inspection of \er{eq:exact-a-dot} and \er{eq:phi-properties-1-2} {\em(i)}, $a'$ is differentiable with derivative $a''$ given by
$
	a''(\beta)
	= [(\phi')^{-1}]'(\beta) = 1/[\phi''\circ(\phi')^{-1}(\beta)]
$
for all $\beta\in\R_{\ge\phi'(0)}$, which (as indicated previously) is strictly positive by \er{eq:phi-properties-1-2} {\em (i)}, {\em (v)}. Hence, \er{eq:a-ddot-positive} holds.

As $a$ is strictly increasing, the existence of its strictly increasing inverse $a^{-1}$, with domain and range specified by \er{eq:exact-a-inv}, follows immediately from \er{eq:exact-a}. The chain rule and \er{eq:exact-a-dot} subsequently imply that $a^{-1}$ is also differentiable, with derivative \er{eq:exact-a-inv-dot}. There, the range of this derivative follows by \er{eq:phi-properties-1-2} {\em (v)}, \er{eq:exact-a-dot}. The two limits in \er{eq:exact-a-inv-dot-limits} follow directly from \er{eq:exact-a-inv-dot}, with
\begin{align}
	\lim_{\alpha\uparrow\infty} (a^{-1})'(\alpha)
	& = \lim_{\alpha\uparrow\infty} \frac{1}{a'\circ a^{-1}(\alpha)}
	= \frac{1}{\lim_{\alpha\uparrow\infty} (\phi')^{-1}\circ a^{-1}(\alpha)}
	= \frac{1}{\lim_{\beta\uparrow\infty} (\phi')^{-1}(\beta)} = \frac{1}{b^2},
	\nn\\
	\lim_{\alpha\downarrow -\phi(0)} (a^{-1})'(\alpha)
	& = \frac{1}{\lim_{\beta\downarrow \phi'(0)} (\phi')^{-1}(\beta)}
	= +\infty,
	\nn
\end{align}
These limits, along with the fact that $(a^{-1})'$ is decreasing, confirm the range in \er{eq:exact-a-inv-dot}.
\end{proof}


\begin{lemma}
\label{lem:gamma-properties}
Given $\rho\in\R_{\ge 0}$, and $a^{-1}$ as per \er{eq:exact-a-inv}, the map $\gamma_\rho:\R_{\ge -\phi(0)}\rightarrow\R$ defined by
\begin{align}
	\gamma_\rho(\alpha)
	& \doteq a^{-1}(\alpha)\, \rho - \alpha,
	\quad \rho\in\R_{\ge 0},
	\label{eq:gamma}
\end{align}
is twice differentiable with derivatives $\gamma_\rho':\R_{>-\phi(0)} \rightarrow(\rho/b^2-1,\infty)$ and $\gamma_\rho'':\R_{>-\phi(0)} \rightarrow \R_{<0}$ given by
\begin{align}
	\gamma_\rho'(\alpha)
	= \frac{\rho}{a'\circ a^{-1}(\alpha)} - 1 = \frac{\rho}{\hat\rho(\alpha)} - 1,
	\quad
	\gamma_\rho''(\alpha)
	= -\rho\, \frac{ a''\circ a^{-1}(\alpha)}{[a'\circ a^{-1}(\alpha)]^3} = \frac{-\rho}{\phi''\circ\hat\rho(\alpha)\, [\hat\rho(\alpha)]^3}
	\label{eq:gamma-derivatives}
\end{align}
with $\hat\rho$ as per \er{eq:Phi-M-from-A-and-rho-M}. If $\rho\in\R_{\ge b^2}$, then $\alpha\mapsto\gamma_\rho(\alpha)$ is
strictly increasing with $\lim_{\alpha\rightarrow\infty}\gamma_\rho(\alpha) = + \infty$.
\end{lemma}

\if{false}

\begin{lemma}
\label{lem:gamma-properties}
A function $\gamma_\rho:\R_{\ge -\phi(0)}\rightarrow\R$ defined by
\begin{align}
	\gamma_\rho(\alpha)
	& \doteq a^{-1}(\alpha)\, \rho - \alpha,
	\quad \rho\in\R_{\ge 0},
	\label{eq:gamma}
\end{align}
with $a^{-1}$ as per \er{eq:exact-a-inv}, satisfies the following properties:
\begin{enumerate}[\it (i)] \itsep
\item $\gamma_\rho$ is differentiable for all $\rho\in\R_{\ge 0}$, and twice differentiable for all $\rho\in\R_{>0}$, with 
$\gamma_0'(\alpha) = -1$, $\gamma_0''(\alpha) = 0$, and
\begin{align}
	& \gamma_\rho':\R_{>-\phi(0)} \rightarrow(\rho/b^2-1,\infty),
	&& \gamma_\rho'(\alpha) = \frac{\rho}{a'\circ a^{-1}(\alpha)} - 1, && \rho\in\R_{\ge 0},
	\label{eq:gamma-dot}
	\\
	& \gamma_\rho'':\R_{>-\phi(0)} \rightarrow \R_{<0}, 
	&& \gamma_\rho''(\alpha) = - \rho\, \frac{ a''\circ a^{-1}(\alpha)}{[a'\circ a^{-1}(\alpha)]^3}, && \rho\in\R_{>0},
	\label{eq:gamma-ddot}
\end{align}
for all $\alpha\in\R_{>-\phi(0)}$;

\item Derivatives $\gamma_\rho'$, $\gamma_\rho''$ of \er{eq:gamma-dot}, \er{eq:gamma-ddot} satisfies the limits
\begin{align}
	\lim_{\alpha\downarrow -\phi(0)} \gamma_\rho'(\alpha) & = \infty\,,
	&& \rho\in\R_{>0}\,,
	\label{eq:gamma-dot-limit-low}
	\\
	\lim_{\alpha\uparrow\infty} \gamma_\rho'(\alpha) & = \ts{\frac{\rho}{b^2}} - 1\,,
	&& \rho\in\R_{\ge b^2}\,,
	\label{eq:gamma-limit-high}
	\\
	\lim_{\alpha\downarrow -\phi(0)} \gamma_\rho''(\alpha) & = -\infty\,,
	&& \rho\in\R_{>0}\,;
	\label{eq:gamma-ddot-limit-low}
\end{align}	

\item $\gamma_\rho$ is strictly decreasing if $\rho = 0$, see (i); 

\item $\gamma_\rho$ is strictly concave if $\rho\in\R_{>0}$;

\item $\gamma_\rho$ attains its maximum in $\R_{>-\phi(0)}$ if $\rho\in(0,b^2)$, in which case the maximum is finite; and

\item $\gamma_\rho$ is strictly increasing if $\rho\in\R_{\ge b^2}$, and satisfies $\lim_{\alpha\rightarrow\infty} \gamma_\rho(\alpha) = +\infty$.
\end{enumerate} 
\end{lemma}

\begin{proof}
The stated properties follow (with some calculation) from \er{eq:phi-properties-1-2}, \er{eq:exact-a}, \er{eq:gamma}, Lemma \ref{lem:a-properties}, and properties of monotone and differentiable functions. The details are omitted.
\end{proof}

\fi

\if{false}

\begin{proof}
{\em (i)} Given any fixed $\rho\in\R_{\ge 0}$, \er{eq:gamma} and Lemma \ref{lem:a-properties} immediately imply that $\gamma_\rho(\cdot)$ is differentiable, with derivative $\gamma_\rho'$ as per \er{eq:gamma-dot}. The $\rho=0$ case is immediate by inspection. In the $\rho\in\R_{>0}$ case, the range of $\gamma_\rho'$ follows by Lemma \ref{lem:a-properties}. The assertions regarding $\gamma_\rho''$ follow similarly.

{\em (ii)} With $\rho\in\R_{>0}$, \er{eq:gamma-dot-limit-low} follows by \er{eq:exact-a-inv-dot-limits}, \er{eq:gamma-dot}. With $\rho\in\R_{>b^2}$, \er{eq:gamma-dot} implies that $\gamma_\rho(\cdot)$ is strictly increasing and grows asymptotically linearly. That is, \er{eq:gamma-limit-high} holds. Alternatively, with $\rho = b^2$, the change of variable $\alpha = a(\beta)$ with $\beta\in[\phi'(0),\infty)$, along with \er{eq:exact-a}, imply that
\begin{align}
	\gamma_\rho\circ a(\beta)
	& = \beta\, b^2 - a(\beta)
	= \beta \, \underbrace{[b^2 - (\phi')^{-1}(\beta)]}_{>0} + \phi\circ(\phi')^{-1}(\beta)
	\ge \phi\circ(\phi')^{-1}(\beta)
	\nn
\end{align}
for all $\beta\in[\phi'(0),\infty)\cap\R_{\ge 0}$. Using $a$ of \er{eq:exact-a} as a change of variable $\alpha \doteq a(\beta)$ via Lemma \ref{lem:a-properties},
\begin{align}
	\lim_{\alpha\uparrow\infty} \gamma_\rho(\alpha)
	& = \lim_{\beta\uparrow\infty} \gamma_\rho\circ a(\beta)
	\ge \lim_{\beta\uparrow\infty} \phi\circ(\phi')^{-1}(\beta) = +\infty,
	\nn
\end{align}
so that \er{eq:gamma-limit-high} again holds. Recalling Lemma \ref{lem:a-properties}, note that $a''\circ a^{-1}(\alpha) = 1/\phi''\circ a'\circ a^{-1}(\alpha)$, so that \er{eq:gamma-ddot} yields
\begin{align}
	\gamma_\rho''(\alpha)
	& = -\rho\, \frac{1}{\phi''(a'\circ a^{-1}(\alpha))\, [ a'\circ a^{-1}(\alpha) ]^3}\,.
	\nn
\end{align}
Note in particular that $\lim_{\alpha\downarrow-\phi(0)} a'\circ a^{-1}(\alpha) = 0$, and consequently, $\limsup_{\alpha\downarrow-\phi(0)} \phi''(a'\circ a^{-1}(\alpha)) = \limsup_{\rho\downarrow 0} \phi''(\rho)\in\R$. Hence, taking the required limit of $\gamma_\rho''(\alpha)$ above yields \er{eq:gamma-ddot-limit-low}.

{\em (iii), (iv)} Immediate by inspection of \er{eq:gamma-dot}, \er{eq:gamma-ddot}.

{\em (v)} Fix arbitrary $\rho\in(0,b^2)$. By inspection of \er{eq:gamma-limit-high}, $\lim_{\alpha\uparrow\infty} \gamma_\rho(\alpha) = -(1 - \frac{\rho}{b^2}) < 0$. Continuity of $\gamma_\rho'$ along with \er{eq:gamma-dot-limit-low} implies that there exists an $\bar\alpha\in\R_{>-\phi(0)}$ such that $\gamma'(\bar\alpha) = 0$. A maximum is achieved at $\bar\alpha$ by {\em (iv)}.

{\em (vi)} Strict monotonicity is immediate by inspection of \er{eq:gamma-dot}. The asserted infinite limit follows as per the proof of assertion {\em (ii)}.
\end{proof}

\fi

\begin{lemma}
\label{lem:inequalities}
Given $M\in\R_{\ge -\phi(0)}$,
\begin{align}
	0 & \ge \phi(0) + \phi'(0)\, \rho - \Phi(\rho)
	&& \forall \, \rho\ge 0,
	\label{eq:inequality-1}
	\\
	0 & \ge \phi(0) + \phi'(0)\, \rho - a^{-1}(M)\, \rho + M
	&& \forall \, \rho\ge\hat\rho(M),
	\label{eq:inequality-2}
	\\
	0 & \ge \lambda_+^M(\beta)
	&& \forall \, \beta\le a^{-1}(-\phi(0)),
	\label{eq:inequality-3}
	\\
	0 & \ge \lambda_+^M(\beta) - a(\beta) - \phi(0)
	&& \forall \, \beta\in[a^{-1}(-\phi(0)), a^{-1}(M)],
	\label{eq:inequality-4}
\end{align}
in which $a$, $a^{-1}$, $\hat\rho(M)$ are given by \er{eq:exact-a}, \er{eq:exact-a-inv}, \er{eq:Phi-M-from-A-and-rho-M}, and
$\lambda_+^M:\R\rightarrow\R$ is defined by
\begin{align}
	\lambda_+^M(\beta)
	& \doteq M + \phi(0) - (a^{-1}(M) - \beta)\, \hat\rho(M).
	\label{eq:lambda-M-p}
\end{align} 
\end{lemma}

\begin{proof}
{\em [Inequality \er{eq:inequality-1}]} See Corollary \ref{cor:Phi-M-bound}.

%
%
{\em [Inequality \er{eq:inequality-2}]} Fix $M\in\R_{\ge -\phi(0)}$. As $a^{-1}(M)\ge a^{-1}(-\phi(0))$, 
\begin{align}
	\phi(0) + \phi'(0)\, \rho - a^{-1}(M)\, \rho + M
	& 
	= M + \phi(0) - [ a^{-1}(M) - a^{-1}(-\phi(0))]\, \rho
	\nn\\
	& \le M + \phi(0) -  [ a^{-1}(M) - a^{-1}(-\phi(0)) ] \, \hat\rho(M)
	= \lambda_+^M(a^{-1}(-\phi(0))),
	\nn
\end{align}
for any $\rho\ge\hat\rho(M) = a'\circ a^{-1}(M)\in[0,b^2)$, see \er{eq:Phi-M-from-A-and-rho-M}, \er{eq:exact-a-dot},
where $\lambda_+^M$ is as per \er{eq:lambda-M-p}. Hence, inequality \er{eq:inequality-2} is a special case of inequality \er{eq:inequality-3}. 

{\em [Inequality \er{eq:inequality-3}]} There are two cases to consider, namely, {\em (i)} $M = -\phi(0)$, and {\em (ii)} $M>-\phi(0)$.

{\em (i)} Fix $M = -\phi(0)$. By inspection of the definition of $\lambda_+^M$ and \er{eq:Phi-M-from-A-and-rho-M}, \er{eq:exact-a-inv},
\begin{align}
	\lambda_+^M(\beta) 
	& = -(a^{-1}(-\phi(0)) - \beta) \, \hat\rho(-\phi(0)) 
	= -(a^{-1}(-\phi(0)) - \beta) \, (\phi')^{-1}\circ a^{-1}(-\phi(0))
	\nn\\
	& = -(a^{-1}(-\phi(0)) - \beta) \, (\phi')^{-1}\circ\phi'(0) = 0\,,
	\label{eq:lambda-M-p-zero}
\end{align}
so that \er{eq:inequality-3} holds.

{\em (ii)} Alternatively, fix $M>-\phi(0)$. Differentiating \er{eq:lambda-M-p}, $(\lambda_+^M)'(\beta) = \hat\rho(M)\in(0,b^2)$ for all $\beta\in\R$, see \er{eq:Phi-M-from-A-and-rho-M}. Hence, $\lambda_+^M(\beta)$ is strictly increasing in $\beta\in\R$. In particular,
$\beta \le a^{-1}(-\phi(0))$ implies that $\lambda_+^M(\beta) \le \lambda_+^M(a^{-1}(-\phi(0))) = \mu(M)$,
where $\mu\in C(\R_{\ge -\phi(0)};\R)$ is defined by
$\mu(\alpha) \doteq \alpha + \phi(0) - (a^{-1}(\alpha) - a^{-1}(-\phi(0)))\, \hat\rho(\alpha)$ for all $\alpha\in\R_{\ge -\phi(0)}$.	
Recall the derivative of $a^{-1}$ from \er{eq:exact-a-inv-dot}, and similarly differentiate $\hat\rho$ via \er{eq:Phi-M-from-A-and-rho-M}, yielding
$(a^{-1})'(\alpha) = 1/(a'\circ a^{-1}(\alpha)) = 1/\hat\rho(\alpha) > 0$,
$\hat\rho'(\alpha)
	= (a''\circ a^{-1}(\alpha))/(a'\circ a^{-1}(\alpha)) = 1/([\phi''\circ\hat\rho(\alpha)]\, \hat\rho(\alpha)) > 0$,
for all $\alpha\in\R_{>-\phi(0)}$. Note that neither derivative is defined at $\alpha = -\phi(0)$, as $a'\circ a^{-1}(-\phi(0)) = 0$. Hence, by inspection of its definition above, $\mu$ is differentiable on $\R_{>-\phi(0)}$, with the product rule yielding
$
	\mu'(\alpha)
	= 
	[a^{-1}(\alpha) - a^{-1}(-\phi(0))]\, ([\phi''\circ\hat\rho(\alpha)]\, \hat\rho(\alpha))^{-1}
	< 0
$,
for all $\alpha\in\R_{>-\phi(0)}$. That is, $\mu$ is continuous on $\R_{\ge -\phi(0)}$ and strictly decreasing on $\R_{>-\phi(0)}$, so that
$
	0 = \mu(-\phi(0)) 
	> \mu(M) = \lambda_+^M(a^{-1}(-\phi(0))) \ge \lambda_+^M(\beta)
$
for all $\beta\le a^{-1}(-\phi(0))$, as required by \er{eq:inequality-3}.

{\em [Inequality \er{eq:inequality-4}]} Again there are two cases to consider, {\em (i)} $M = -\phi(0)$, and {\em (ii)} $M>-\phi(0)$.

{\em (i)} Fix $M = -\phi(0)$. Recall in this case that $\lambda_+^M(\beta) = 0$ for all $\beta\in\R$, see \er{eq:lambda-M-p-zero}. Hence, recalling that $a$ is strictly increasing, the right-hand side of \er{eq:inequality-4} for $\beta\in[a^{-1}(-\phi(0)), a^{-1}(M)]$ is 
$
	\lambda_+^M(\beta) - a(\beta) - \phi(0)
	= -a(\beta) - \phi(0)
	\le a\circ a^{-1}(-\phi(0))) - \phi(0) = 0
$,
so that \er{eq:inequality-4} holds.eq:Gamma-p-M-and-beta-p-star-M

{\em (ii)} Alternatively, fix $M > -\phi(0)$.
Define $\eta\in C([-\phi(0),M];\R)$ by $\eta(\alpha)\doteq \lambda_+^M\circ a^{-1}(\alpha) - \alpha$ for all $\alpha\in [-\phi(0),M]$.
By inspection, $\eta$ is differentiable on $(-\phi(0),M]$, with
\begin{align}
	\eta'(\alpha)
	& = (\lambda_+^M)'\circ a^{-1}(\alpha) \, (a^{-1})'(\alpha) - 1
	= \hat\rho(M)\, (a^{-1})'(\alpha) - 1
	= \frac{a'\circ a^{-1}(M)}{a'\circ a^{-1}(\alpha)} - 1 > 0\,,
	\nn
\end{align}
for all $\alpha\in (-\phi(0),M]$, in which the final inequality follows as $a'$, $a^{-1}$ are strictly increasing. Again note that differentiability is lost at $\alpha = -\phi(0)$, as $a'\circ a^{-1}(-\phi(0)) = (\phi')^{-1} (\phi'(0)) = 0$. Hence, $\mu$ is continuous on $[-\phi(0),M]$ and strictly increasing on $(-\phi(0),M]$, so that
$
	\eta(\alpha) \le \eta(M) = \lambda_+^M\circ a^{-1}(M) - M = \phi(0)
$
for all $\alpha\in[-\phi(0),M]$. 
Setting $\alpha = a(\beta)$ for $\beta\in[a^{-1}(-\phi(0)), a^{-1}(M)]$ yields $\lambda_+^M(\beta) - a(\beta) \le \phi(0)$, as required by \er{eq:inequality-4}.
\end{proof}


\if{false}

\subsection{Proof of Lemmas \ref{lem:chi} and \ref{lem:dom-value-W-bar}}
\label{app:chi-and-dom}

\begin{proof}{\lbrack Lemma \ref{lem:chi}\rbrack}
Fix $t\in\R_{\ge 0}$, $x\in\R^n$. 
{\em (i)} Fix $r,s\in[0,t]$, $u\in\cU[0,t]$. Define $\xi \doteq \chi(x,u)$. Recalling \er{eq:dynamics},
\begin{align}
	\xi_s - \xi_r
	& = e^{A\,s} \, [ I - e^{-A(s-r)} ] \left( x + \int_0^r e^{-A\, \sigma}\, B\, u_\sigma\, d\sigma \right) 
		+ e^{A\, s}\, \int_r^s e^{-A\, \sigma}\, B \, u_\sigma\, d\sigma\,.
	\label{eq:uniform-1}
\end{align}
By definition of the matrix exponential, and H\"{o}lder's inequality,
\begin{align}
	\left\| I - e^{-A(s-r)} \right\|
	& \le \|A\| \, e^{\|A\|\, t} \, |s-r|\,,
	\quad
	\left| \int_r^s e^{-A\, \sigma}\, B \, u_\sigma \, d\sigma \right|
	\le \sup_{s\in[0,t]} \| e^{-A\, s}\, B\| \, \|u\|_{\cU[0,t]} \, \sqrt{|s-r|}\,.
	\nn
\end{align}
Applying these inequalities in \er{eq:uniform-1}, via the triangle inequality,
\begin{align}
	| \xi_s - \xi_r |
	& \le \omega_{\|u\|_{\cU[0,t]}} (|s-r|)\,,
	\label{eq:uniform-2}
\end{align}
in which $\omega_U:\R_{\ge 0} \rightarrow\R_{\ge 0}$, $U\in\R_{\ge 0}$, is a modulus of continuity defined by
\begin{align}
	\omega_U(\delta)
	& \doteq
	e^{2\, \|A\| \, t} \, \|A\| \left(
		\|x\| + \sup_{s\in[0,t]} \| e^{-A\, s}\, B\| \, U  \, \sqrt{t}
	\right) \delta
	+ e^{\|A\| \, t} \sup_{s\in[0,t]} \| e^{-A\, s}\, B\| \, U \sqrt{\delta}\,,
	\qquad \delta\in\R_{\ge 0}\,.
	\nn
\end{align}
Hence, $\xi:[0,t]\rightarrow\R^n$ is uniformly continuous for fixed $u\in\cU[0,t]$ by \er{eq:uniform-2}.

{\em (ii)} This follows directly by \er{eq:dynamics} and H\"{o}lder's inequality, and the details are omitted.

{\em (iii)} Fix $\ol{U}\in\R_{\ge 0}$ and $\{ u_k \}_{k\in\N}\subset\cB_{\cU[0,t]}[0;\ol{U}]$ and define $\{\xi_k\}_{k\in\N}\subset C([0,t];\R^n)$ via {\em (i)} by $\xi_k \doteq  \chi(x,u_k)$ for all $k\in\N$. Applying \er{eq:uniform-2}, $|[\xi_k]_s - [ \xi_k ]_r| \le \omega_\ol{U}(|s-r|)$ for all $k\in\N$, so that $\{\xi_k\}_{k\in\N}$ is uniformly equicontinuous. Recalling \er{eq:dynamics}, and the triangle and H\"{o}lder's inequalities,
\begin{align}
	\|\xi_k\|_{C([0,t];\R^n)}
	& \le \sup_{s\in[0,t]} \| e^{A\, s} \| \left( \|x\| + \|B\| \, \sqrt{t}\, U \right).
	\nn
\end{align}
Hence, $\{\xi_k\}_{k\in\N}$ is both uniformly equicontinuous and uniformly bounded, so that the Arzela-Ascoli Theorem, e.g. \cite[p.295]{TL:80}, implies existence of a subsequence $\{\hat y_k\}_{k\in\N}\subset\{\xi_k\}_{k\in\N}$ and a $\bar y\in C([0,t];\R^n)$ such that $\lim_{k\rightarrow\infty} \sup_{s\in[0,t]} | [\hat y_k]_s - \bar y_s|=0$, i.e. $\hat y_k\rightarrow \bar y$ uniformly as $k\rightarrow\infty$. Let $\{\hat u_k\}_{k\in\N}\subset\{ u_k \}_{k\in\N}$ be the subsequence of inputs corresponding to $\{\hat y_k\}_{k\in\N}$. As $\{\hat u_k\}_{k\in\N}\subset\cB_{\cU[0,t]}[0;\ol{U}]\subset\cU[0,t]$ is bounded and $\cU[0,t]$ is a reflexive Banach space, \cite[Theorem III.10.6, p.177]{TL:80} implies that there exists a $\bar u\in\cU[0,t]$ and a weakly convergent subsequence $\{v_k\}_{k\in\N}\subset\{\hat u_k\}_{k\in\N}$ satisfying $v_k\weakly\bar u$. Define $\bar\xi \doteq \chi(x,\bar u)$ and $\{y_k\}_{k\in\N}\subset\{\hat y_k\}_{k\in\N}\subset\{\xi_k\}_{k\in\N}$ by $y_k \doteq \chi(x,v_k)$ for all $k\in\N$. Note as a consequence that $y_k\rightarrow\bar y$ uniformly as $k\rightarrow\infty$. It remains to show that $\bar y = \bar\xi$.

Applying the Riesz representation theorem for $\cU[0,t]$, weak convergence $v_k\weakly \bar u$ is equivalent to $\lim_{k\rightarrow\infty} \langle w,\, v_k \rangle_{\cU[0,t]} = \langle w,\, \bar u \rangle_{\cU[0,t]}$ for all $w\in\cU[0,t]$. Fix any $\bar x\in\R^n$, $s\in[0,t]$, and select a specific $w\in\cU[0,t]$ defined by $w_\sigma\doteq \chi^{[0,s]}\, B'\, e^{A'(s-\sigma)}\, \bar x$ for all $\sigma\in[0,t]$, in which $\chi^{I}:[0,t]\rightarrow\{0,1\}$ denotes the indicator for interval $I\subset[0,t]$. Hence,
\begin{align}
	\langle w,\, v_k \rangle_{\cU[0,t]}
	& = \int_0^t \chi_\sigma^{[0,s]} \, \langle B'\, e^{A'(s-\sigma)}\, \bar x,\, [v_k]_\sigma \rangle\, d\sigma
	= \left\langle \bar x,\, \int_0^s e^{A\, (s-\sigma)}\, B\, [v_k]_\sigma \, d\sigma \right\rangle,
	\nn\\
	\langle w,\, \bar u \rangle_{\cU[0,t]}
	& = \left\langle \bar x,\, \int_0^s e^{A\, (s-\sigma)}\, B\, \bar u_\sigma \, d\sigma \right\rangle\,,
	\nn
\end{align}
in which the inner products are on $\R^n$. As $\bar x\in\R^n$ is arbitrary, it follows that
\begin{align}
	& \int_0^s e^{A\, (s-\sigma)}\, B\, [v_k]_\sigma \, d\sigma \weakly \int_0^s e^{A\, (s-\sigma)}\, B\, \bar u_\sigma \, d\sigma\,,
	\nn
\end{align}
As weak and strong convergence coincide on finite dimensional spaces, this in turn implies that
\begin{align}
	\bar y_s = \lim_{k\rightarrow\infty} [y_k]_s
	& = e^{A\, s}\, x + \lim_{k\rightarrow\infty} \int_0^s e^{A\, (s-\sigma)}\, B\, [v_k]_\sigma\, d\sigma
	= e^{A\, s}\, x + \int_0^s e^{A\, (s-\sigma)}\, B\, \bar u_\sigma\, d\sigma = \bar \xi_s\,,
	\nn
\end{align}
in which the limits are pointwise in $\R^n$. As $s\in[0,t]$ is arbitrary, $\bar y = \bar \xi$, as required.
\end{proof}

\begin{proof} {\lbrack Lemma \ref{lem:dom-value-W-bar}\rbrack}
Fix $t\in\R_{\ge 0}$, $\bar x \doteq 0\in\R^n$. By inspection of \er{eq:value-W}, note that $\bar u\doteq 0\in\cU[0,t]$ is suboptimal in the definition of $\ol{W}_t(\bar x)$. With $\bar \xi \doteq \chi(\bar x, \bar u) = 0\in C([0,t];\R^n)$, see \er{eq:dynamics}, it immediately follows that 
\begin{align}
	& \ol{W}_t(\bar x)
	\le \bar J_t(\bar x, \bar u) = \bar I_t(\bar x, \bar u) + I_t^\kappa(\bar u) + \Psi([\bar \xi]_t)
	= \ts{\frac{\phi(0)}{2}} \, t + \demi\, \langle z,\, P_t\, z \rangle < \infty,
	\nn
\end{align}
so that $0=\bar x\in\dom\ol{W}_t$ as required.
\end{proof}

\fi


\section{Proof of Lemmas \ref{lem:exact-convex-duality} and \ref{lem:approx-convex-duality}}
\label{app:convex-duality-proofs}

\begin{proof} 
{\em [Lemma \ref{lem:exact-convex-duality}]}
The barrier function $\Phi$ of \er{eq:barrier} is closed and convex on $\R$ by \er{eq:barrier}, \er{eq:phi-properties-1-2}, \cite[(3.8), pp.15,17]{R:74}. Hence, there exists a one-to-one pairing between $\Phi$ and its Fenchel transform $\Ahat = \Phi^*$, as indicated by \er{eq:exact-Phi-and-A}, see \cite[Theorem 5, p.16]{R:74}. The objectives are to establish the explicit form \er{eq:exact-Phi-and-A} of the function $\Ahat$, its range, and the optimizers \er{eq:exact-beta-and-rho-star} attending the suprema in \er{eq:exact-Phi-and-A}. To this end, note by inspection of \er{eq:barrier} and the definition of $\Theta$ in \er{eq:exact-Phi-and-A} that
\begin{align}
	\Ahat(\beta) 
	& = \sup_{\rho\in[0,b^2)} \pi_\beta(\rho),
	\quad \pi_\beta(\rho) \doteq \beta\, \rho - \phi(\rho),
	\label{eq:A-pi}
\end{align}
for all $\beta\in\R$, $\rho\in[0,b^2)$. If $\beta\in\R_{\ge\phi'(0)}$, the supremum is attained at $\rho = \rho^* = (\phi')^{-1}(\beta)$, as $0 = \pi_\beta'(\rho^*) = \beta - \phi'(\rho^*)$.
Note by \er{eq:phi-properties-1-2} {\em (v)} that $\rho^*\in[0,b^2)$. The supremum is then
$	\pi_\beta(\rho^*) 
	= \beta\, (\phi')^{-1}(\beta) - \phi\circ(\phi')^{-1}(\beta) \doteq a(\beta),
$
as per \er{eq:exact-a}. Alternatively, if $\beta\in\R_{<\phi'(0)}$, \er{eq:phi-properties-1-2} {\em (iv)} implies that $\pi_\beta'(\rho) = \beta - \phi'(\rho) < \phi'(0) - \phi'(\rho) \le 0$ for $\rho\in[0,b^2)$. Hence, the supremum must be achieved at $\rho^* = 0$, and $\pi_\beta(\rho^*) = -\phi(0)$. Combining both of the above cases immediately yields the right-hand equations in \er{eq:exact-Phi-and-A} and \er{eq:exact-beta-and-rho-star}.

In order to demonstrate that the left-hand equality in \er{eq:exact-beta-and-rho-star},
holds, note by \er{eq:exact-Phi-and-A} that
\begin{align}
	\Phi(\rho) & = \max\{ \Gamma_-(\rho), \, \Gamma_+(\rho) \},
	\label{eq:Phi-from-Gamma-pm}
	\\
	\Gamma_-(\rho)
	& \doteq \sup_{\beta<\phi'(0)} \{ \beta\, \rho + \phi(0) \}
	= \left\{\ba{rl}
		+\infty,
		& \rho\in\R_{<0},
		\\
		\phi(0) + \phi'(0) \, \rho,
		& \rho\in\R_{\ge 0},
	\ea \right.
	\label{eq:Gamma-m-explicit}
	\\
	\Gamma_+(\rho)
	& \doteq \sup_{\beta\ge\phi'(0)} \chi_\beta(\rho),
	\quad \chi_\rho(\beta) \doteq \beta\, \rho - a(\beta),
	\label{eq:Gamma-p}
\end{align}
for all $\beta, \rho\in\R$. The supremum in \er{eq:Gamma-m-explicit} is achieved at
\begin{align}
	& \beta = \hat\beta_-^*(\rho) \doteq \left\{ \ba{rl}
	-\infty, & \rho\in\R_{<0},
	\\
	\phi'(0), & \rho\in\R_{\ge 0}.
	\ea \right.
	\label{eq:beta-m-explicit}
\end{align}
In considering $\Gamma_+$ of \er{eq:Gamma-p}, recall that $a$ of \er{eq:exact-a} is differentiable by Lemma \ref{lem:a-properties}, with derivative given by \er{eq:exact-a-dot}. Three cases are subsequently considered, {\em (i)} $\rho\in[0,b^2)$, {\em (ii)} $\rho\in\R_{<0}$, and {\em (iii)} $\rho\in\R_{\ge b^2}$.

{\em (i)} $\rho\in[0,b^2)$: As $a'(\beta) = (\phi')^{-1}(\beta)$ is well-defined for all $\beta\in\R_{\ge\phi'(0)}$, see \er{eq:exact-a-dot}, $\chi_\rho$ is differentiable with $\chi_\rho'(\beta) = \rho - (\phi')^{-1}(\beta)$ for all $\beta\in\R_{\ge\phi'(0)}$. Substituting $\beta = \beta^* \doteq \phi'(\rho)\in\R_{\ge\phi'(0)}$ yields $\chi_\rho'(\beta^*) = \rho - (\phi')^{-1}(\beta^*) = 0$. Hence, the supremum in \er{eq:Gamma-p} is attained at $\beta = \beta^*$, with 
$
	\chi_\rho(\beta^*) 
	= \rho\, \phi'(\rho) - a\circ\phi'(\rho)
	= \rho\, \phi'(\rho) - [ \phi'(\rho)\, \rho - \phi(\rho) ]
	= \phi(\rho)
$.
{\em (ii)} $\rho\in\R_{<0}$: As $a'$ has a nonnegative range, see \er{eq:exact-a-dot}, $\chi_\rho'(\beta) = \rho - a'(\beta) < -a'(\beta) \le 0$ for all $\beta\in\R_{\ge\phi'(0)}$. Hence, the supremum in \er{eq:Gamma-p} is achieved at $\beta^* = \phi'(0)$, and
$
	\chi_\rho(\beta^*) = \phi'(0)\, \rho - [ 0 - \phi(0)] = \phi(0) + \phi'(0)\, \rho
$.
{\em (iii)} $\rho\in\R_{\ge b^2}$: Observe that $\chi_\rho(\beta) = \gamma_\rho\circ a(\beta)$ for all $\beta\in\R_{\ge\phi'(0)}$, in which $\gamma_\rho$ is defined in \er{eq:gamma}. As $a:\R_{\ge\phi'(0)}\rightarrow\R_{\ge -\phi(0)}$ is strictly increasing and has an unbounded range, Lemma \ref{lem:gamma-properties} 
implies that $\lim_{\beta\rightarrow\infty} \chi_\rho(\beta) = \lim_{\beta\rightarrow\infty} \gamma_\rho\circ a(\beta) = \infty$. Hence, the supremum in \er{eq:Gamma-p} is achieved at $\beta^* = \infty$, and $\chi_\rho(\beta^*) = \infty$.

Combining cases {\em (i)} -- {\em (iii)}, $\Gamma_+$ of \er{eq:Gamma-p} may be written explicitly as
\begin{align}
	\Gamma_+(\rho)
	& = \left\{ \ba{rl}
		\phi(0) + \phi'(0)\, \rho, 
		& \rho\in\R_{<0},
		\\
		\phi(\rho),
		& \rho\in[0,b^2),
		\\
		+\infty,
		& \rho\in\R_{\ge b^2},
	\ea \right.
	\qquad
	\beta = \beta^* = \hat\beta_+^*(\rho)
	\doteq \left\{ \ba{rl}
		\phi'(0),
		& \rho\in\R_{<0},
		\\
		\phi'(\rho),
		& \rho\in[0,b^2),
		\\
		+\infty,
		& \rho\in\R_{\ge b^2}.
	\ea \right.
	\label{eq:Gamma-p-and-beta-p-explicit}
\end{align}
with the supremum achieved at the $\beta^* = \hat\beta_+^*(\rho)$ specified.
\if{false}

\begin{align}
	& \beta = \beta^* = \hat\beta_+^*(\rho)
	\doteq \left\{ \ba{rl}
		\phi'(0),
		& \rho\in\R_{<0},
		\\
		\phi'(\rho),
		& \rho\in[0,b^2),
		\\
		+\infty,
		& \rho\in\R_{\ge b^2}.
	\ea \right.
	\label{eq:beta-p-explicit}
\end{align}

\fi
Combining \er{eq:Phi-from-Gamma-pm} with \er{eq:Gamma-m-explicit}, \er{eq:Gamma-p-and-beta-p-explicit}, and the fact that $\phi(\rho) \ge \phi(0) + \phi'(0)\, \rho$ for all $\rho\in[0,b^2)$ by \er{eq:phi-properties-1-2} {\em (i)}, yields \er{eq:barrier}. Similarly, combining \er{eq:beta-m-explicit} and \er{eq:Gamma-p-and-beta-p-explicit} yields the left-hand equation in \er{eq:exact-beta-and-rho-star}.
\end{proof}



\begin{proof}
{\em [Lemma \ref{lem:approx-convex-duality}]}
%
{\em (i)} Fix $M\in\R_{\ge -\phi(0)}$. By the monotonicity of $a^{-1}$, see Lemma \ref{lem:a-properties} and \er{eq:exact-a-inv}, 
note that $a^{-1}(M) \ge a^{-1}(-\phi(0)) = \phi'(0)$. Hence, recalling \er{eq:Phi-M-from-A-and-rho-M}, and subsequently \er{eq:exact-Phi-and-A},
\begin{align}
	\Phi^M(\rho)
	& = \max\left\{ \sup_{\beta<\phi'(0)} \{ \beta\, \rho + \phi(0) \}, \, \sup_{\beta\in[\phi'(0),a^{-1}(M)]} \{ \beta\, \rho - \Ahat(\beta) \} \right\}
	= \max\{ \Gamma_-(\rho), \, \Gamma_+^M(\rho) \},
	\label{eq:Phi-M-max}
\end{align}
where $\Gamma_-$ is as per \er{eq:Gamma-m-explicit}, and
$	\Gamma_+^M(\rho)
	\doteq \sup_{\beta\in[\phi'(0), a^{-1}(M)]} \{ \beta\, \rho - a(\beta) \}$.
Modifying the argument preceding \er{eq:Gamma-p-and-beta-p-explicit} in the proof of Lemma \ref{lem:exact-convex-duality}, 
\begin{align}
	\Gamma_+^M(\rho)
	& = \left\{ \ba{rl}
		\phi(0) + \phi'(0)\, \rho,
		& \rho\in\R_{<0},
		\\
		\phi(\rho),
		& \rho\in[0, \hat\rho(M)],
		\\
		a^{-1}(M) \, \rho - M,
		& \rho\in\R_{>\hat\rho(M)},
	\ea \right.
	\label{eq:Gamma-p-M-and-beta-p-star-M}
\end{align}
with the supremum achieved at the
$\beta = \hat\beta_+^{M*}(\rho)$ specified.
\if{false}

\begin{align}
	\beta = \hat\beta_+^{M*}(\rho)
	& \doteq \left\{ \ba{rl}
		\phi'(0),
		& \rho\in\R_{<0},
		\\
		\phi'(\rho),
		& \rho\in[0,\hat\rho(M)],
		\\
		a^{-1}(M),
		& \rho\in\R_{>\hat\rho(M)}.
	\ea \right.
\end{align}

\fi
The pointwise maximum \er{eq:Phi-M-max} may be evaluated via \er{eq:Gamma-m-explicit}, \er{eq:Gamma-p-M-and-beta-p-star-M}, and the inequalities \er{eq:inequality-1}, \er{eq:inequality-2} of Lemma \ref{lem:inequalities}. Indeed, inspection of \er{eq:Phi-M-max}, \er{eq:Gamma-m-explicit}, \er{eq:Gamma-p-M-and-beta-p-star-M}, \er{eq:inequality-1}, \er{eq:inequality-2} immediately yields \er{eq:Phi-M-explicit}. The optimizer \er{eq:beta-M-star-explicit} that achieves the supremum in \er{eq:Phi-M-from-A-and-rho-M} follows by matching the corresponding cases in \er{eq:beta-m-explicit}, \er{eq:Gamma-p-M-and-beta-p-star-M}.

{\em (ii)} In view of $\Phi^M$, $\hat\rho(M)$ of \er{eq:Phi-M-explicit}, \er{eq:Phi-M-from-A-and-rho-M}, define
\begin{align}
	L & \doteq 
	\Phi^M\circ\hat\rho(M)
	= \phi\circ (\phi')^{-1}\circ a^{-1}(M),
	\quad
	U \doteq a^{-1}(M)\, \hat\rho(M) - M
	= a^{-1}(M)\, (\phi')^{-1}\circ a^{-1}(M) - M,
	\nn
\end{align}
for $M\in\R_{\ge -\phi(0)}$. With $\bar\beta\doteq a^{-1}(M)$, note that
$L = \phi\circ (\phi')^{-1}(\bar\beta)$ and $U = \bar\beta \, (\phi')^{-1}(\bar\beta) - M$,
so that
$U - L = [\bar\beta \, (\phi')^{-1}(\bar\beta) - \phi\circ (\phi')^{-1}(\bar\beta)] - M = a(\bar\beta) - M = a\circ a^{-1}(M) - M = 0$,
via \er{eq:exact-a}.
That is, $\Phi^M$ is continuous at $\hat\rho(M)$, and $\Phi^M\in C(\R_{\ge 0};\R)$. 
By inspection of \er{eq:Phi-M-explicit},
\begin{align}
	(\Phi^M)'(\rho)
	& = \left\{ \ba{rl}
			\phi'(\rho),		&		\rho\in(0,\hat\rho(M)),
			\\
			a^{-1}(M),		&		\rho\in(\hat\rho(M),\infty),
	\ea \right.
	\nn
\end{align}
and $\lim_{\rho\uparrow\hat\rho(M)} (\Phi^M)'(\rho) = \phi'\circ\hat\rho(M) = a^{-1}(M) = \lim_{\rho\downarrow\hat\rho(M)} (\Phi^M)'(\rho)$ via \er{eq:Phi-M-from-A-and-rho-M}. Hence, $\Phi^M\in C(\R_{\ge 0};\R)\cap C^1(\R_{>0};\R)$.
As $(\Phi^M)'$ is non-decreasing on $\R_{>0}$, and infinite elsewhere, $\Phi^M:\R\rightarrow\ol{\R}^+$ is (lower) closed convex on $\R$, see for example \cite[(3.8), pp.15,17]{R:74}.

{\em (iii)} Follows by inspection of \er{eq:barrier}, \er{eq:exact-Phi-and-A}, \er{eq:Phi-M-from-A-and-rho-M}, via Lemma \ref{lem:exact-convex-duality}. ${ }^{ }$\hfill{$\blacksquare$}

{\em (iv)} The following claim is first demonstrated.

{\em Claim:} Given $M\in\R_{\ge -\phi(0)}$, there exists $\Ahat^M:\R\rightarrow\R_{\ge -\phi(0)}^+$ such that
\begin{align}
	\Phi^M(\rho)
	& = \sup_{\beta\in\R} \{ \beta\, \rho - \Ahat^M(\beta) \},
	\label{eq:Phi-M-from-A-M}
	\\
	\Ahat^M(\beta)
	& = \sup_{\rho\in\R} \{ \beta\, \rho - \Phi^M(\rho) \}
	= \left\{ \ba{rl}
		-\phi(0),	&		\beta\in\R_{<\phi'(0)},
		\\
		a(\beta),	&		\beta\in[\phi'(0), a^{-1}(M)],
		\\
		+\infty,	&		\beta\in\R_{>a^{-1}(M)},
	\ea \right.
	\label{eq:A-M-from-Phi-M}
\end{align}
for all $\rho,\beta\in\R$, with $\phi$, $a$ as per \er{eq:phi-properties-1-2}, \er{eq:exact-a}. 

{\em Proof of Claim:}
Convexity assertion {\em (ii)} and \cite[Theorem 5, p.16]{R:74} imply that their exists a one-to-one pairing between $\Phi^M$ and its Fenchel transform $\Ahat^M:\R\rightarrow\ol{\R}^+$ as per \er{eq:Phi-M-from-A-M} and the left-hand equation in \er{eq:A-M-from-Phi-M}. It remains to show that the right-hand equation in \er{eq:A-M-from-Phi-M} holds.

By \er{eq:Phi-M-explicit}, the supremum in the left-hand equation in \er{eq:A-M-from-Phi-M} is never achieved at $\rho\in\R_{<0}$. Hence,
\begin{gather}
	\Ahat^M(\beta)
	= \max\{ \Lambda_-^M(\beta),\, \Lambda_+^M(\beta) \},
	\label{eq:A-M-max}
	\\
	\Lambda_-^M(\beta)
	\doteq \sup_{\rho\in[0,\hat\rho(M)]} \pi_\beta(\rho),
	\quad
	\Lambda_+^M(\beta)
	\doteq \sup_{\rho\in\R_{> \hat\rho(M)}} \{ (\beta - a^{-1}(M) ) \, \rho + M \},
	\label{eq:Lambda-M-p}
\end{gather}
for all $\beta\in\R$, with $\pi_\beta$ as per \er{eq:A-pi}. Replacing $b^2$ with $\hat\rho(M)$ in the argument following \er{eq:A-pi} yields
\begin{align}
	\Lambda_-^M(\beta)
	& = \left\{ \ba{rl}
		-\phi(0),
		& \beta\in\R_{<\phi'(0)},
		\\
		a(\beta),
		& \beta\in[\phi'(0), a^{-1}(M)],
		\\
		\lambda_-^M(\beta),
		& \beta\in\R_{>a^{-1}(M)},
	\ea\right.
	\qquad\qquad
	\lambda_-^M(\beta)
	\doteq \beta\, \hat\rho(M) - \phi\circ\hat\rho(M).
	\label{eq:Lambda-M-m-explicit}
\end{align}
By inspection of \er{eq:Lambda-M-p}, and recalling \er{eq:lambda-M-p},
\begin{align}
	\Lambda_+^M(\beta)
	& = \left\{ \ba{rl}
		\lambda_+^M(\beta) - \phi(0)\,,
		& \beta\in\R_{\le a^{-1}(M)}\,,
		\\
		\infty\,,
		& \beta\in\R_{>a^{-1}(M)}\,.
	\ea \right.
	\label{eq:Lambda-M-p-explicit}
\end{align}
Hence, the pointwise maximum in \er{eq:A-M-max} may be evaluated by \er{eq:Lambda-M-m-explicit}, \er{eq:Lambda-M-p-explicit} and inequalities \er{eq:inequality-3}, \er{eq:inequality-4} from Lemma \ref{lem:inequalities} in Appendix \ref{app:properties}, which yields the right-hand equation in \er{eq:A-M-from-Phi-M}. 
${ }^{ }$\hfill{$\blacksquare$}

Returning to the proof of {\em (iv)}, by Lemma \ref{lem:a-properties}, there exists an $M_1\in\R_{\ge -\phi(0)}$ such that $a^{-1}(M)\in\R_{>0}$ for all $M\in\R_{\ge M_1}$. Meanwhile, applying the above claim, in particular \er{eq:A-M-from-Phi-M}, $\Ahat^M(0) = \sup_{\rho\in\R} \{ -\Phi^M(\rho) \}$ for any $M\in\R_{\ge -\phi(0)}$. Hence, recalling \er{eq:A-M-from-Phi-M},
\begin{align}
	& \hat c\doteq \inf_{M\ge M_1}
	\inf_{\rho\in\R} \Phi^M(\rho)
	= - \sup_{M\ge M_1} \Ahat^M(0)
	= \left\{ \ba{cl}
			\phi(0),				&	0\in\R_{<\phi'(0)},
			\\
			\phi\circ(\phi')^{-1}(0), 	&	0\in\R_{\ge \phi'(0)},
		\ea \right.
	\nn
\end{align}
so that the assertion is proved for any $c\in\R$ satisfying $c<\hat c\in\R$, as required.
\end{proof}


\bibliographystyle{IEEEtran}
\bibliography{barrier}


\end{document}